\documentclass{amsproc}

\usepackage{amsmath,mathrsfs,amsthm,amsfonts}
\usepackage[margin=1in]{geometry}
\usepackage{enumitem}
\usepackage[figuresright]{rotating} 

\usepackage{tikz,siunitx}

\usetikzlibrary{matrix,positioning,calc}

\theoremstyle{plain}
\newtheorem{Th}{Theorem}[section]

\newtheorem{Lem}[Th]{Lemma}
\theoremstyle{definition}
\newtheorem{Rem}[Th]{Remark}
\newtheorem{Def}[Th]{Definition}

\newtheorem{Prop}[Th]{Proposition}

\newcommand{\h}{\mathbf k}

\def\ot{\otimes}
\def\hot{\hat{\otimes }}
\allowdisplaybreaks[4]

\pgfdeclarelayer{edgelayer}
\pgfdeclarelayer{nodelayer}
\pgfsetlayers{edgelayer,nodelayer,main}

\tikzstyle{none}=[inner sep=0pt]

\begin{document}
\title{Hom-Tensor Categories and the Hom-Yang-Baxter Equation}

\author{Florin Panaite}
\address{Institute of Mathematics of the Romanian Academy,
PO-Box 1-764, RO-014700 Bucharest, Romania}
\email{Florin.Panaite@imar.ro}

\author{Paul Schrader}
\address{Department of Mathematics and Statistics, Bowling Green State University, Bowling Green, OH 43403}
\email{stpaul@bgsu.edu}

\author{Mihai D. Staic}
\address{Department of Mathematics and Statistics, Bowling Green State University, Bowling Green, OH 43403 }
\address{Institute of Mathematics of the Romanian Academy, PO.BOX 1-764, RO-70700 Bu\-cha\-rest, Romania.}
\email{mstaic@bgsu.edu}




\begin{abstract} 
 We  introduce a new type of categorical object called a \emph{hom-tensor category} and show that it provides the appropriate setting for modules over an arbitrary hom-bialgebra. Next we introduce  the notion of  \emph{hom-braided category} and show that this is the right setting for modules over quasitriangular hom-bialgebras. We also show how the hom-Yang-Baxter equation fits into this framework and how the category of Yetter-Drinfeld modules over  a hom-bialgebra with bijective structure map can be organized as a hom-braided category. Finally we prove that, under certain conditions, one can obtain a tensor category (respectively a braided tensor category) from a hom-tensor category (respectively a hom-braided category). 
\end{abstract}

\maketitle

\section{Introduction}

Tensor categories were introduced by B\'{e}nabou in \cite{be}. A basic example is the category of vector spaces over a field  $k$. More interesting examples can be obtained from bialgebras.  If $A$ is an algebra and $\Delta:A\to A\otimes A$ is a morphism of algebras, then the category of $A$-modules is a tensor category (with the tensor product induced by $\Delta$ and trivial associativity constraint) if and only if $A$ is a bialgebra. 
 
The \emph{Yang-Baxter equation} was introduced by Yang and Baxter (see \cite{RB:YB1}, \cite{CY:YB1}). It has applications to knot invariants and it was intensively studied over the last thirty years. 

Braided categories were introduced by Joyal and Street in \cite{js}. The main example is the braid category; it satisfies a universal property for braided categories (see \cite{Kas:QG}). Other examples  are obtained from quasitriangular Hopf algebras. Braided categories can be used to construct representations for the braid group and invariants for tangles, knots and 3-manifolds (see \cite{tu}). The braiding of a braided category satisfies a dodecagonal equation (see \cite{Kas:QG}) that may be 
regarded as a categorical analogue of the Yang-Baxter equation.

The genesis of hom-structures may be found in the physics literature from the years 1990, concerning quantum deformations 
of algebras of vector fields, especially Witt and Virasoro algebras (e.g., see \cite{Aiza:qDeform}, \cite{ChaiKul:qDef}, \cite{CurtZac:DefMap}, \cite{Daska:GenDefVir} and \cite{Kas:CycHom}). These classes of examples led to the development first of hom-Lie algebras (\cite{HartSilv:DefLA}, \cite{LarSilv:QHLA}), which are analogues of Lie algebras where the 
Jacobi identity  is twisted by a linear map. 
This was followed by the development of hom-analogues of associative algebras, coalgebras, bialgebras, Hopf algebras, etc. 
(e.g., see \cite{BenaMak:HLA}, \cite{CG:biH}, \cite{chenwangzhang}, \cite{liguo}, \cite{HassSha:CycHA}, 
\cite{LiuShen:Rad}, \cite{MP:HA},  \cite{MS:HA}, \cite{MS:HA3}, \cite{MS:HA4},  \cite{SH:HA}, \cite{YA:HA3}, \cite{YA:HA4}, \cite{YA:HA}). The reader can find a concise history on hom-structures in the introduction of 
\cite{MP:HA}.

One natural question to ask is what type of categorical framework these hom-structures fit into. In 
the original concept of hom-bialgebra (see \cite{MS:HA3}, \cite{MS:HA4}), two distinct linear maps twist the associative and co-associative structures of a bialgebra. When the two twisting maps are inverses to each other it was proved in  \cite{CG:biH} that the category of modules is a tensor category. The question that may be asked is what kind of categorical framework does a hom-bialgebra where two arbitrary linear maps twist the associative and co-associative structure fit into? Moreover, is there an analogue to the classical relationship between quasitriangular bialgebras and braided tensor categories for quasitriangular hom-bialgebras? It is these questions that motivated this paper and the concepts it contains.

There are two main objectives to this paper. The first one is to introduce a hom-analogue to a tensor category, called a \emph{hom-tensor category}. In a hom-tensor category $\mathcal{C}$ the usual associator is replaced by a natural isomorphism $a_{U,V,W}:(U\otimes V)\otimes F(W)\to F(U)\otimes (V\otimes W)$ that satisfies a generalized pentagonal equation (here 
$F:\mathcal{C}\rightarrow \mathcal{C}$ is a functor; when $F$ is the identity functor we recover the definition of a tensor category without unit). We show that the category of modules over a hom-bialgebra (as it is posed in \cite{MS:HA3} and \cite{MS:HA4})  fits in the categorical framework of hom-tensor categories. 

The second objective is to introduce a hom-analogue to a braided tensor category, called a \emph{hom-braided category}. In a hom-braided category $\mathcal{C}$ we have a natural morphism $c_{U,V}:U\otimes V\to G(V)\otimes G(U)$ that satisfies a generalization of the hexagonal axioms (where $G:\mathcal{C}\rightarrow \mathcal{C}$ is another functor). We show that this new categorical framework provides the right setting for modules over quasitriangular hom-bialgebras. We also show how the hom-Yang-Baxter equation (introduced by D. Yau in \cite{YA:HYB}) fits in the context of hom-braided categories, and we prove that the category of Yetter-Drinfeld modules (introduced in \cite{MP:HA}) over a 
hom-bialgebra with bijective structure map becomes a hom-braided category. 

As applications to our theory, we give new proofs for Yau's result from \cite{YA:HA2} saying roughly that a quasitriangular hom-bialgebra $H$ provides a solution for the hom-Yang-Baxter equation on any $H$-module and for 
the result in \cite{MP:HA} saying that $_H^H{\mathcal YD}$, the category of 
Yetter-Drinfeld modules 
$(M, \alpha _M)$ with $\alpha _M$ bijective over a hom-bialgebra $H$ with bijective structure map, is a quasi-braided 
category.

The structure of this paper is as follows. Section \ref{sec2} begins with recalling some definitions and concepts of hom-structures necessary in presenting the upcoming results. We begin Section \ref{sec3} by defining a hom-tensor category,  and then we show how a hom-tensor category is the appropriate categorical framework for hom-bialgebras with arbitrary twisting maps. Section \ref{sec4} introduces the notions of algebras in a hom-tensor category, left $H$-module hom-algebras over a hom-bialgebra $H$ and a categorical analogue to a Yau twist. In Section \ref{sec5} we define hom-braided categories and prove that they provide the right categorical framework for quasitriangular hom-bialgebras. In Section \ref{sec6} we show how to regard the hom-Yang-Baxter equation  in the categorical framework of hom-braided categories. In Section \ref{sec7} the category of Yetter-Drinfeld modules over a hom-bialgebra as seen in \cite{MP:HA} is organized  under the framework of a hom-braided category. Finally, in Section \ref{sec8} we show that under certain conditions one can obtain a tensor category (respectively a braided tensor category) from a hom-tensor category (respectively a hom-braided category). 


\section{Preliminaries} \label{sec2}

We work over a base field $\h$. An unlabeled 
tensor product means either a functor $\otimes :\mathcal{C}\times \mathcal{C}\rightarrow \mathcal{C}$ on a 
category $\mathcal{C}$ or the tensor product over $\h$. 
For a comultiplication 
$\Delta :C\rightarrow C\ot C$ on a $\h$-vector space $C$ we use a 
Sweedler-type notation $\Delta (c)=\sum c_{(1)}\ot c_{(2)}$, for $c\in C$. Unless 
otherwise specified, the (co)algebras ((co)associative or not) that will appear 
in what follows are {\em not} supposed to be (co)unital, and a multiplication 
$\mu :V\ot V\rightarrow V$ on a $\h$-vector space $V$ is denoted by juxtaposition: 
$\mu (v\ot v')=vv'$. 

We will use the following terminology for categories. A {\em pre-tensor category} is a category satisfying all the axioms of a tensor category in \cite{Kas:QG} except for the fact that we do not require the existence of a 
unit object. If $({\mathcal C}, \otimes, a)$ is a pre-tensor 
category, a {\em quasi-braiding} $c$ in ${\mathcal C}$ is a family of natural morphisms 
$c_{V, W}:V\ot W\rightarrow W\ot V$ in ${\mathcal C}$ satisfying all the axioms of 
a braiding in \cite{Kas:QG} except for the fact that we do not require $c_{V, W}$ to be isomorphisms; 
in this case, $({\mathcal C}, \otimes, a, c)$ is called a quasi-braided pre-tensor category. 

We recall now some definitions, notation and results taken from \cite{MS:HA}, 
\cite{MS:HA3}, \cite{MS:HA4}, \cite{YA:HA3} and \cite{YA:HA}. 

\begin{Def}
A \emph{hom-associative $\h$-algebra} is a triple $\left(A,m_{A},\alpha_{A}\right)$, where $A$ is a $\h$-vector space,\\
 $m_{A}:A\otimes A\rightarrow A$ is a $\h$-linear map denoted by $m_{A}\left(a\otimes b\right)=ab$, for all $a,b\in A$, and $\alpha_{A}:A\rightarrow A$ is a $\h$-linear map satisfying the following conditions, for all $a,b,c\in A$:
\begin{eqnarray}
&&  \alpha_{A}\left(ab\right)=\alpha_{A}\left(a\right)\alpha_{A}\left(b\right),
 \label{eq1}\\
 && \alpha_{A}\left(a\right)\left(bc\right)=\left(ab\right)\alpha_{A}\left(c\right).
	\label{eq2}
\end{eqnarray}

Let $\left(A,m_{A},\alpha_{A}\right)$ and $\left(B,m_{B},\alpha_{B}\right)$ be two hom-associative $\h$-algebras. A \emph{morphism} of hom-associative algebras $f:\left(A,m_{A},\alpha_{A}\right)\rightarrow\left(B,m_{B},\alpha_{B}\right)$ is a $\h$-linear map $f:A\rightarrow B$ such that $\alpha_{B}\circ f=f\circ\alpha_{A}$ and $f\circ m_{A}=m_{B}\circ\left(f\otimes f\right)$.
\end{Def}
\begin{Rem}
If $\left(A,m_{A},\alpha_{A}\right)$, $\left(B,m_{B},\alpha_{B}\right)$ are hom-associative $\h$-algebras, then 
$(A\otimes B, m_{A\otimes B}, \alpha _{A\otimes B})$ is also a hom-associative $\h$-algebra,  
where $m_{A\otimes B}((a\otimes b)\otimes (a'\otimes b'))=aa'\otimes bb'$ and 
$\alpha _{A\otimes B}=\alpha _A\otimes \alpha _B$. 
\end{Rem}

\begin{Def}
A \emph{hom-coassociative $\h$-coalgebra} is a triple $\left(C,\Delta_{C},\psi_{C}\right)$, where $C$ is a $\h$-vector space,\\
 $\Delta_{C}:C\rightarrow C\otimes C$ and $\psi_{C}:C\rightarrow C$ are $\h$-linear maps satisfying the following conditions:
\begin{eqnarray}
&&  \left(\psi_{C}\otimes\psi_{C}\right)\circ\Delta_{C}=\Delta_{C}\circ\psi_{C},
 \label{eq3}\\
&&  \left(\Delta_{C}\otimes\psi_{C}\right)\circ\Delta_{C}=\left(\psi_{C}\otimes\Delta_{C}\right)\circ\Delta_{C}.
 \label{eq4}
\end{eqnarray}

   Let $\left(C,\Delta_{C},\psi_{C}\right)$, $\left(D,\Delta_{D},\psi_{D}\right)$ be two hom-coassociative $\h$-coalgebras. 
A \emph{morphism} of hom-coassociative $\h$-coalgebras $g:\left(C,\Delta_{C},\psi_{C}\right)\rightarrow 
\left(D,\Delta_{D},\psi_{D}\right)$ is a $\h$-linear map $g:C\rightarrow D$ such that $\psi _D\circ g=g\circ \psi _C$ 
and $(g\otimes g)\circ \Delta _C=\Delta _D\circ g$. 
\label{def4}
\end{Def}
\begin{Def} \label{defhombialgebra}
A \emph{hom-bialgebra} is a 5-tuple $\left(B,m_{B},\Delta_{B},\alpha_{B},\psi_B\right)$, where 
$\left(B,m_{B},\alpha_{B}\right)$ is a hom-associative $\h$-algebra, $\left(B,\Delta_{B},\psi_{B}\right)$ is a hom-coassociative $\h$-coalgebra, $\Delta_{B}$ is a morphism of hom-associative $\h$-algebras, $\alpha_{B}$ is a morphisms of 
hom-coassociative $\h$-coalgebras and $\psi_{B}$ is a morphism of hom-associative $\h$-algebras (in particular we have $\alpha_B\circ \psi_B=\psi_B\circ \alpha_B$).
\label{def5} 
\end{Def}
\begin{Rem}\label{rem2}
The following statement is equivalent to Definition \ref{def5}. A hom-bialgebra is a hom-associative $\h$-algebra $\left(B,m_{B},\alpha_{B}\right)$ together with two $\h$-linear maps $\Delta_{B}:B\rightarrow B\otimes B$ and $\psi_B:B\to B$ such that 
$\alpha_B\circ \psi_B=\psi_B\circ \alpha_B$ and 
the following conditions are satisfied for all $b,b^{\prime}\in B$:
\begin{eqnarray}
&& \sum b_{\left(1\right)_{\left(1\right)}}\otimes b_{\left(1\right)_{\left(2\right)}}\otimes\psi_{B}\left(b_{\left(2\right)}\right)=\sum\psi_{B}\left(b_{\left(1\right)}\right)\otimes b_{\left(2\right)_{\left(1\right)}}\otimes b_{\left(2\right)_{\left(2\right)}},
\label{eq5}\\
&& \sum\left(bb^{\prime}\right)_{\left(1\right)}\otimes\left(bb^{\prime}\right)_{\left(2\right)}=\sum b_{\left(1\right)}b^{\prime}_{\left(1\right)}\otimes b_{\left(2\right)}b^{\prime}_{\left(2\right)},
\label{eq6}\\
&&\sum \alpha_{B}\left(b\right)_{\left(1\right)}\otimes \alpha_{B}\left(b\right)_{\left(2\right)}=
\sum\alpha_{B}\left(b_{\left(1\right)}\right)\otimes\alpha_{B}\left(b_{\left(2\right)}\right),
\label{eq7}\\
&&\sum \psi_{B}\left(b\right)_{\left(1\right)}\otimes \psi_{B}\left(b\right)_{\left(2\right)}=
\sum\psi_{B}\left(b_{\left(1\right)}\right)\otimes\psi_{B}\left(b_{\left(2\right)}\right),
\label{eq7111}\\
&&\psi_B(bb')=\psi_B(b)\psi_B(b').
\label{7112}
\end{eqnarray}
\end{Rem}
\begin{Rem}
In the literature, most of the results  about hom-bialgebras use the extra assumption that $\psi_B=\alpha _B$ or 
$\psi _B=\alpha _B^{-1}$ (see \cite{CG:biH}, \cite{MP:HA}, \cite{YA:HA2}). We treat the general situation, to cover both cases of interest. 
\end{Rem}

We recall now the so-called ''twisting principle'' or ''Yau twisting''. 
\begin{Prop} Let $(A, \mu )$ be an associative $\h$-algebra and $\alpha :A\rightarrow A$ an algebra endomorphism. Define 
a new multiplication $\mu _{\alpha }:A\otimes A\rightarrow A$, $\mu _{\alpha }:=\alpha \circ \mu =
\mu \circ (\alpha \otimes \alpha )$. 
Then 
$(A, \mu _{\alpha }, \alpha )$ is a hom-associative $\h$-algebra, denoted by $A_{\alpha }$ and called the \emph{Yau twist} 
of $A$.  
\end{Prop}


\begin{Def}\label{def6}
Let $M$ be a $\h$-vector space, $\left(A,m_{A},\alpha_{A}\right)$ be a hom-associative $\h$-algebra and $\alpha_{M}:M\rightarrow M$ be a $\h$-linear map. A \emph{left A-module structure} on $\left(M,\alpha_{M}\right)$ 
consists of a $\h$-linear map $\mu_{M}:A\otimes M\rightarrow M$, with notation $\mu _M(a\otimes m)=a\cdot m$, 
such that the following conditions are satisfied for all $a,b\in A$ and $m\in M$:
\begin{eqnarray}
&& \alpha_{M}\left(a\cdot m\right)=\alpha_{A}\left(a\right)\cdot\alpha_{M}\left(m\right),
 \label{eq8}\\
&& \alpha_{A}\left(a\right)\cdot\left(b\cdot m\right)=\left(ab\right)\cdot\alpha_{M}\left(m\right).
 \label{eq9}
\end{eqnarray}

Let $\left(M,\alpha_{M}\right)$ and $\left(N,\alpha_{N}\right)$ be two left $A$-modules. A \emph{morphism} of left A-modules 
is a 
$\h$-linear map $f:M\rightarrow N$ satisfying the conditions $\alpha_{N}\circ f=f\circ\alpha_{M}$ and $f\left(a\cdot m\right)=
a\cdot f\left(m\right)$ for all $a\in A$, $m\in M$.
\end{Def}
\begin{Def}
Let $(C, \Delta _C , \psi _C)$ be a hom-coassociative $\h$-coalgebra, $M$ a $\h$-vector space and $\psi _M:M
\rightarrow M$ a $\h$-linear map. A {\em left $C$-comodule structure} on $(M, \psi _M)$ consists of a $\h$-linear map 
$\lambda _M:M\rightarrow C\otimes M$ (usually denoted by $\lambda _M(m)=\sum m_{(-1)}\otimes m_{(0)}$, for all $m\in M$), 
satisfying the following conditions:
\begin{eqnarray}
&&(\psi _C\otimes \psi _M)\circ \lambda _M=\lambda _M\circ \psi _M,  
\label{comodul1}\\
&&(\Delta _C\otimes \psi _M)\circ \lambda _M=(\psi _C\otimes \lambda _M)\circ \lambda _M. 
\label{comodul2}
\end{eqnarray} 

If $(M, \psi _M)$ and $(N, \psi _N)$ are left $C$-comodules, with structures 
$\lambda _M:M\rightarrow C\otimes M$ and $\lambda _N:N\rightarrow C\otimes N$, 
a {\em morphism} of left $C$-comodules $g:M\rightarrow N$ is a $\h$-linear map satisfying the conditions 
$\psi _N\circ g=g\circ \psi _M$  and $(id_C\otimes g)\circ \lambda _M=\lambda _N\circ g$. 
\end{Def}

We define general quasitriangular hom-bialgebras (see \cite{YA:HA4}, \cite{YA:HA2} for the case $\alpha=\psi$). 
\begin{Def}\label{def11}
Let $\left(H,m,\Delta,\alpha,\psi \right)$ be a hom-bialgebra and let $R\in H\otimes H$ be given as 
$R=\sum_{i} s_{i}\otimes t_{i}$. We call $\left(H,m,\Delta,\alpha,\psi,R\right)$ a 
\emph{quasitriangular hom-bialgebra} if the following conditions are satisfied:
\begin{eqnarray}
&& R\Delta\left(h\right)=\Delta^{\text{cop}}\left(h\right)R,\text{ for all } h\in H,
\label{eq29}\\
&& \left(\Delta\otimes \alpha \right)\left(R\right)=\sum_{i,j}\psi \left(s_{i}\right)\otimes\psi \left(s_{j}\right)\otimes t_{i}t_{j},
\label{eq30}\\
&& \left(\alpha \otimes\Delta\right)\left(R\right)=\sum_{i,j}s_{i}s_{j}\otimes\psi \left(t_{j}\right)\otimes\psi \left(t_{i}\right), 
\label{eq31}
\end{eqnarray}
where we denoted as usual $\Delta ^{cop}(h)=\sum h_{(2)}\otimes h_{(1)}$, for $h\in H$. 
\end{Def}

\begin{Rem}\label{rem9}
Let $H=\left(H,m,\Delta ,\alpha ,\psi ,R \right)$ be a quasitriangular hom-bialgebra and $h\in H$. We can reformulate conditions (\ref{eq29}), (\ref{eq30}) and (\ref{eq31}) respectively  in Definition \ref{def11} using Sweedler notation as follows:
 \begin{eqnarray}
&&  \sum_{i}s_{i}h_{\left(1\right)}\otimes t_{i}h_{\left(2\right)}=\sum_{i}h_{\left(2\right)}s_{i}\otimes h_{\left(1\right)}t_{i}, 
 \label{eq38}\\
 && \sum_{i}\left(s_{i}\right)_{\left(1\right)}\otimes\left(s_{i}\right)_{\left(2\right)}\otimes\alpha \left(t_{i}\right)=\sum_{i,j}\psi  \left(s_{i}\right)\otimes\psi \left(s_{j}\right)\otimes t_{i}t_{j}, 
 \label{eq39}\\
&&\sum_{i}\alpha \left(s_{i}\right)\otimes\left(t_{i}\right)_{\left(1\right)}\otimes\left(t_{i}\right)_{\left(2\right)}=
\sum_{i,j}s_{i}s_{j}\otimes\psi \left(t_{j}\right)\otimes\psi \left(t_{i}\right).
 \label{eq60}
 \end{eqnarray}
\end{Rem}
\begin{Rem} \label{remQT}
Notice that if $(\psi \otimes \psi )(R)=R$ then conditions (\ref{eq30}) and (\ref{eq31}) are equivalent to 
\begin{eqnarray}
&& \left(\Delta\otimes (\alpha\circ \psi )\right)\left(R\right)=\sum_{i,j}s_{i} \otimes s_{j}\otimes t_{i}t_{j},\\
&& \left((\alpha\circ \psi )\otimes\Delta\right)\left(R\right)=\sum_{i,j}s_{i}s_{j}\otimes t_{j}\otimes t_{i}.
\end{eqnarray}
\end{Rem}

We introduce now the following concept, to be used in subsequent sections.  
\begin{Def}
Let $A=\left(A,m_{A},\alpha_{A}\right)$ be a hom-associative $\h$-algebra. \\
(i) Suppose that $h\cdot m=0$  for any $A$-module $\left(M,\alpha_{M}\right)$ and for all $m\in M$ implies $h=0$. 
Then we say that $A$ is \emph{nondegenerate}. \\
(ii) Suppose that $h\cdot m=0$  for any $A$-module $\left(M,\alpha_{M}\right)$ and for all $m\in \alpha_M(M)$ implies $h=0$. 
Then we say that $A$ is \emph{strongly nondegenerate}.
\end{Def}

\begin{Lem} \label{lemmanondeg}
Let $A$ be a nondegenerate hom-associative $\h$-algebra and $x\in A\otimes A$ such that $x\cdot (u\otimes v)=0$ 
for all $u\in U$ and all $v\in V$ and for all left $A$-modules $(U, \alpha _U)$ and $(V, \alpha _V)$. Then $x=0$. 
A similar result is true for $y\in A^{\otimes 3}$. Similar results are true for strongly nondegenerate algebras. 
\end{Lem}
\begin{proof} Let $x=\sum_{p=1}^na_p\otimes b_p$, where $a_p$, $b_p\in A$ and $\{b_p\}_{1\leq p\leq n}$ are $\h$-linearly independent. Fix a left $A$-module $(U, \alpha _U)$ and fix a $\h$-basis $\{e_i\}_{i\in I}$ for $U$. Consider the set of $\h$-linear applications $e_i^*\in U^*$ determined by $e_i^*(e_j)=\delta_i^j$. For every left $A$-module $(V, \alpha _V)$ and 
for every $u\in U$, $v\in V$ and $i\in I$ we have: 
\begin{equation*}
0=(e_i^*\otimes \text{id}_V)(x\cdot (u\otimes v))=(e_i^*\otimes \text{id}_V)(\sum_{p=1}^n(a_p\otimes b_p)\cdot (u\otimes v))=\sum_{p=1}^ne_i^*(a_p\cdot u)b_p\cdot v.
\end{equation*}
Since $A$ is nondegenerate we get that for every $i\in I$ and every $u\in U$ we have that 
$\sum_{p=1}^ne^*_i(a_p\cdot u)b_p=0\in A$. But $\{b_p\}_{1\leq p\leq n}$ are linearly independent, so for every 
$1\leq p\leq n$ we have $e_i^*(a_p\cdot u)=0$ for all $i\in I$ and for all $u\in U$. Now since 
$\{e_i\}_{i\in I}$ is a $\h$-basis for $U$ we must have that $a_p\cdot u=0$ for all $1\leq p\leq n$ and for all $u\in U$. But $U$ can be any $A$-module and $A$ is nondegenerate which implies that $a_p=0$ for all $1\leq p\leq n$ and so $x=\sum_{p=1}^na_p\otimes b_p=0$.  
\end{proof}

\section{Hom-Tensor Categories}\label{sec3}

We introduce a new type of categories called \emph{hom-tensor categories}, which 
have a tensor functor $\otimes : \mathcal{C}\times \mathcal{C}\to \mathcal{C}$ with the usual associativity condition replaced by a more relaxed condition (see Definition \ref{def1}).  We show that  hom-bialgebras fit very nicely in this framework. Unlike the  tensor category introduced in \cite{CG:biH}, a hom-tensor category can be associated even to hom-bialgebras for which $\alpha_A$ is not necessary bijective.  
\begin{Def}
A \emph{hom-tensor category} is a 6-tuple $\left(\mathcal{C},\otimes,F,G,a,\Phi\right)$, where:
\begin{enumerate}[label=(\arabic*)]
\item $\mathcal{C}$ is a category.
\item $\otimes:\mathcal{C}\times\mathcal{C}\rightarrow\mathcal{C}$ is a covariant functor (called the \emph{hom-tensor product}). 
\item $F:\mathcal{C}\rightarrow\mathcal{C}$ is a covariant functor such that $F\left(U\otimes V\right)=F\left(U\right)\otimes F\left(V\right)$ for all objects
 $U,V\in\mathcal{C}$ and $F\left(f\otimes g\right)=F\left(f\right)\otimes F\left(g\right)$ for all morphisms $f,g\in\textbf{Hom}\left(\mathcal{C}\right)$. 
\item $a_{X,Y,Z}:\left(X\otimes Y\right)\otimes F\left(Z\right)\rightarrow F\left(X\right)\otimes\left(Y\otimes Z\right)$ is a natural isomorphism that satisfies the ``Pentagon'' axiom as seen in Figure \ref{fig1} for all objects $X,Y,Z,T\in\mathcal{C}$. We call $a$ the \emph{hom-associativity constraint} of the hom-tensor category.

\item $G:\mathcal{C}\rightarrow\mathcal{C}$ is a covariant functor such that $G\left(U\otimes V\right)=G\left(U\right)\otimes G\left(V\right)$ for all objects
 $U,V\in\mathcal{C}$ and $G\left(f\otimes g\right)=G\left(f\right)\otimes G\left(g\right)$ for all morphisms $f,g\in\textbf{Hom}\left(\mathcal{C}\right)$. 

\item There exists  a natural transformation $\Phi:\text{id}_{\mathcal{C}}\to G$. 

\item $FG=GF$. 

\item $F(\Phi _U)=\Phi _{F(U)}$, $G(\Phi _U)=\Phi _{G(U)}$, for every object $U\in \mathcal{C}$. 

\item $\Phi _{M\otimes N}=\Phi_M\otimes \Phi _N$, for all objects $M, N\in \mathcal{C}$. 

\end{enumerate}
\begin{center}
\begin{figure}[!ht]

\begin{center}
\begin{tikzpicture}[scale=0.9, every node/.style={scale=0.9}]
  \matrix (m) [matrix of math nodes,row sep=3em,column sep=4em,minimum width=2em] 
	 { & \stackrel{\stackrel{\mbox{$\left(F\left(X\right)\otimes F\left(Y\right)\right)\otimes F\left(Z\otimes T\right)$}}{\mid\mid}}{F\left(X\otimes Y\right)\otimes\left(F\left(Z\right)\otimes F\left(T\right)\right)} & \\
																		& & \\
\left(\left(X\otimes Y\right)\otimes F\left(Z\right)\right)\otimes F^{2}\left(T\right) & & F^{2}\left(X\right)\otimes\left(F\left(Y\right)\otimes \left(Z\otimes T\right)\right) \\
																	  & & \\
\left(F\left(X\right)\otimes \left(Y\otimes Z\right)\right)\otimes F^{2}\left(T\right) & & F^{2}\left(X\right)\otimes\left(\left(Y\otimes Z\right)\otimes F\left(T\right)\right) \\};
  \path[-stealth]
	  (m-1-2) edge node [right] {\tiny{$ ~\ a_{F\left(X\right),F\left(Y\right),Z\otimes T}$}} (m-3-3)
		(m-3-1) edge node [left] {\tiny{$a_{X\otimes Y,F\left(Z\right),F\left(T\right)} ~\ $}} (m-1-2)
		        edge node [right]  {\tiny{$a_{X,Y,Z}\otimes\text{id}_{F^{2}\left(T\right)} ~\ $}} (m-5-1)
		(m-5-1) edge node [above] {\tiny{$a_{F\left(X\right),Y\otimes Z,F\left(T\right)}$}} (m-5-3)
		(m-5-3) edge node [left] {\tiny{$\text{id}_{F^{2}\left(X\right)}\otimes a_{Y,Z,T}$}} (m-3-3);
\end{tikzpicture}
\end{center}
\caption{The ``Pentagon'' axiom for the hom-associativity constraint $a$}
\label{fig1}
\end{figure}
\end{center}

\label{def1} 
\end{Def}
\begin{Rem} Note that we do not require the existence of a unit in the category, and so a more appropriate name for 
the structure we defined would be hom-pre-tensor category. However, in order to simplify the terminology, 
we prefer to call it hom-tensor category. 
\end{Rem}
\begin{Rem} One may relax the above definition, by removing some of the axioms. For instance, one may remove 
the condition $\Phi _{M\otimes N}=\Phi_M\otimes \Phi _N$, which is used later in only one place, in the last section. 
The importance of the functor $G$ will become apparent later, when we talk about hom-braided categories. 
\end{Rem}

We present now a first class of examples of hom-tensor categories. 
\begin{Prop} \label{mainexample}
Let $(\mathcal{C}, \otimes , a)$ be a pre-tensor category. We define the category $\mathfrak{h}(\mathcal{C})$ as 
follows: objects are pairs $(M, \alpha _M)$, where $M$ is an object in $\mathcal{C}$ and $\alpha _M\in 
Hom_{\mathcal{C}}(M, M)$, morphisms $f:(M, \alpha _M)\rightarrow (N, \alpha _N)$ are morphisms 
$f:M\rightarrow N$ in $\mathcal{C}$ such that $\alpha _N\circ f=f\circ \alpha _M$. Then 
$(\mathfrak{h}(\mathcal{C}), \otimes, F, G, a, \Phi )$ is a hom-tensor category, where the tensor product 
$\otimes $ is defined by $(M, \alpha _M)\otimes (N, \alpha _N)=(M\otimes N, \alpha _M\otimes \alpha _N)$ on 
objects and by the tensor product in $\mathcal{C}$ on morphisms, the functors $F$ and $G$ are both identity, the 
natural transformation $\Phi _{(M, \alpha _M)}:(M, \alpha _M)\rightarrow (M, \alpha _M)$ is defined by 
$\Phi _{(M, \alpha _M)}:=\alpha _M$, and the natural isomorphism $a$ is defined by 
$a_{(M, \alpha _M), (N, \alpha _N), (P, \alpha _P)}:=a_{M, N, P}$.  
\end{Prop}
\begin{proof}
A straightforward verification. 
\end{proof}


The following proposition gives the relation between hom-tensor categories and hom-bialgebras.
\begin{Prop}\label{prop3.3}
 Let $H=\left(H,m_{H},\alpha_{H}\right)$ be a hom-associative $\h$-algebra and let $\Delta_{H}:H\rightarrow H\otimes H$, $\psi_H:H\to H$ be morphisms of hom-associative $\h$-algebras. 
Consider the two statements (A) and (B) below. Then we have that (A) implies (B) and if $H$ is nondegenerate (B) implies (A).
\begin{enumerate}[label=(\Alph*)]
\item $\left(H,m_{H},\Delta_{H},\alpha_{H},\psi_H\right)$ is a hom-bialgebra.
\item The category $\mathscr{H}=\left(H\textbf{-mod},\otimes,F,G,a,\Phi\right)$ is a hom-tensor category, where:
\begin{enumerate}[label=(\roman*)]
\item The objects of $\mathscr{H}$ are left $H$-modules.
\item The morphisms of $\mathscr{H}$ are left $H$-module morphisms.  
\item The hom-tensor product of $\left(U,\alpha_{U}\right)$ and $\left(V,\alpha_{V}\right)$ is given by $\left(U,\alpha_{U}\right)\otimes\left(V,\alpha_{V}\right)\stackrel{\text{def}}{:=}\left(U\otimes V, \alpha_{U\otimes V}\right)$,
 where $\alpha_{U\otimes V}{:=}\alpha_{U}\otimes\alpha_{V}$ and 
the left $H$-action $H\otimes\left(U\otimes V\right)\rightarrow U\otimes V$ is defined 
for all elements $u\in U$, $v\in V$ and $h\in H$ by 
$h\cdot\left(u\otimes v\right)\stackrel{\text{def}}{:=}\Delta\left(h\right)\cdot\left(u\otimes v\right)=\sum\left(h_{\left(1\right)}\cdot u\right)\otimes\left(h_{\left(2\right)}\cdot v\right)$.  
If  $f\in\textbf{Hom}_{H\textbf{-mod}}\left(U,W\right)$ and $g\in\textbf{Hom}_{H\textbf{-mod}}\left(V,X\right)$ then  
$f\otimes g:U\otimes V\rightarrow W\otimes X$ is defined by $\left(f\otimes g\right)\left(u\otimes v\right)\stackrel{\text{def}}{:=}f\left(u\right)\otimes g\left(v\right)$, for all $u\in U$, $v\in V$. 
\item $F:H\textbf{-mod}\rightarrow H\textbf{-mod}$ is the covariant functor defined by
$F\left(\left(U,\alpha_{U}\right)\right)=\left(U^{\psi},\alpha_{U^{\psi}}\right)$, 
 where $U^{\psi}=U$ as a $\h$-vector space (we will denote an element of $U^{\psi}=U$ as $\bar{u}$), and the $H$-module structure $H\otimes U^{\psi}\rightarrow U^{\psi}$ is given by 
$h\cdot_{\psi}\bar{u}=\overline{\psi_{H}\left(h\right)\cdot u}$, 
for all $h\in H$, $u\in U$. Furthermore, the $\h$-linear map $\alpha_{U^{\psi}}:U^{\psi}\rightarrow U^{\psi}$ is defined by $\alpha_{U^{\psi}}\left(\bar{u}\right)=\overline{\alpha_{U}\left(u\right)}$ for all $\bar{u}\in U^{\psi}$. For all morphisms $f:U\to V$ we have   $F\left(f\right)\left(\bar{u}\right)=\overline{f\left(u\right)}$ for all $\bar{u}\in U^{\psi}$.
\item The hom-associativity constraint $a_{U,V,W}:\left(U\otimes V\right)\otimes W^{\psi}\rightarrow U^{\psi}\otimes\left(V\otimes W\right)$ is defined by the natural isomorphism $a_{U, V, W}\left(\left(u\otimes v\right)\otimes\bar{w}\right)=\bar{u}\otimes\left(v\otimes w\right)$, for all $u\in U$, $v\in V$ and $\bar{w}\in W^{\psi}$ such that 
$\left(U,\alpha_{U}\right)$, $\left(V,\alpha_{V}\right)$ and $\left(W,\alpha_{W}\right)$ are objects in $H\textbf{-mod}$.  

\item $G:H\textbf{-mod}\rightarrow H\textbf{-mod}$ is the covariant functor defined by
$G\left(\left(U,\alpha_{U}\right)\right)=\left(U^{\alpha},\alpha_{U^{\alpha}}\right)$, 
 where $U^{\alpha}=U$ as a $\h$-vector space (we will denote an element of $U^{\alpha}=U$ as $\widetilde{u}$), and the $H$-module structure $H\otimes U^{\alpha}\rightarrow U^{\alpha}$ is given by
$h\cdot_{\alpha}\widetilde{u}=\widetilde{\alpha_{H}\left(h\right)\cdot u}$,
for all $h\in H$, $u\in U$. Furthermore, the $\h$-linear map $\alpha_{U^{\alpha}}:U^{\alpha}\rightarrow U^{\alpha}$ is defined by $\alpha_{U^{\alpha}}\left(\widetilde{u}\right)=\widetilde{\alpha_{U}\left(u\right)}$ for all $\widetilde{u}\in U^{\alpha}$. For all morphisms $f:U\to V$ we have   $G\left(f\right)\left(\widetilde{u}\right)=\widetilde{f\left(u\right)}$ for all $\widetilde{u}\in U^{\alpha}$.
\item $\Phi:U\to U^{\alpha}$ is determined  by $\Phi_U(u)=\widetilde{\alpha_U(u)}$, for all $u\in U$.  

\end{enumerate}
\end{enumerate}
\label{prop1}  
\end{Prop}
\begin{Rem}
Notice that as $\h$-vector spaces we have $U^{\psi}\otimes V^{\psi}=U\otimes V=\left(U\otimes V\right)^{\psi}$.  This allows us to identify  $\bar{u}\otimes\bar{v}\in U^{\psi}\otimes V^{\psi}$ with $\overline{u\otimes v}\in (U\otimes V)^{\psi}$. Also notice that if $\bar{u_1}=\bar{u_2}\in U^{\psi}$ then $u_1=u_2\in U$. A similar statement is true for $U^{\alpha}$.
\label{rem6}
\end{Rem}
\begin{proof}
$\left(A\right)\Rightarrow\left(B\right)$ Suppose that $H=\left(H,m_H,\Delta_{H},\alpha_{H},\psi_H\right)$ is a hom-bialgebra. To begin, we want to show that the tensor product is well-defined in $\mathscr{H}$. Let $\left(U,\alpha_{U}\right),\left(V,\alpha_{V}\right)\in\textbf{Ob}\left(\mathscr{H}\right)$. Our first claim is that $\left(U\otimes V,\alpha_{U\otimes V}\right)\in\textbf{Ob}\left(\mathscr{H}\right)$. We need to check that both conditions (\ref{eq8}) and (\ref{eq9}) hold for $U\otimes V$.

When using a certain property to establish a particular equality within the following computations, we will indicate that property above the corresponding equal sign. So, suppose that $h\in H$, $u\in U$ and $v\in V$. Then checking for property (\ref{eq8}) of $U\otimes V$ results in
\begin{eqnarray*}
\left(\alpha_{U\otimes V}\right)\left(h\cdot\left(u\otimes v\right)\right) &\stackrel{\left(\text{iii}\right)}{=}& \left(\alpha_{U}\otimes\alpha_{V}\right)\sum\left(h_{\left(1\right)}\cdot u\right)\otimes\left(h_{\left(2\right)}\cdot v\right)\\
 &=& \sum\alpha_{U}\left(h_{\left(1\right)}\cdot u\right)\otimes\alpha_{V}\left(h_{\left(2\right)}\cdot v\right)\\
 &\stackrel{\left(\ref{eq8}\right)}{=}& \sum\left(\alpha_{H}\left(h_{\left(1\right)}\right)\cdot\alpha_{U}\left(u\right)\right)\otimes\left(\alpha_{H}\left(h_{\left(2\right)}\right)\cdot\alpha_{V}\left(v\right)\right)\\
 &\stackrel{\left(\ref{eq7}\right)}{=}& \sum\left(\left(\alpha_{H}\left(h\right)\right)_{\left(1\right)}\cdot\alpha_{U}\left(u\right)\right)\otimes\left(\left(\alpha_{H}\left(h\right)\right)_{\left(2\right)}\cdot\alpha_{V}\left(v\right)\right)\\
 &\stackrel{\left(\text{iii}\right)}{=}& \alpha_{H}\left(h\right)\cdot\left(\alpha_{U}\left(u\right)\otimes\alpha_{V}\left(v\right)\right)
 =\alpha_{H}\left(h\right)\cdot\left(\alpha_{U}\otimes\alpha_{V}\right)\left(u\otimes v\right)\\
 &\stackrel{\left(\text{iii}\right)}{=}& \alpha_{H}\left(h\right)\cdot\left(\alpha_{U\otimes V}\right)\left(u\otimes v\right).
\end{eqnarray*} 

Next, suppose that $h,h^{\prime}\in H$, $u\in U$ and $v\in V$. Checking for condition (\ref{eq9}) of $U\otimes V$ results in
\begin{eqnarray*}
\alpha_{H}\left(h\right)\cdot\left(h^{\prime}\cdot\left(u\otimes v\right)\right) &\stackrel{\left(\text{iii}\right)}{=}& \alpha_{H}\left(h\right)\cdot\sum\left(h^{\prime}_{\left(1\right)}\cdot u\right)\otimes\left(h^{\prime}_{\left(2\right)}\cdot v\right)\\
 &\stackrel{\left(\text{iii}\right)}{=}& \sum\left(\left(\alpha_{H}\left(h\right)\right)_{\left(1\right)}\cdot\left(h^{\prime}_{\left(1\right)}\cdot u\right)\right)\otimes\left(\left(\alpha_{H}\left(h\right)\right)_{\left(2\right)}\cdot\left(h^{\prime}_{\left(2\right)}\cdot v\right)\right)\\
 &\stackrel{\left(\ref{eq7}\right)}{=}& \sum\left(\alpha_{H}\left(h_{\left(1\right)}\right)\cdot\left(h^{\prime}_{\left(1\right)}\cdot u\right)\right)\otimes\left(\alpha_{H}\left(h_{\left(2\right)}\right)\cdot\left(h^{\prime}_{\left(2\right)}\cdot v\right)\right)\\
 &\stackrel{\left(\ref{eq9}\right)}{=}& \sum\left(\left(h_{\left(1\right)}h^{\prime}_{\left(1\right)}\right)\cdot\alpha_{U}\left(u\right)\right)\otimes\left(\left(h_{\left(2\right)}h^{\prime}_{\left(2\right)}\right)\cdot\alpha_{V}\left(v\right)\right)\\
 &\stackrel{\left(\ref{eq6}\right)}{=}& \sum\left(\left(hh^{\prime}\right)_{\left(1\right)}\cdot\alpha_{U}\left(u\right)\right)\otimes\left(\left(hh^{\prime}\right)_{\left(2\right)}\cdot\alpha_{V}\left(v\right)\right)\\
 &\stackrel{\left(\text{iii}\right)}{=}& \left(hh^{\prime}\right)\cdot\left(\alpha_{U}\otimes\alpha_{V}\right)\left(u\otimes v\right)
=\left(hh^{\prime}\right)\cdot\left(\alpha_{U\otimes V}\right)\left(u\otimes v\right).
\end{eqnarray*}
So condition (\ref{eq9}) holds for $U\otimes V$. Thus $\left(U\otimes V,\alpha_{U\otimes V}\right)\in\textbf{Ob}\left(\mathscr{H}\right)$.

Let $\left(U,\alpha_{U}\right),\left(V,\alpha_{V}\right),\left(U^{\prime},\alpha_{U^{\prime}}\right),\left(V^{\prime},\alpha_{V^{\prime}}\right)\in\textbf{Ob}\left(\mathscr{H}\right)$, $f\in\textbf{Hom}_{\mathscr{H}}\left(U,U^{\prime}\right)$ and $f^{\prime}\in\textbf{Hom}_{\mathscr{H}}\left(V,V^{\prime}\right)$. We claim that $f\otimes f^{\prime}\in\textbf{Hom}_{\mathscr{H}}\left(U\otimes V,U^{\prime}\otimes V^{\prime}\right)$. First we check that 
$\left(\alpha_{U^{\prime}\otimes V^{\prime}}\right)\left(f\otimes f^{\prime}\right)=\left(f\otimes f^{\prime}\right)\left(\alpha_{U\otimes V}\right)$:
\begin{eqnarray*}
 &&\left(\alpha_{U^{\prime}\otimes V^{\prime}}\right)\left(f\otimes f^{\prime}\right)\left(u\otimes v\right) =\left(\alpha_{U^{\prime}}\otimes\alpha_{V^{\prime}}\right)\left(f\left(u\right)\otimes f^{\prime}\left(v\right)\right)
 = \alpha_{U^{\prime}}\left(f\left(u\right)\right)\otimes \alpha_{V^{\prime}}\left(f^{\prime}\left(v\right)\right), \\
&&\left(f\otimes f^{\prime}\right)\left(\alpha_{U\otimes V}\right)\left(u\otimes v\right) = \left(f\otimes f^{\prime}\right)\left(\alpha_{U}\otimes\alpha_{V}\right)\left(u\otimes v\right)
=f\left(\alpha_{U}\left(u\right)\right)\otimes f^{\prime}\left(\alpha_{V}\left(v\right)\right).
\end{eqnarray*}
Since $f\in\textbf{Hom}_{\mathscr{H}}\left(U,U^{\prime}\right)$ and $f^{\prime}\in\textbf{Hom}_{\mathscr{H}}\left(V,V^{\prime}\right)$ we have that $\alpha_{U^{\prime}}\left(f\left(u\right)\right)=f\left(\alpha_{U}\left(u\right)\right)$ and $\alpha_{V^{\prime}}\left(f^{\prime}\left(v\right)\right)=f^{\prime}\left(\alpha_{V}\left(v\right)\right)$ for all $u\in U$, 
$v\in V$. Thus the two expressions above are equal.

Next, for $u\in U$, $v\in V$, $h\in H$, we check that 
$\left(f\otimes f^{\prime}\right)\left(h\cdot\left(u\otimes v\right)\right)=
h\cdot\left(f\otimes f^{\prime}\right)\left(u\otimes v\right)$:
\begin{eqnarray*}
\left(f\otimes f^{\prime}\right)\left(h\cdot\left(u\otimes v\right)\right) &\stackrel{\left(\text{iii}\right)}{=}& \left(f\otimes f^{\prime}\right)\left(\sum\left(h_{\left(1\right)}\cdot u\right)\otimes\left(h_{\left(2\right)}\cdot v\right)\right)
 =\sum f\left(h_{\left(1\right)}\cdot u\right)\otimes f^{\prime}\left(h_{\left(2\right)}\cdot v\right)\\
 &=& \sum\left(h_{\left(1\right)}\cdot f\left(u\right)\right)\otimes\left(h_{\left(2\right)}\cdot f^{\prime}\left(v\right)\right)\\
&=&h\cdot\left(f\left(u\right)\otimes f^{\prime}\left(v\right)\right)
=h\cdot\left(f\otimes f^{\prime}\right)\left(u\otimes v\right).
\end{eqnarray*}
Therefore, $f\otimes f^{\prime}\in\textbf{Hom}_{\mathscr{H}}\left(U\otimes V,U^{\prime}\otimes V^{\prime}\right)$.

Now that the hom-tensor product is well-defined for objects and morphisms in $\mathscr{H}$ one can easily show that the remaining properties for being functorial are satisfied. That is, if we let $f\in\textbf{Hom}_{H\text{-mod}}\left(S,T\right)$,\\
 $f^{\prime}\in\textbf{Hom}_{H\text{-mod}}\left(T,W\right)$, $g\in\textbf{Hom}_{H\text{-mod}}\left(U,V\right)$ and $g^{\prime}\in\textbf{Hom}_{H\text{-mod}}\left(V,X\right)$ then  $\left(f^{\prime}\otimes g^{\prime}\right)\circ\left(f\otimes g\right)=\left(f^{\prime}\circ f\right)\otimes\left(g^{\prime}\circ g\right)$. 

Furthermore, for any $\left(U,\alpha_{U}\right),\left(V,\alpha_{V}\right)\in\textbf{Ob}\left(H\text{-mod}\right)$, we have that  $\text{id}_{U\otimes V}=\text{id}_{U}\otimes\text{id}_{V}$.  Thus, the hom-tensor product $\otimes:H\textbf{-mod}\times H\textbf{-mod}\rightarrow H\textbf{-mod}$ of $\mathscr{H}$ is a covariant functor. 

Next we will show that $F:\mathscr{H}\rightarrow \mathscr{H}$ is a covariant functor.  We want to show for any object $\left(U,\alpha_{U}\right)$ in $H\textbf{-mod}$ that $\left(U^{\psi},\alpha_{U^{\psi}}\right)$ is an $H$-module. Let $h\in H$ and $\bar{u}\in U^{\psi}$. We check condition (\ref{eq8}): 
\begin{eqnarray*}
\alpha_{U^{\psi}}\left(h\cdot_{\psi}\bar{u}\right) &\stackrel{\left(\text{iv}\right)}{=}& \alpha_{U^{\psi}}\left(\overline{\psi_{H}\left(h\right)\cdot u}\right)
=\overline{\alpha_{U}\left(\psi_{H}\left(h\right)\cdot u\right)}\\
&\stackrel{\left(\ref{eq8}\right)}{=}& \overline{\alpha_{H}\left(\psi_{H}\left(h\right)\right)\cdot\alpha_{U}\left(u\right)}
=\overline{\psi_{H}\left(\alpha_{H}\left(h\right)\right)\cdot\alpha_{U}\left(u\right)}\\
&\stackrel{\left(\text{iv}\right)}{=}& \alpha_{H}\left(h\right)\cdot_{\psi}\overline{\alpha_{U}\left(u\right)}
=\alpha_{H}\left(h\right)\cdot_{\psi}\alpha_{U^{\psi}}\left(\bar{u}\right).
\end{eqnarray*} 

Let $h,h^{\prime}\in H$ and $u\in U$ where $\left(U,\alpha_{U}\right)$ is an object in $H\textbf{-mod}$. Then checking condition (\ref{eq9})  results in
\begin{eqnarray*}
\alpha_{H}\left(h\right)\cdot_{\psi}\left(h^{\prime}\cdot_{\psi}\bar{u}\right) &\stackrel{\left(\text{iv}\right)}{=}& \alpha_{H}\left(h\right)\cdot_{\psi}\overline{\left(\psi_{H}\left(h^{\prime}\right)\cdot u\right)}\\
&\stackrel{\left(\text{iv}\right)}{=}& \overline{\psi_{H}\left(\alpha_{H}\left(h\right)\right)\cdot\left(\psi_{H}\left(h^{\prime}\right)\cdot u\right)}
=\overline{\alpha_{H}\left(\psi_{H}\left(h\right)\right)\cdot\left(\psi_{H}\left(h^{\prime}\right)\cdot u\right)}\\
&\stackrel{\left(\ref{eq9}\right)}{=}& \overline{\left(\psi_{H}\left(h\right)\psi_{H}\left(h^{\prime}\right)\right)\cdot\alpha_{U}\left(u\right)}\\
&\stackrel{\left(\ref{7112}\right)}{=}& \overline{\psi_{H}\left(\left(hh^{\prime}\right)\right)\cdot\alpha_{U}\left(u\right)}\\
&\stackrel{\left(\text{iv}\right)}{=}&\left(hh^{\prime}\right)\cdot_{\psi}\overline{\alpha_{U}\left(u\right)}=
\left(hh^{\prime}\right)\cdot_{\psi}\alpha_{U^{\psi}}\left(\bar{u}\right).
\end{eqnarray*}
Thus, $\left(U^{\psi},\alpha_{U^{\psi}}\right)$ is an object in $\mathscr{H}$.

Next we claim that $F$ maps morphisms to morphisms in $\mathscr{H}$.  
Let $f\in\textbf{Hom}_{H\textbf{-mod}}\left(V,W\right)$ and $\bar{v}\in V^{\psi}$. 
We will prove our claim by checking the following two properties:
\begin{eqnarray}
&&\left(\alpha_{W^{\psi}}\circ F\left(f\right)\right)\left(\bar{v}\right)=\left(F\left(f\right)\circ\alpha_{V^{\psi}}\right)\left(\bar{v}\right),
\label{eq32}\\
&&F\left(f\right)\left(h\cdot_{\psi}\bar{v}\right)=h\cdot_{\psi}F\left(f\right)\left(\bar{v}\right).
\label{eq33}
\end{eqnarray}
Relation (\ref{eq32}) is a consequence of the following computations:
\begin{eqnarray*}
&&\left(\alpha_{W^{\psi}}\circ F\left(f\right)\right)\left(\bar{v}\right) 
=\alpha_{W^{\psi}}\circ F\left(f\right)\left(\bar{v}\right)=
\alpha_{W^{\psi}}\left(\overline{f\left(v\right)}\right)=
\overline{\alpha_{W}\left(f\left(v\right)\right)}, \\
&&\left(F\left(f\right)\circ\alpha_{V^{\psi}}\right)\left(\bar{v}\right)=
 F\left(f\right)\left(\alpha_{V^{\psi}}\left(\bar{v}\right)\right)=
 F\left(f\right)\left(\overline{\alpha_{V}\left(v\right)}\right)=
\overline{f\left(\alpha_{V}\left(v\right)\right)}=\overline{\alpha_{W}\left(f\left(v\right)\right)}.
\end{eqnarray*}
Relation (\ref{eq33}) holds by the following sequence of equalities:
\begin{eqnarray*}
F\left(f\right)\left(h\cdot_{\psi}\bar{v}\right) &\stackrel{\left(\text{iv}\right)}{=}& F\left(f\right)\left(\overline{\psi_{H}\left(h\right)\cdot v}\right)=\overline{f\left(\psi_{H}\left(h\right)\cdot v\right)}
=\overline{\psi_{H}\left(h\right)\cdot f\left(v\right)}\\
&\stackrel{\left(\text{iv}\right)}{=}& h\cdot_{\psi}\overline{f\left(v\right)}=
h\cdot_{\psi} F\left(f\right)\left(\bar{v}\right).
\end{eqnarray*}

The fact that $F$ preserves compositions of morphisms and the identity morphism in $H\textbf{-mod}$ is obvious. 
Therefore, $F$ is indeed a covariant functor.

Next we need to check that 
$F\left(\left(U,\alpha_{U}\right)\otimes\left(V,\alpha_{V}\right)\right)=
F\left(\left(U,\alpha_{U}\right)\right)\otimes F\left(\left(V,\alpha_{V}\right)\right)$, 
for all $\left(U,\alpha_{U}\right),\left(V,\alpha_{V}\right)\in\textbf{Ob}\left(H\textbf{-mod}\right)$. As noticed in Remark \ref{rem6}, we already have that identification at the level of $\h$-vector spaces. We need to show that the $H$-module 
structure is preserved, too. 
Let $u\in U$, $v\in V$, 
$h\in H$. Then: 
\begin{eqnarray*}
F\left(\left(U,\alpha_{U}\right)\otimes\left(V,\alpha_{V}\right)\right) &\stackrel{\left(\text{iii}\right)}{=}& F\left(U\otimes V,\alpha_{U\otimes V}\right)=\left(\left(U\otimes V\right)^{\psi},\alpha_{\left(U\otimes V\right)^{\psi}}\right), 
\end{eqnarray*}
 with left $H$-action given by
 $h\cdot_{\psi}\overline{\left(u\otimes v\right)}=\overline{\psi_{H}\left(h\right)\cdot \left(u\otimes v\right)}$, and 
\begin{eqnarray*}
F\left(\left(U,\alpha_{U}\right)\right)\otimes F\left(\left(V,\alpha_{V}\right)\right) &\stackrel{\left(\text{iv}\right)}{=}& \left(U^{\psi},\alpha_{U^{\psi}}\right)\otimes\left(V^{\psi},\alpha_{V^{\psi}}\right)\\
 &\stackrel{\left(\text{iii}\right)}{=}& \left(U^{\psi}\otimes V^{\psi},\alpha_{U^{\psi}\otimes V^{\psi}}\right)
 =\left(U^{\psi}\otimes V^{\psi},\alpha_{U^{\psi}}\otimes\alpha_{V^{\psi}}\right), 
 \end{eqnarray*}
with left $H$-action given by
\begin{equation*}
 h\cdot\left(\bar{u}\otimes\bar{v}\right)\stackrel{\left(\text{iii}\right)}{=}\sum h_{\left(1\right)}\cdot_{\psi}\bar{u}\otimes h_{\left(2\right)}\cdot_{\psi}\bar{v}\stackrel{\left(\text{iv}\right)}{=}\sum\overline{\psi_{H}\left(h_{\left(1\right)}\right)\cdot u}\otimes\overline{\psi_{H}\left(h_{\left(2\right)}\right)\cdot v}.
\end{equation*}

\noindent We prove that the two left $H$-actions coincide: 
\begin{eqnarray*}
h\cdot_{\psi}\overline{\left(u\otimes v\right)} &=& \overline{\psi_{H}\left(h\right)\cdot\left(u\otimes v\right)} \\
 &\stackrel{\left(\text{iii}\right)}{=}& \sum\overline{\left(\psi_{H}\left(h\right)\right)_{\left(1\right)}\cdot u\otimes\left(\psi_{H}\left(h\right)\right)_{\left(2\right)}\cdot v} \\
&\stackrel{Remark \;\ref{rem6}}{=}& \sum\overline{\left(\psi_{H}\left(h\right)\right)_{\left(1\right)}\cdot u}\otimes\overline{\left(\psi_{H}\left(h\right)\right)_{\left(2\right)}\cdot v}\\
&\stackrel{\left(\ref{eq7111}\right)}{=}&  \sum\overline{\psi_{H}\left(h_{\left(1\right)}\right)\cdot u}\otimes\overline{\psi_{H}\left(h_{\left(2\right)}\right)\cdot v}\\
&\stackrel{}{=}&  \sum h_{\left(1\right)}\cdot_{\psi} \overline{u}\otimes h_{\left(2\right)}\cdot_{\psi} \overline{v}
=h\cdot (\overline{u}\otimes  \overline{v}), \;\;\;q.e.d.
\end{eqnarray*} 

The fact that $\alpha_{\left(U\otimes V\right)^{\psi}}
=\alpha_{U^{\psi}}\otimes\alpha_{V^{\psi}}$ is obvious. 
 Additionally, by the definition of $F$ one can easily show that $F\left(f\otimes f^{\prime}\right)=F\left(f\right)\otimes F\left(f^{\prime}\right)$ for any morphisms $f,f^{\prime}$ in $H\textbf{-mod}$.

Next we need to check that
\begin{equation}
a_{U^{\psi},V^{\psi},W\otimes X}\circ a_{U\otimes V,W^{\psi},X^{\psi}}=\left(\text{id}_{U^{\psi^{2}}}\otimes a_{V,W,X}\right)\circ a_{U^{\psi},V\otimes W,X^{\psi}}\circ\left(a_{U,V,W}\otimes\text{id}_{X^{\psi^{2}}}\right).
\label{eq13}
\end{equation}

Let $u\in U$, $v\in V$, $w\in W$ and $x\in X$, where $\left(U,\alpha_{U}\right)$, $\left(V,\alpha_{V}\right)$, $\left(W,\alpha_{W}\right)$ and $\left(X,\alpha_{X}\right)$ are objects in $H\textbf{-mod}$. For the left hand side of (\ref{eq13}) we have that
\begin{eqnarray*}
\left(a_{U^{\psi},V^{\psi},W\otimes X}\circ a_{U\otimes V,W^{\psi},X^{\psi}}\right)\left(\left(\left(u\otimes v\right)\otimes\bar{w}\right)\otimes\bar{\bar{x}}\right) &\stackrel{\left(\text{v}\right)}{=}& a_{U^{\psi},V^{\psi},W\otimes X}\left(\left(\overline{u\otimes v}\right)\otimes\left(\bar{w}\otimes\bar{x}\right)\right)\\
 &\stackrel{Remark \;\ref{rem6}}{=}& a_{U^{\psi},V^{\psi},W\otimes X}\left(\left(\bar{u}\otimes\bar{v}\right)\otimes\left(\overline{w\otimes x}\right)\right)\\
 &\stackrel{\left(\text{v}\right)}{=}& \bar{\bar{u}}\otimes\left(\bar{v}\otimes\left(w\otimes x\right)\right).
\end{eqnarray*}
For the right hand side of (\ref{eq13}) we have that
\begin{equation*}
\left(\left(\text{id}_{U^{\psi^{2}}}\otimes a_{V,W,X}\right)\circ a_{U^{\psi},V\otimes W,X^{\psi}}\circ\left(a_{U,V,W}\otimes\text{id}_{X^{\psi^{2}}}\right)\right)\left(\left(\left(u\otimes v\right)\otimes\bar{w}\right)\otimes\bar{\bar{x}}\right)\vspace{-3mm}
\end{equation*}
\begin{eqnarray*}
&\stackrel{\left(\text{v}\right)}{=}&\left(\left(\text{id}_{U^{\psi^{2}}}\otimes a_{V,W,X}\right)\circ a_{U^{\psi},V\otimes W,X^{\psi}}\right)\left(\left(\bar{u}\otimes\left(v\otimes w\right)\right)\otimes\bar{\bar{x}}\right)\\
&=&\left(\text{id}_{U^{\psi^{2}}}\otimes a_{V,W,X}\right)\left(\bar{\bar{u}}\otimes\left(\left(v\otimes w\right)\otimes\bar{x}\right)\right)\\
&\stackrel{\left(\text{v}\right)}{=} &\bar{\bar{u}}\otimes\left(\bar{v}\otimes\left(w\otimes x\right)\right). 
\end{eqnarray*} 
Thus, the hom-associativity constraint $a$ satisfies the ``Pentagon'' axiom from Figure \ref{fig1}.

Next we show that the constraint map $a$ is a morphism in $\mathscr{H}$. First, we will check that
\begin{equation}
    \alpha_{U^{\psi}\otimes\left(V\otimes W\right)}\circ a_{U,V,W}\left(\left(u\otimes v\right)\otimes\bar{w}\right)=a_{U,V,W}\circ\alpha_{\left(U\otimes V\right)\otimes W^{\psi}}\left(\left(u\otimes v\right)\otimes\bar{w}\right)
\label{eq14}
\end{equation}
for all $u\in U$, $v\in V$ and $w\in W$, where $\left(U,\alpha_{U}\right)$, $\left(V,\alpha_{V}\right)$ and $\left(W,\alpha_{W}\right)$ are objects in $\mathscr{H}$. We compute:
\begin{eqnarray*}
\alpha_{U^{\psi}\otimes\left(V\otimes W\right)}\circ a_{U,V,W}\left(\left(u\otimes v\right)\otimes\bar{w}\right) &\stackrel{\left(\text{v}\right)}{=}&  \alpha_{U^{\psi}\otimes\left(V\otimes W\right)}\left(\bar{u}\otimes\left(v\otimes w\right)\right)\\
 &=& \alpha_{U^{\psi}}\left(\bar{u}\right)\otimes\alpha_{V\otimes W}\left(v\otimes w\right)\\
 &=& \alpha_{U^{\psi}}\left(\bar{u}\right)\otimes\left(\alpha_{V}\left(v\right)\otimes\alpha_{W}\left(w\right)\right)\\
 &\stackrel{\left(\text{iv}\right)}{=}& \overline{\alpha_{U}\left(u\right)}\otimes\left(\alpha_{V}\left(v\right)\otimes\alpha_{W}\left(w\right)\right), 
\end{eqnarray*}
\begin{eqnarray*}
a_{U,V,W}\circ\alpha_{\left(U\otimes V\right)\otimes W^{\psi}}\left(\left(u\otimes v\right)\otimes\bar{w}\right) &=& a_{U,V,W}\left(\alpha_{U\otimes V}\left(u\otimes v\right)\otimes\alpha_{W^{\psi}}\left(\bar{w}\right)\right)\\
 &=& a_{U,V,W}\left(\left(\alpha_{U}\left(u\right)\otimes\alpha_{V}\left(v\right)\right)\otimes\alpha_{W^{\psi}}\left(\bar{w}\right)\right)\\
  &\stackrel{\left(\text{iv}\right)}{=}& a_{U,V,W}\left(\left(\alpha_{U}\left(u\right)\otimes\alpha_{V}\left(v\right)\right)\otimes\overline{\alpha_{W}\left(w\right)}\right)\\
 &\stackrel{\left(\text{v}\right)}{=}& \overline{\alpha_{U}\left(u\right)}\otimes\left(\alpha_{V}\left(v\right)\otimes\alpha_{W}\left(w\right)\right).
\end{eqnarray*} 
So (\ref{eq14}) holds. Next we show that 
$a_{U,V,W}\left(h\cdot\left(\left(u\otimes v\right)\otimes\bar{w}\right)\right)=h\cdot a_{U,V,W}\left(\left(u\otimes v\right)\otimes\bar{w}\right)$, 
for all $h\in H$, $u\in U$, $v\in V$ and $w\in W$, where $\left(U,\alpha_{U}\right)$, $\left(V,\alpha_{V}\right)$ and $\left(W,\alpha_{W}\right)$ are objects in $\mathscr{H}$:

\begin{eqnarray*}
a_{U,V,W}\left(h\cdot\left(\left(u\otimes v\right)\otimes\bar{w}\right)\right) &\stackrel{\left(\text{iii}\right)}{=}& a_{U,V,W}\left(\sum\left(h_{\left(1\right)}\cdot\left(u\otimes v\right)\right)\otimes\left(h_{\left(2\right)}\cdot_{\psi}\bar{w}\right)\right)\\
 &\stackrel{\left(\text{iii}\right),\left(\text{iv}\right)}{=}& a_{U,V,W}\left(\left(\sum\left(h_{\left(1\right)_{\left(1\right)}}\cdot u\right)\otimes\left(h_{\left(1\right)_{\left(2\right)}}\cdot v\right)\right)\otimes\left(\overline{\psi_{H}\left(h_{\left(2\right)}\right)\cdot w}\right)\right)\\
&\stackrel{\left(\text{v}\right)}{=}& \sum\left(\overline{h_{\left(1\right)_{\left(1\right)}}\cdot u}\right)\otimes\left(\left(h_{\left(1\right)_{\left(2\right)}}\cdot v\right)\otimes\left(\psi_{H}\left(h_{\left(2\right)}\right)\cdot w\right)\right)\\
 &\stackrel{\left(\ref{eq5}\right)}{=}& \sum\overline{\psi_{H}\left(h_{\left(1\right)}\right)\cdot u}\otimes\left(\left(h_{\left(2\right)_{\left(1\right)}}\cdot v\right)\otimes\left(h_{\left(2\right)_{\left(2\right)}}\cdot w\right)\right)\\
 &\stackrel{\left(\text{iv}\right),\left(\text{iii}\right)}{=}& \sum\left(h_{\left(1\right)}\cdot_{\psi}\bar{u}\right)\otimes h_{\left(2\right)}\cdot\left(v\otimes w\right)\\
 &\stackrel{\left(\text{iii}\right)}{=}& h\cdot\left(\bar{u}\otimes\left(v\otimes w\right)\right)\\
 &\stackrel{\left(\text{v}\right)}{=}& h\cdot a_{U,V,W}\left(\left(u\otimes v\right)\otimes\bar{w}\right), \;\;\;q.e.d.
\end{eqnarray*}

Thus, $a$ is a morphism in $\mathscr{H}$. The fact that $a$ is an isomorphism follows directly from the definition.
Just like for $F$ one can show that $G$ is a  functor and  $G((U,\alpha_U)\otimes (V,\alpha_V))=G(U,\alpha_U)\otimes G(V,\alpha_V)$ and $G(f\otimes g)=G(f)\otimes G(g)$. Since $\alpha_H\circ \psi _H=\psi _H\circ \alpha _H$, one can easily see that $FG=GF$. 

Next we check that $\Phi_U$ is a morphism of $H$-modules. Indeed, 
$$\Phi_U(h\cdot u)=\widetilde{\alpha_U(h\cdot u)}=
\widetilde{\alpha_H(h)\cdot \alpha_U(u)}
=h\cdot_{\alpha}\widetilde{\alpha_U(u)}
=h\cdot_{\alpha}\Phi_U(u),$$
$$\Phi_U(\alpha_U(u))=\widetilde{\alpha _U(\alpha_U(u))}
=\alpha_{U^{\alpha}}\widetilde{\alpha_U(u)}
=\alpha_{U^{\alpha}}(\Phi_U(u)).$$
Moreover, if $f:U\to V$ is a morphism of $H$-modules then we have  $\widetilde{f}\Phi_U=\Phi_Vf$. In other words $\Phi$ is a natural transformation between the functors $\text{id}_{H\text{-mod}}$ and $G$. 
It is obvious that $F(\Phi _U)=\Phi _{F(U)}$, $G(\Phi_U)=\Phi_{G(U)}$ and $\Phi _{U\otimes V}=\Phi_U\otimes \Phi _V$.

Therefore, $\mathscr{H}=\left(H\textbf{-mod},\otimes,F, G,a,\Phi\right)$ is a hom-tensor category.


$\left(B\right)\Rightarrow\left(A\right)$ Suppose that $\mathscr{H}=\left(H\textbf{-mod},\otimes,F,G,a,\Phi\right)$ is a hom-tensor category and $H$ is nondegenerate. In view of Definition \ref{defhombialgebra} and given our assumptions,  in order for $\left(H,m_{H},\Delta_{H},\alpha_{H},\psi_H\right)$ to be a hom-bialgebra what remains to be shown is that $\left(H,\Delta_{H},\psi_{H}\right)$ is a hom-coassociative $\h$-coalgebra. 

We begin with proving (\ref{eq3}).  Let $(U, \alpha _U), (V, \alpha _V)\in \mathscr{H}$. 
We have that 
\begin{equation}
F\left(\left(U\otimes V,\alpha_{U\otimes V}\right)\right)=F\left(\left(U,\alpha_{U}\right)\right)\otimes F\left(\left(V,\alpha_{V}\right)\right). 
\label{eq16}
\end{equation}
In particular the map $\overline{(u\otimes v)}\to \overline{u}\otimes \overline{v}$ from Remark \ref{rem6} is a morphism of $H$-modules. 

Let $h\in H$, $u\in U$, and $v\in V$. The $H$-action for the left hand side of (\ref{eq16})
\begin{eqnarray*}
 h\cdot_{\psi}\overline{\left(u\otimes v\right)} &\stackrel{\left(\text{iv}\right)}{=}& \overline{\psi_{H}\left(h\right)\cdot \left(u\otimes v\right)}\\
                                                   &\stackrel{\left(\text{iii}\right)}{=}& \overline{\sum\left(\left(\psi_{H}\left(h\right)\right)_{\left(1\right)}\cdot u\right)\otimes\left(\left(\psi_{H}\left(h\right)\right)_{\left(2\right)}\cdot v\right)}                                                  
\end{eqnarray*}
is equal to the $H$-action for the right hand side of (\ref{eq16}) 
\begin{eqnarray*}
 h\cdot_{\psi}\left(\bar{u}\otimes\bar{v}\right) &\stackrel{\left(\text{iii}\right)}{=}& \sum h_{\left(1\right)}\cdot_{\psi}\bar{u}\otimes h_{\left(2\right)}\cdot_{\psi}\bar{v}\\
                                                   &\stackrel{\left(\text{iv}\right)}{=}& \sum\overline{\psi_{H}\left(h_{\left(1\right)}\right)\cdot u}\otimes\overline{\psi_{H}\left(h_{\left(2\right)}\right)\cdot v}.
\end{eqnarray*} 
That is,
\begin{equation*}
 \overline{\sum\left(\left(\psi_{H}\left(h\right)\right)_{\left(1\right)}\cdot u\right)\otimes\left(\left(\psi_{H}\left(h\right)\right)_{\left(2\right)}\cdot v\right)}=\sum\overline{\psi_{H}\left(h_{\left(1\right)}\right)\cdot u}\otimes\overline{\psi_{H}\left(h_{\left(2\right)}\right)\cdot v},
\label{eq17}
\end{equation*}
which means that 
\begin{equation*}
\sum\left(\left(\psi_{H}\left(h\right)\right)_{\left(1\right)}\otimes\left(\psi_{H}\left(h\right)\right)_{\left(2\right)}\right)\cdot \left(u\otimes v\right)=\sum\left(\psi_{H}\left(h_{\left(1\right)}\right)\otimes\psi_{H}\left(h_{\left(2\right)}\right)\right)\cdot \left(u\otimes v\right).
\label{eq34}
\end{equation*}
Since $H$ is assumed to be nondegenerate and by using Lemma \ref{lemmanondeg},  this equation implies that
\begin{equation*}
\sum\left(\left(\psi_{H}\left(h\right)\right)_{\left(1\right)}\otimes\left(\psi_{H}\left(h\right)\right)_{\left(2\right)}\right)
=\sum\psi_{H}\left(h_{\left(1\right)}\right)\otimes\psi_{H}\left(h_{\left(2\right)}\right), 
\end{equation*} 
or equivalently 
$\left(\Delta_{H}\circ\psi_{H}\right)\left(h\right)=\left(\left(\psi_{H}\otimes\psi_{H}\right)\circ\Delta_{H}\right)\left(h\right)$, 
for all $h\in H$. So (\ref{eq3}) holds for $\Delta_{H}$. 

Next we will show (\ref{eq4}) for $\Delta_{H}$. Since  $\mathscr{H}$ is a hom-tensor category it means that the hom-associativity constraint is a morphism of left $H$-modules. This means that for $h\in H$, $u\in U$, $v\in V$ and $w\in W$ we have 
$a_{U,V,W}\left(h\cdot\left(\left(u\otimes v\right)\otimes\bar{w}\right)\right)=h\cdot a_{U,V,W}\left(\left(u\otimes v\right)\otimes\bar{w}\right)$. For the left and respectively right hand side of this equality we have 
\begin{eqnarray*}
a_{U,V,W}\left(h\cdot\left(\left(u\otimes v\right)\otimes\bar{w}\right)\right) &\stackrel{\left(\text{iii}\right),\left(\text{iv}\right)}{=}& a_{U,V,W}\left(\left(\sum\left(h_{\left(1\right)_{\left(1\right)}}\cdot u\right)\otimes\left(h_{\left(1\right)_{\left(2\right)}}\cdot v\right)\right)\otimes\left(\overline{\psi_{H}\left(h_{\left(2\right)}\right)\cdot w}\right)\right) \\
 &\stackrel{\left(\text{v}\right)}{=}& \sum\left(\overline{h_{\left(1\right)_{\left(1\right)}}\cdot u}\right)\otimes\left(\left(h_{\left(1\right)_{\left(2\right)}}\cdot v\right)\otimes\left(\psi_{H}\left(h_{\left(2\right)}\right)\cdot w\right)\right),
\end{eqnarray*}
\begin{eqnarray*}
h\cdot a_{U,V,W}\left(\left(u\otimes v\right)\otimes\bar{w}\right) &\stackrel{\left(\text{v}\right)}{=}& h\cdot\left(\bar{u}\otimes\left(v\otimes w\right)\right)\\
 &\stackrel{\left(\text{iii}\right),\left(\text{iv}\right)}{=}& \sum\overline{\psi_{H}\left(h_{\left(1\right)}\right)\cdot u}\otimes\left(\left(h_{\left(2\right)_{\left(1\right)}}\cdot v\right)\otimes\left(h_{\left(2\right)_{\left(2\right)}}\cdot w\right)\right).
\end{eqnarray*}
So, we must have
\begin{equation*}
\sum\left(h_{\left(1\right)_{\left(1\right)}}\cdot u\right)\otimes\left(\left(h_{\left(1\right)_{\left(2\right)}}\cdot v\right)\otimes\left(\psi_{H}\left(h_{\left(2\right)}\right)\cdot w\right)\right){=}\sum\psi_{H}\left(h_{\left(1\right)}\right)\cdot u\otimes\left(\left(h_{\left(2\right)_{\left(1\right)}}\cdot v\right)\otimes\left(h_{\left(2\right)_{\left(2\right)}}\cdot w\right)\right),
\end{equation*}
or equivalently 
\begin{equation*}
\left(\sum h_{\left(1\right)_{\left(1\right)}}\otimes h_{\left(1\right)_{\left(2\right)}}\otimes\psi_{H}\left(h_{\left(2\right)}\right)\right)\cdot\left(u\otimes\left(v\otimes w\right)\right)=\left(\sum\psi_{H}\left(h_{\left(1\right)}\right)\otimes h_{\left(2\right)_{\left(1\right)}}\otimes h_{\left(2\right)_{\left(2\right)}}\right)\cdot\left(u\otimes\left(v\otimes w\right)\right).
\label{eq37}
\end{equation*}
 Since $H$ is nondegenerate and by using again Lemma \ref{lemmanondeg}, this equation implies that
\begin{equation*}
\sum h_{\left(1\right)_{\left(1\right)}}\otimes h_{\left(1\right)_{\left(2\right)}}\otimes\psi_{H}\left(h_{\left(2\right)}\right)=\sum\psi_{H}\left(h_{\left(1\right)}\right)\otimes h_{\left(2\right)_{\left(1\right)}}\otimes h_{\left(2\right)_{\left(2\right)}}, 
\end{equation*}
or equivalently 
$\left(\left(\Delta_{H}\otimes\psi_{H}\right)\circ\Delta_{H}\right)\left(h\right)=
\left(\left(\psi_{H}\otimes\Delta_{H}\right)\circ\Delta_{H}\right)\left(h\right)$, 
for all $h\in H$. Thus, $\Delta_{H}$ satisfies (\ref{eq4}) and $\left(H,\Delta_{H},\psi_{H}\right)$ is a hom-coassociative 
$\h$-coalgebra. Therefore, $\left(H,m_{H},\Delta_{H},\alpha_{H},\psi_H\right)$ is a hom-bialgebra.
\end{proof}

\section{Algebras in Hom-Tensor Categories}\label{sec4}

\begin{Def}
Let $(\mathcal{C}, \otimes , F, G, a, \Phi )$ be a hom-tensor category. An \emph{algebra} in $\mathcal{C}$ is 
a pair $(A, \mu _A)$, where $A\in \textbf{Ob}(\mathcal{C})$ and $\mu _A:A\otimes A\rightarrow FG(A)$ is a 
morphism in $\mathcal{C}$ such that the diagram in Figure \ref{fig200} is commutative.

\begin{center}
\begin{figure}[!ht]
\begin{center}
\begin{tikzpicture}[scale=0.9, every node/.style={scale=0.9}]
	\begin{pgfonlayer}{nodelayer}
		\node [style=none] (0) at (-4, -3) {$F\left(A\right)\otimes\left(A\otimes A\right)$};
		\node [style=none] (1) at (4, -3) {$GF\left(A\right)\otimes FG\left(A\right)$};
		\node [style=none] (2) at (4, -2.625) {$\mid\mid$};
		\node [style=none] (3) at (4, -2) {};
		\node [style=none] (4) at (0, -2.75) {\tiny{$\Phi_{F\left(A\right)}\otimes\mu_{A}$}};
		\node [style=none] (5) at (0, 3) {$FG\left(A\right)\otimes GF\left(A\right)$};
		\node [style=none] (6) at (0, 3.375) {$\mid\mid$};
		\node [style=none] (7) at (0, 3.75) {$FG\left(A\otimes A\right)$};
		\node [style=none] (8) at (-4, -0) {$\left(A\otimes A\right)\otimes F\left(A\right)$};
		\node [style=none] (9) at (4, -0) {$FG\left(FG\left(A\right)\right)=F^{2}G^{2}\left(A\right)$};
		\node [style=none] (10) at (-4, -0.25) {};
		\node [style=none] (11) at (-4, 0.25) {};
		\node [style=none] (12) at (4, 0.25) {};
		\node [style=none] (13) at (4, -0.25) {};
		\node [style=none] (14) at (-2.625, -3) {};
		\node [style=none] (15) at (-4, -2.75) {};
		\node [style=none] (16) at (2.5, -3) {};
		\node [style=none] (17) at (-4.5, -1.375) {\tiny{$a_{A,A,A}$}};
		\node [style=none] (18) at (-3.125, 1.5) {\tiny{$\mu_{A}\otimes\Phi_{F\left(A\right)}$}};
		\node [style=none] (19) at (2.875, 1.5) {\tiny{$FG\left(\mu_{A}\right)$}};
		\node [style=none] (20) at (4.625, -1.375) {\tiny{$FG\left(\mu_{A}\right)$}};
		\node [style=none] (21) at (-0.25, 2.75) {};
		\node [style=none] (22) at (0.25, 2.75) {};
		\node [style=none] (23) at (4, -2.25) {$FG\left(A\otimes A\right)$};
	\end{pgfonlayer}
	\begin{pgfonlayer}{edgelayer}
		\draw[->] (10.center) to (15.center);
		\draw[->] (14.center) to (16.center);
		\draw[->] (3.center) to (13.center);
		\draw[->] (11.center) to (21.center);
		\draw[->] (22.center) to (12.center);
	\end{pgfonlayer}
\end{tikzpicture}
\end{center}
\caption{Definition of an algebra in $\mathcal{C}$}
\label{fig200}
\end{figure}
\end{center}

\end{Def}
\begin{Prop}\label{prop4.2}
We denote by $\h\textbf{-mod}$ the category of $\h$-vector spaces with its usual structure as a tensor category and we 
consider the hom-tensor category $\mathfrak{h}(\h\textbf{-mod})$ as in Proposition \ref{mainexample}. Then an algebra 
in $\mathfrak{h}(\h\textbf{-mod})$ is exactly a hom-associative $\h$-algebra. 
\end{Prop}
\begin{proof}
Let $((A, \alpha _A), \mu _A)$ be an algebra in $\mathfrak{h}(\h\textbf{-mod})$. This means that $A$ is a $\h$-vector space, 
$\alpha _A:A\rightarrow A$ is a $\h$-linear map and 
$\mu _A:(A, \alpha _A)\otimes (A, \alpha _A)\rightarrow (A, \alpha _A)$  is a morphism in $\mathfrak{h}(\h\textbf{-mod})$, 
meaning that $\mu _A:A\otimes A\rightarrow A$ is a $\h$-linear map satisfying $\alpha _A\circ \mu _A=
\mu _A\circ (\alpha _A\otimes \alpha _A)$, such that $\mu _A\circ (\mu _A\otimes \Phi _A)=\mu _A\circ (\Phi _A\otimes 
\mu _A)$; the second condition is equivalent to $(ab)\alpha _A(c)=\alpha _A(a)(bc)$, for all $a, b, c\in A$, where we denoted 
$\mu _A(a\otimes b)=ab$, for $a, b\in A$. This means exactly that $(A, \mu_A, \alpha _A)$ is a 
hom-associative $\h$-algebra. 
\end{proof}

We recall the following concept from \cite{liguo}. 
\begin{Def}
A \emph{hom-semigroup} is a set $S$ together with a binary operation $\mu :S\times S\rightarrow S$ (denoted by 
$\mu ((x, y))=xy$ for $x, y\in S$) and a function $\alpha :S\rightarrow S$ satisfying 
$\alpha (x)(yz)=(xy)\alpha (z)$, for all $x, y, z\in S$. 

The hom-semigroup $(S, \mu , \alpha )$ is called \emph{multiplicative} if $\alpha (xy)=\alpha (x)\alpha (y)$ for all $x, y\in S$. 
\end{Def}

We have an analogue of Proposition \ref{prop4.2} for hom-semigroups (the proof is similar and will be omitted). 
\begin{Prop}
Let \textbf{Set} be the category of sets with its usual structure as a pre-tensor category: the tensor product is the 
cartesian product of sets and the associativity constraint is defined by 
\begin{eqnarray*}
&&a_{X, Y, Z}:(X\times Y)\times Z\rightarrow X\times (Y\times Z), \;\;\;\;\;a_{X, Y, Z}(((x, y), z))=(x, (y, z)). 
\end{eqnarray*}
Consider the hom-tensor category  $\mathfrak{h}(\textbf{Set})$ as in Proposition \ref{mainexample}. Then an algebra in 
$\mathfrak{h}(\textbf{Set})$  is exactly a multiplicative hom-semigroup. 
\end{Prop}
\begin{Def}
Let $(H, \mu _H, \Delta _H, \alpha _H, \psi _H)$ be a hom-bialgebra. A hom-associative $\h$-algebra 
$(A, \mu _A, \alpha _A)$ is called a \emph{left $H$-module hom-algebra} if $(A, \alpha _A)$ is a left $H$-module, 
with action denoted by $H\otimes A\rightarrow A$, $h\otimes a\mapsto h\cdot a$, such that the following 
condition is satisfied:
\begin{eqnarray}
&&\alpha _H\psi _H(h)\cdot (aa')=\sum (h_{(1)}\cdot a)(h_{(2)}\cdot a'), \;\;\;\forall \;h\in H, \;a, a'\in A. 
\label{compmodulealgebra}
\end{eqnarray}
\end{Def}
\begin{Rem}
This concept contains as particular cases the concepts of module algebras 
for the situation $\psi _H=\alpha _H$ (introduced in \cite{YA:HA3}) and for the situation $\psi _H=\alpha _H^{-1}$ 
(introduced in \cite{chenwangzhang}).
\end{Rem}
\begin{Prop}
Let $(H, \mu _H, \Delta _H, \alpha _H, \psi _H)$ be a hom-bialgebra and consider the hom-tensor category 
$\mathscr{H}=\left(H\textbf{-mod},\otimes,F,G,a,\Phi\right)$ introduced in Proposition \ref{prop1}. Then an algebra in 
$\mathscr{H}$ is exactly a left $H$-module hom-algebra. 
\end{Prop}
\begin{proof}
Let $((A, \alpha _A), \mu _A)$ be an algebra in $\mathscr{H}$. This means that:\\
(i) $(A, \alpha _A)$ is a left $H$-module (we denote the action of $H$ on $A$ by $h\otimes a\mapsto h\cdot a$);\\
(ii) we have a morphism $\mu _A:(A, \alpha _A)\otimes (A, \alpha _A)\rightarrow 
FG((A, \alpha _A))$ in $\mathscr{H}$, denoted by $\mu_A(a\otimes b)=ab$ for all $a$, $b\in A$. By taking into account the structure of $\mathscr{H}$ as a hom-tensor category 
presented in Proposition \ref{prop1}, this means that $\alpha _A\circ \mu _A=\mu _A\circ (\alpha _A\otimes \alpha _A)$ 
and 
$\mu _A(h\cdot (a\otimes a'))=h\cdot _{\alpha \psi }\mu _A(a\otimes a')$, 
for all $h\in H$ and $a, a'\in A$, which is equivalent to saying that $\alpha _A(aa')=\alpha _A(a)\alpha _A(a')$ and 
$\sum (h_{(1)}\cdot a)(h_{(2)}\cdot a')=
\alpha _H\psi _H(h)\cdot (aa')$. \\
(iii) we have $\mu _A\circ (\mu _A\otimes \alpha _A)=\mu _A\circ (\alpha _A\otimes 
\mu _A)$, which means that $(ab)\alpha _A(c)=\alpha _A(a)(bc)$, for all $a, b, c\in A$.

In conclusion, $((A, \alpha _A), \mu _A)$ is exactly a left $H$-module hom-algebra $(A, \mu _A, \alpha _A)$. 
\end{proof}
The next result may be regarded as a categorical analogue of the Yau twisting. 
\begin{Prop}
Let $(\mathcal{C}, \otimes , a)$ be a pre-tensor category, $(A, \mu _A)$ an algebra in $\mathcal{C}$ and 
$\alpha _A:A\rightarrow A$ an algebra morphism. Define $m_A:A\otimes A\rightarrow A$, 
$m_A:=\alpha _A\circ \mu _A=\mu _A\circ (\alpha _A\otimes \alpha _A)$. Then $((A, \alpha _A), m_A)$ is an algebra 
in the hom-tensor category $\mathfrak{h}(\mathcal{C})$. 
\end{Prop}
\begin{proof}
First, by $\alpha _A\circ \mu _A=\mu _A\circ (\alpha _A\otimes \alpha _A)$ one obtains immediately that 
$\alpha _A\circ m_A=m_A\circ (\alpha _A\otimes \alpha _A)$, that is $m_A:(A, \alpha _A)\otimes 
(A, \alpha _A)\rightarrow (A, \alpha _A)$ is a morphism in $\mathfrak{h}(\mathcal{C})$. So, we only need to prove that 
$m_A\circ (m_A\otimes \Phi _A)=m_A\circ (\Phi _A\otimes m_A)\circ a_{A, A, A}$. We compute:
\begin{eqnarray*}
m_A\circ (\Phi _A\otimes m_A)\circ a_{A, A, A}&=&
m_A\circ (\alpha _A\otimes m_A)\circ a_{A, A, A}\\
&=&\alpha _A\circ \mu _A\circ (\alpha _A\otimes \alpha _A\circ \mu _A)\circ a_{A, A, A}\\
&=&\alpha _A\circ \mu _A\circ (\alpha _A\otimes \alpha _A)\circ (\text{id}_A\otimes \mu _A)\circ a_{A, A, A}\\
&=&\alpha _A^2\circ \mu _A\circ (\text{id}_A\otimes \mu _A)\circ a_{A, A, A}
=\alpha _A^2\circ \mu _A\circ (\mu _A\otimes \text{id}_A)\\
&=&\alpha _A\circ \mu _A\circ (\alpha _A\otimes \alpha _A)\circ (\mu _A\otimes \text{id}_A)\\
&=&\alpha _A\circ \mu _A\circ (\alpha _A\circ \mu _A\otimes \alpha _A)
=m_A\circ (m_A\otimes \Phi _A), 
\end{eqnarray*}
finishing the proof. 
\end{proof}


\section{Hom-Braided Categories}\label{sec5}

We introduce  hom-braided categories and present their connection with quasitriangular hom-bialgebras. 
Here $\tau:\mathcal{C}\times\mathcal{C}\rightarrow\mathcal{C}\times\mathcal{C}$ is the functor defined by $\tau\left(V,W\right)=\left(W,V\right)$ for any pair of objects in a category $\mathcal{C}$.

\begin{Def}
Let $\mathscr{C}=\left(\mathcal{C},\otimes,F,G,a,\Phi\right)$ be a hom-tensor category. A \emph{hom-commutativity constraint} $d$ is a natural morphism $d:\otimes\rightarrow\otimes\tau (G\times G)$. That is, for any $V, W\in \textbf{Ob}(\mathscr{C})$ we have the morphism $d_{V,W}:V\otimes W\rightarrow G\left(W\right)\otimes G\left(V\right)$ such that the diagram in Figure \ref{fig2} commutes for all morphisms $f$, $g\in\mathscr{C}$.

\begin{center}
\begin{figure}[!ht]
\begin{center}
\begin{tikzpicture}[scale=0.9, every node/.style={scale=0.9}]
  \matrix (m) [matrix of math nodes,row sep=3em,column sep=4em,minimum width=2em] 
     {V\otimes W             & G\left(W\right)\otimes G\left(V\right)\\
      V^{\prime}\otimes W^{\prime}     & G\left(W^{\prime}\right)\otimes G\left(V^{\prime} \right)\\};
  \path[-stealth]
	  (m-1-1) edge node [above] {\tiny{$d_{V,W}$}} (m-1-2)
            edge node [right] {\tiny{$f\otimes g$}} (m-2-1)
    (m-1-2) edge node [right] {\tiny{$G\left(g\right)\otimes G\left(f\right)$}} (m-2-2)
		(m-2-1) edge node [above] {\tiny{$d_{V^{\prime},W^{\prime}}$}} (m-2-2);
\end{tikzpicture}
\end{center}
\caption{The naturality of $d_{V,W}:V\otimes W\rightarrow G\left(W\right)\otimes G\left(V\right)$}
\label{fig2}
\end{figure}
\end{center}
\label{def8}
\end{Def}
Next we introduce the analog of the hexagon axiom in the context of hom-tensor categories. 
\begin{Def}
 We say that the hom-commutativity constraint $d$ satisfies the \emph{(H1) property} if the diagram, as seen in Figure \ref{fig3}, commutes for all objects $U$, $V$ and $W$ of the category $\mathcal{C}$. Furthermore, we say that the hom-commutativity constraint $d$ satisfies the \emph{(H2) property} if the diagram, as seen in Figure \ref{fig4}, commutes for all objects 
$U$, $V$ and $W$ of the category $\mathscr{C}$.

\begin{center}
\begin{figure}[!ht]
\begin{center}
\begin{tikzpicture}
  \matrix (m) [matrix of math nodes,row sep=2.5em,column sep=2em,minimum width=2em] 
     {  &  \stackrel{\stackrel{\mbox{$\left(G\left(V\right)\otimes G\left(W\right)\right)\otimes FG\left(U\right)$}}{\mid\mid}}{G\left(V\otimes W\right)\otimes GF\left(U\right)} & \\                 
		F\left(U\right)\otimes\left(V\otimes W\right) & & FG\left(V\right)\otimes\left(G\left(W\right)\otimes G\left(U\right)\right) \\
											     & & \\
       \left(U\otimes V\right)\otimes F\left(W\right) & & FG\left(V\right)\otimes\left(G\left(W\right)\otimes G^{2}\left(U\right)\right) \\
														 & & \\
 \left(G\left(V\right)\otimes G\left(U\right)\right)\otimes F\left(W\right) & & FG\left(V\right)\otimes\left(G\left(U\right)\otimes W\right) \\};
  \path[-stealth]
	  (m-2-1) edge node [above] {\tiny{$d_{F\left(U\right),V\otimes W} ~\ ~\ ~\ ~\ ~\ $}} (m-1-2)
		(m-1-2) edge node [above] {\tiny{$ \hskip0.5in a_{G\left(V\right),G\left(W\right),G\left(U\right)}$}} (m-2-3)
		(m-2-3) edge node [left] {\tiny{$ ~\ ~\ \text{id}_{FG\left(V\right)}\otimes\left(\text{id}_{G\left(W\right)}\otimes\Phi_{G\left(U\right)}\right)$}} (m-4-3)
    (m-4-1) edge node [right] {\tiny{$a_{U,V,W} ~\ $}} (m-2-1)
		        edge node [right] {\tiny{$d_{U,V}\otimes\text{id}_{F\left(W\right)}~\ $}} (m-6-1)
		(m-6-1) edge node [above] {\tiny{$a_{G\left(V\right),G\left(U\right),W}$}} (m-6-3)
		(m-6-3) edge node [left] {\tiny{$\text{id}_{FG\left(V\right)}\otimes d_{G\left(U\right),W}$}} (m-4-3);
\end{tikzpicture}
\end{center}
\caption{The  (H1) property}
\label{fig3}
\end{figure}
\end{center}

\begin{center}
\begin{figure}[!ht]
\begin{center}
\begin{tikzpicture}
   \matrix (m) [matrix of math nodes,row sep=2.5em,column sep=2em,minimum width=2em] 
     {  &  \stackrel{\stackrel{\mbox{$FG\left(W\right)\otimes\left(G\left(U\right)\otimes G\left(V\right)\right)$}}{\mid\mid}}{GF\left(W\right)\otimes G\left(U\otimes V\right)} & \\                 
		\left(U\otimes V\right)\otimes F\left(W\right) & & \left(G\left(W\right)\otimes G\left(U\right)\right)\otimes FG\left(V\right) \\
		                       & & \\
       F\left(U\right)\otimes\left(V\otimes W\right) & & \left(G^{2}\left(W\right)\otimes G\left(U\right)\right)\otimes FG\left(V\right) \\
			                       & & \\
 F\left(U\right)\otimes \left(G\left(W\right)\otimes G\left(V\right)\right) & & \left(U\otimes G\left(W\right)\right)\otimes FG\left(V\right) \\};
  \path[-stealth]
	  (m-2-1) edge node [above] {\tiny{$d_{U\otimes V,F\left(W\right)} ~\ ~\ ~\ ~\ ~\ $}} (m-1-2)
		(m-1-2) edge node [above] {\tiny{$ \hskip0.5in a^{-1}_{G\left(W\right),G\left(U\right),G\left(V\right)}$}} (m-2-3)
		(m-2-3) edge node [left] {\tiny{$ ~\ ~\ \left(\Phi_{G\left(W\right)}\otimes \text{id}_{G\left(U\right)}\right)\otimes\text{id}_{FG\left(V\right)}$}} (m-4-3)
    (m-4-1) edge node [right] {\tiny{$a^{-1}_{U,V,W} ~\ $}} (m-2-1)
		        edge node [right] {\tiny{$\text{id}_{F\left(U\right)}\otimes d_{V,W} ~\ $}} (m-6-1)
		(m-6-1) edge node [above] {\tiny{$a^{-1}_{U,G\left(W\right),G\left(V\right)}$}} (m-6-3)
		(m-6-3) edge node [left] {\tiny{$d_{U,G\left(W\right)}\otimes\text{id}_{FG\left(V\right)}$}} (m-4-3);
\end{tikzpicture}
\end{center}
\caption{The  (H2) property}
\label{fig4}
\end{figure}
\end{center}
\label{def10}
\end{Def}

\begin{Def}
Let $\mathscr{C}=\left(\mathcal{C},\otimes,F, G,a,\Phi\right)$ be a hom-tensor category. A \emph{hom-braiding} $d$ in $\mathscr{C}$ is a hom-commutativity constraint with the following conditions:
\begin{enumerate}[label=(\roman*)]
 \item $d$ satisfies (H1) and (H2);
 \item For all objects $U$ and $V$ in the category $\mathcal{C}$,  $G\left(d_{U,V}\right)=d_{G\left(U\right),G\left(V\right)}$.
\end{enumerate}
A \emph{hom-braided category} $\left(\mathcal{C},\otimes,F,G,a,\Phi,d\right)$ is a hom-tensor category with hom-braiding $d$.
\label{def7}
\end{Def}
\begin{Rem} Note that $d$ is not required to be invertible and so a more appropriate name for the above structure 
would be hom-quasi-braided category. However, in order to simplify the terminology, we prefer to call it hom-braided category. 
\end{Rem}

We present a first class of examples of hom-braided categories.  
\begin{Prop}
Let $\left(\mathcal{C},\otimes, a, c\right)$ be a quasi-braided pre-tensor category. Then the hom-tensor category 
$\mathfrak{h}(\mathcal{C})$ constructed in Proposition \ref{mainexample} is a hom-braided category, with 
hom-braiding defined by 
\begin{eqnarray*}
&&d_{(M, \alpha _M), (N, \alpha _N)}:(M, \alpha _M)\otimes (N, \alpha _N)\rightarrow (N, \alpha _N)\otimes (M, \alpha _M), \\
&&d_{(M, \alpha _M), (N, \alpha _N)}:=(\alpha _N\otimes \alpha _M)\circ c_{M, N}=c_{M, N}\circ (\alpha _M\otimes \alpha _N), 
\end{eqnarray*}
for all objects $(M, \alpha _M)$, $(N, \alpha _N)$ in $\mathfrak{h}(\mathcal{C})$.
\end{Prop}
\begin{proof}
We only prove that the first hexagonal relation satisfied by $c_{-, -}$, namely 
\begin{eqnarray}
&&a_{V, W, U}\circ c_{U, V\otimes W}\circ a_{U, V, W}=(\text{id}_V\otimes c_{U, W})\circ a_{V, U, W}
\circ (c_{U, V}\otimes \text{id}_W), 
\label{H1c}
\end{eqnarray}
for all $U, V, W\in \textbf{Ob}(\mathcal{C})$, implies property (H1) for $d_{-, -}$ and leave the rest of the proof to the reader. 

Let $(U, \alpha _U), (V, \alpha _V), (W, \alpha _W)\in \textbf{Ob}(\mathfrak{h}(\mathcal{C}))$. We compute, 
by applying repeatedly the fact that $\otimes $ is a functor, the naturality of $c_{-, -}$ and the naturality of 
$a_{-, -, -}$:\\[2mm]
${\;\;\;\;\;}$$(\text{id}_{(V, \alpha _V)}\otimes d_{(U, \alpha _U), (W, \alpha _W)})\circ 
a_{(V, \alpha _V), (U, \alpha _U), (W, \alpha _W)}\circ (d_{(U, \alpha _U), (V, \alpha _V)}\otimes \text{id}_{(W, \alpha _W)})$
\begin{eqnarray*}
&=&(\text{id}_V\otimes ((\alpha _W\ot \alpha _U)\circ c_{U, W}))\circ a_{V, U, W}\circ ((c_{U, V}\circ (\alpha _U\otimes \alpha _V))
\otimes \text{id}_W)\\
&=&(\text{id}_V\ot (\text{id}_W\ot \alpha _U))\circ (\text{id}_V\ot (\alpha _W\ot \text{id}_U))\circ (\text{id}_V\ot c_{U, W})\circ 
a_{V, U, W}\\
&&\circ (c_{U, V}\ot \text{id}_W)\circ ((\alpha _U\ot \alpha _V)\ot \text{id}_W)\\
&=&(\text{id}_V\ot (\text{id}_W\ot \alpha _U))\circ 
(\text{id}_V\ot c_{U, W})\circ (\text{id}_V\ot (\text{id}_U\ot \alpha _W))
\circ a_{V, U, W}\\
&&\circ (c_{U, V}\ot \text{id}_W)\circ ((\alpha _U\ot \alpha _V)\ot \text{id}_W)\\
&=&(\text{id}_V\ot (\text{id}_W\ot \alpha _U))\circ 
(\text{id}_V\ot c_{U, W})
\circ a_{V, U, W}\circ ((\text{id}_V\ot \text{id}_U)\ot \alpha _W)\\
&&\circ (c_{U, V}\ot \text{id}_W)\circ ((\alpha _U\ot \alpha _V)\ot \text{id}_W)\\
&=&(\text{id}_V\ot (\text{id}_W\ot \alpha _U))\circ 
(\text{id}_V\ot c_{U, W})
\circ a_{V, U, W}
\circ (c_{U, V}\ot \text{id}_W)
\circ ((\alpha _U\ot \alpha _V)\ot \alpha _W)\\
&\stackrel{\left(\ref{H1c}\right)}{=}&(\text{id}_V\ot (\text{id}_W\ot \alpha _U))\circ 
a_{V, W, U}\circ c_{U, V\otimes W}\circ a_{U, V, W}\circ ((\alpha _U\ot \alpha _V)\ot \alpha _W)\\
&=&(\text{id}_V\ot (\text{id}_W\ot \alpha _U))\circ 
a_{V, W, U}\circ c_{U, V\otimes W}\circ (\alpha _U\ot (\alpha _V\ot \alpha _W))\circ a_{U, V, W}\\
&=&(\text{id}_{(V, \alpha _V)}\ot (\text{id}_{(W, \alpha _W)}\ot \Phi _{(U, \alpha _U)}))\circ 
a_{(V, \alpha _V), (W, \alpha _W), (U, \alpha _U)}\\
&&\circ d_{(U, \alpha _U), (V, \alpha _V)\otimes (W, \alpha _W)}
\circ a_{(U, \alpha _U), (V, \alpha _V), (W, \alpha _W)},
\end{eqnarray*}
finishing the proof.
\end{proof}

The following proposition gives the connection between hom-braided categories and quasitriangular hom-bialgebras.
\begin{Prop} \label{prop2}
Let $\left(H,m_{H},\Delta_{H},\alpha_{H}, \psi _H \right)$ be a hom-bialgebra and take $\mathscr{H}$ to be the hom-tensor category described in Proposition \ref{prop1}. Let $R=\sum_i s_i\otimes t_i\in H\otimes H$ and assume that 
$(\alpha _H\otimes \alpha _H)(R)=R=(\psi _H\otimes \psi _H)(R)$. 
Consider the two statements (A) and (B) below. The we have that (A) implies (B) and if $H$ is strongly nondegenerate then (B) implies (A). 
\begin{enumerate}[label=(\Alph*)]
\item $\left(H,m_{H},\Delta_{H},\alpha_{H}, \psi _H, R\right)$ is a quasitriangular hom-bialgebra. 
\item The category $\mathscr{H}=\left(H\textbf{-mod},\otimes,F,G,a,\Phi,c\right)$ is a hom-braided category with hom-braiding $c$ given as follows. For two objects  $\left(U,\alpha_{U}\right)$ and $\left(V,\alpha_{V}\right)$, we have $c_{U,V}:\left(U,\alpha_{U}\right)\otimes\left(V,\alpha_{V}\right)\rightarrow\left(V^{\alpha},\alpha_{V^{\alpha}}\right)\otimes\left(U^{\alpha},\alpha_{U^{\alpha}}\right)$
defined for all $u\in U$ and $v\in V$ by
\begin{eqnarray}
&&c_{U,V}\left(u\otimes v\right)=\left(\tau_{U^{\alpha},V^{\alpha}}\right)\left(R\left(u\otimes v\right)\right)=\sum_{i}\left(t_{i}\cdot_{\alpha}\widetilde{v}\right)\otimes\left(s_{i}\cdot_{\alpha}\widetilde{u}\right).
\label{eq27}
\end{eqnarray}
\end{enumerate}
\end{Prop}
\begin{proof}
$\left(A\right)\Rightarrow\left(B\right)$ Suppose that $\left(H,m_H, \Delta_{H}, \alpha _H, \psi _H, R\right)$ is a quasitriangular hom-bialgebra such that $\left(\alpha_{H}\otimes\alpha_{H}\right)\left(R\right)=R=(\psi _H\otimes \psi _H)(R)$; 
in particular we have 
\begin{eqnarray}
&& \sum_{i}\alpha_{H}\left(s_{i}\right)\otimes\alpha_{H}\left(t_{i}\right)=\sum_{i} s_{i}\otimes t_{i}.
 \label{eq41}
\end{eqnarray}

Let $\left(U,\alpha_{U}\right)$ and $\left(V,\alpha_{V}\right)$ be objects in $H\textbf{-mod}$ and let $u\in U$, $v\in V$. 
We claim that $c_{U,V}$ is a morphism in $\mathscr{H}$. First we check that 
 $\left(\alpha_{V^{\alpha}\otimes U^{\alpha}}\right)\left(c_{U,V}\right)\left(u\otimes v\right)=\left(c_{U,V}\right)\left(\alpha_{U\otimes V}\right)\left(u\otimes v\right)$: 
\begin{eqnarray*}
  \left(\alpha_{V^{\alpha}\otimes U^{\alpha}}\right)\left(c_{U,V}\right)\left(u\otimes v\right) &=& \left(\alpha_{V^{\alpha}}\otimes\alpha_{U^{\alpha}}\right)\left(c_{U,V}\right)\left(u\otimes v\right)\\
	&\stackrel{\left(\ref{eq27}\right)}{=}& \left(\alpha_{V^{\alpha}}\otimes\alpha_{U^{\alpha}}\right)\left(\sum_{i}\left(t_{i}\cdot_{\alpha}\widetilde{v}\right)
\otimes\left(s_{i}\cdot_{\alpha}\widetilde{u}\right)\right)\\
	&=& \sum_{i}\alpha_{V^{\alpha}}\left(\widetilde{\alpha_{H}\left(t_{i}\right)\cdot v}\right)\otimes\alpha_{U^{\alpha}}\left(\widetilde{\alpha_{H}\left(s_{i}\right)\cdot u}\right)\\
	&=& \sum_{i}\widetilde{\alpha_{V}\left(\alpha_{H}\left(t_{i}\right)\cdot v\right)}\otimes\widetilde{\alpha_{U}\left(\alpha_{H}\left(s_{i}\right)\cdot u\right)} \\
	&\stackrel{\left(\ref{eq8}\right)}{=}& \sum_{i}\widetilde{\alpha^{2}_{H}\left(t_{i}\right)\cdot\alpha_{V}\left(v\right)}
\otimes\widetilde{\alpha^{2}_{H}\left(s_{i}\right)\cdot\alpha_{U}\left(u\right)}, 
\end{eqnarray*}
\begin{eqnarray*}
\left(c_{U,V}\right)\left(\alpha_{U\otimes V}\right)\left(u\otimes v\right) &=& \left(c_{U,V}\right)\left(\alpha_{U}\otimes\alpha_{V}\right)\left(u\otimes v\right)
=c_{U,V}\left(\alpha_{U}\left(u\right)\otimes\alpha_{V}\left(v\right)\right)\\
&\stackrel{\left(\ref{eq27}\right)}{=}& \sum_{i}\left(t_{i}\cdot_{\alpha}\widetilde{\alpha_{V}\left(v\right)}\right)
\otimes\left(s_{i}\cdot_{\alpha}\widetilde{\alpha_{U}\left(u\right)}\right)\\
&\stackrel{\left(\ref{eq41}\right)}{=}& \sum_{i}\left(\alpha_{H}\left(t_{i}\right)\cdot_{\alpha}\widetilde{\alpha_{V}\left(v\right)}\right)
\otimes\left(\alpha_{H}\left(s_{i}\right)\cdot_{\alpha}\widetilde{\alpha_{U}\left(u\right)}\right)\\
&=& \sum_{i}\widetilde{\alpha^{2}_{H}\left(t_{i}\right)\cdot\alpha_{V}\left(v\right)}
\otimes\widetilde{\alpha^{2}_{H}\left(s_{i}\right)\cdot\alpha_{U}\left(u\right)}, \;\;\;q.e.d.
\end{eqnarray*}

Let $h\in H$, $u\in U$ and $v\in V$. We next check that 
 $\left(c_{U,V}\right)\left(h\cdot\left(u\otimes v\right)\right)=h\cdot\left(c_{U,V}\right)\left(u\otimes v\right)$: 
\begin{eqnarray*}
\left(c_{U,V}\right)\left(h\cdot\left(u\otimes v\right)\right) &=& \left(c_{U,V}\right)\left(\sum \left(h_{\left(1\right)}\cdot u\right)\otimes\left(h_{\left(2\right)}\cdot v\right)\right)\\
&\stackrel{\left(\ref{eq27}\right)}{=}& \left(\tau_{U^{\alpha},V^{\alpha}}\right)\left(\sum_{i}\left(s_{i}\cdot_{\alpha}\left(\widetilde{h_{\left(1\right)}\cdot u}\right)\right)\otimes\left(t_{i}\cdot_{\alpha}\left(\widetilde{h_{\left(2\right)}\cdot v}\right)\right)\right)\\
&=& \left(\tau_{U^{\alpha},V^{\alpha}}\right)\left(\sum_{i}\left(\widetilde{\alpha_{H}\left(s_{i}\right)\cdot\left(h_{\left(1\right)}\cdot u\right)}\right)\otimes\left(\widetilde{\alpha_{H}\left(t_{i}\right)\cdot\left(h_{\left(2\right)}\cdot v\right)}\right)\right)\\
&\stackrel{\left(\ref{eq9}\right)}{=}& \left(\tau_{U^{\alpha},V^{\alpha}}\right)\left(\sum_{i}\left(\widetilde{\left(s_{i}h_{\left(1\right)}\right)\cdot\alpha_{U}\left(u\right)}\right)\otimes\left(\widetilde{\left(t_{i}h_{\left(2\right)}\right)\cdot\alpha_{V}\left(v\right)}\right)\right)\\
&\stackrel{\left(\ref{eq38}\right)}{=}& \left(\tau_{U^{\alpha},V^{\alpha}}\right)\left(\sum_{i}\left(\widetilde{\left(h_{\left(2\right)}s_{i}\right)\cdot\alpha_{U}\left(u\right)}\right)\otimes\left(\widetilde{\left(h_{\left(1\right)}t_{i}\right)\cdot\alpha_{V}\left(v\right)}\right)\right)\\
&\stackrel{\left(\ref{eq9}\right)}{=}& \sum_{i}\left(\widetilde{\alpha_{H}\left(h_{\left(1\right)}\right)\cdot\left(t_{i}\cdot v\right)}\right)\otimes\left(\widetilde{\alpha_{H}\left(h_{\left(2\right)}\right)\cdot\left(s_{i}\cdot u\right)}\right)\\
&=& \sum_{i}\left(h_{\left(1\right)}\cdot_{\alpha}\left(\widetilde{t_{i}\cdot v}\right)\right)\otimes\left(h_{\left(2\right)}\cdot_{\alpha}\left(\widetilde{s_{i}\cdot u}\right)\right)\\
&\stackrel{\left(\ref{eq41}\right)}{=}& h\cdot \left(\sum_{i}\left(\widetilde{\alpha_{H}\left(t_{i}\right)\cdot v}\right)\otimes\left(\widetilde{\alpha_{H}\left(s_{i}\right)\cdot u}\right)\right)\\
&=& h\cdot \left(\sum_{i}\left(t_{i}\cdot_{\alpha}\widetilde{v}\right)\otimes\left(s_{i}\cdot_{\alpha}\widetilde{u}\right)\right)
=h\cdot c_{U,V}\left(u\otimes v\right), \;\;\;q.e.d.
\end{eqnarray*}
Therefore, $c_{U,V}$ is a morphism in $\mathscr{H}$. A similar computation shows the naturality of 
$c$. 

Next we will confirm that $c_{U,V}$ as it is defined satisfies the $\left(\text{H}1\right)$ property. That is, we need to show that
\begin{equation}
 (\text{id}_{V^{\psi\alpha}}\otimes(\text{id}_{W^{\alpha}}\otimes\Phi_{U^{\alpha}}))\circ a_{V^{\alpha},W^{\alpha},U^{\alpha}} \circ c_{U^{\psi},V\otimes W}\circ a_{U,V,W}= 
 (\text{id}_{V^{\psi\alpha}}\otimes c_{U^{\alpha},W})\circ a_{V^{\alpha},U^{\alpha},W}\circ(c_{U,V}\otimes\text{id}_{W^{\psi}})
\label{eq45}
\end{equation}
for all objects $\left(U,\alpha_{U}\right)$, $\left(V,\alpha_{V}\right)$ and $\left(W,\alpha_{W}\right)$ in $H\textbf{-mod}$.

Let $u\in U$, $v\in V$ and $w\in W$. Computing the left hand side of (\ref{eq45}) we have that\\[2mm]
${\;\;\;}$
$\left(\text{id}_{V^{\psi\alpha}}\otimes\left(\text{id}_{W^{\alpha}}\otimes\Phi_{U^{\alpha}}\right)\right)\circ a_{V^{\alpha},W^{\alpha},U^{\alpha}} \circ c_{U^{\psi},V\otimes W}\circ a_{U,V,W}\left(\left(u\otimes v\right)\otimes\overline{w}\right)$
\begin{eqnarray*}
&=&\left(\left(\text{id}_{V^{\psi\alpha}}\otimes\left(\text{id}_{W^{\alpha}}\otimes\Phi_{U^{\alpha}}\right)\right)\circ a_{V^{\alpha},W^{\alpha},U^{\alpha}} \circ c_{U^{\psi},V\otimes W}\right)\left(\overline{u}\otimes\left(v\otimes 
w\right)\right)\\
&  \stackrel{\left(\ref{eq27}\right)}{=}&\left(\left(\text{id}_{V^{\psi\alpha}}\otimes\left(\text{id}_{W^{\alpha}}\otimes\Phi_{U^{\alpha}}\right)\right)\circ a_{V^{\alpha},W^{\alpha},U^{\alpha}}\right)\left(\sum_{i}\left(t_{i}\cdot_{\alpha}\left(\widetilde{v\otimes w}\right)\right)\otimes\left(s_{i}\cdot_{\alpha\psi}\widetilde{\overline{u}}\right)\right)\\
& =&\left(\left(\text{id}_{V^{\psi\alpha}}\otimes\left(\text{id}_{W^{\alpha}}\otimes\Phi_{U^{\alpha}}\right)\right)\circ a_{V^{\alpha},W^{\alpha},U^{\alpha}}\right)\left(\sum_{i}\left(\widetilde{\alpha_{H}\left(t_{i}\right)\cdot\left(v\otimes w\right)}\right)\otimes\left(\widetilde{\overline{\alpha_{H}\psi_H\left(s_{i}\right)\cdot u}}\right)\right)\\
&=&\left(\left(\text{id}_{V^{\psi\alpha}}\otimes\left(\text{id}_{W^{\alpha}}\otimes\Phi_{U^{\alpha}}\right)\right)\circ a_{V^{\alpha},W^{\alpha},U^{\alpha}}\right)\\
&&\left(\left(\sum _{i}\widetilde{\left(\alpha_{H}\left(t_{i}\right)\right)_{\left(1\right)}
\cdot v}\otimes \widetilde{\left(\alpha_{H}\left(t_{i}\right)\right)_{\left(2\right)}\cdot w}\right)
\otimes\left(\overline{\widetilde{\psi_H\alpha_{H}\left(s_{i}\right)\cdot u}}\right)\right)\\
& \stackrel{\left(\ref{eq41}\right)}{=}&\left(\text{id}_{V^{\psi\alpha}}\otimes\left(\text{id}_{W^{\alpha}}\otimes\Phi_{U^{\alpha}}\right)\right)
\left(\sum_{i}\overline{\widetilde{\left(t_{i}\right)_{\left(1\right)}\cdot v}}\otimes\left(\widetilde{\left(t_{i}\right)_{\left(2\right)}\cdot w}\otimes\widetilde{\psi_H\left(s_{i}\right)\cdot u}\right)\right)\\
& \stackrel{\left(\ref{eq8}\right)}{=}&\sum_{i}\overline{\widetilde{\left(t_{i}\right)_{\left(1\right)}\cdot v}}\otimes\left(\widetilde{\left(t_{i}\right)_{\left(2\right)}\cdot w}\otimes\widetilde{\widetilde{\alpha_H\psi_H\left(s_{i}\right)\cdot \alpha_U(u)}}\right).
\end{eqnarray*}
Computing the right hand side of (\ref{eq45}) we have that\\[2mm]
${\;\;\;}$
$\left(\text{id}_{V^{\psi\alpha}}\otimes c_{U^{\alpha},W}\right)\circ a_{V^{\alpha},U^{\alpha},W}\circ\left(c_{U,V}\otimes\text{id}_{W^{\psi}}\right)\left(\left(u\otimes v\right)\otimes\overline{w}\right)$
\begin{eqnarray*}
& \stackrel{\left(\ref{eq27}\right)}{=}&\left(\left(\text{id}_{V^{\psi\alpha}}\otimes c_{U^{\alpha},W}\right)\circ a_{V^{\alpha},U^{\alpha},W}\right)\left(\left(\sum_{i}t_{i}\cdot_{\alpha}\widetilde{v}\otimes s_{i}\cdot_{\alpha}\widetilde{u}\right)\otimes\bar{w}\right)\\
&=&\left(\left(\text{id}_{V^{\psi\alpha}}\otimes c_{U^{\alpha},W}\right)\circ a_{V^{\alpha},U^{\alpha},W}\right)\left(\left(\sum_{i}\widetilde{\alpha_{H}\left(t_{i}\right)\cdot v}\otimes\widetilde{\alpha_{H}\left(s_{i}\right)\cdot u}\right)\otimes\bar{w}\right)\\
& =&\left(\text{id}_{V^{\psi\alpha}}\otimes c_{U^{\alpha},W}\right)\left(\sum_{i}\overline{\widetilde{\alpha_{H}\left(t_{i}\right)\cdot v}}\otimes\left(\widetilde{\alpha_{H}\left(s_{i}\right)\cdot u}\otimes w\right)\right)\\
& \stackrel{\left(\ref{eq27}\right)}{=}&\sum_{i,j}\overline{\widetilde{\alpha_{H}\left(t_{i}\right)\cdot v}}\otimes\left(t_{j}\cdot_{\alpha}\widetilde{w}\otimes s_{j}\cdot_{\alpha^2}\left(\widetilde{\widetilde{\alpha_{H}\left(s_{i}\right)\cdot u}}\right)\right)\\
&  =&\sum_{i,j}\overline{\widetilde{\alpha_{H}\left(t_{i}\right)\cdot v}}\otimes\left(\widetilde{\alpha_{H}\left(t_{j}\right)\cdot w}\otimes\widetilde{\widetilde{\alpha^{2}_{H}\left(s_{j}\right)\cdot\left(\alpha_{H}\left(s_{i}\right)\cdot u\right)}}\right)\\
&  \stackrel{\left(\ref{eq9}\right)}{=}&\sum_{i,j}\overline{\widetilde{\alpha_{H}\left(t_{i}\right)\cdot v}}\otimes\left(\widetilde{\alpha_{H}\left(t_{j}\right)\cdot w}\otimes\widetilde{\widetilde{(\alpha_{H}\left(s_{j}\right)\cdot\alpha_{H}\left(s_{i}\right))\cdot \alpha_U(u)}}\right)\\
&  =&\sum_{i,j}\overline{\widetilde{\alpha_{H}\left(t_{i}\right)\cdot v}}\otimes\left(\widetilde{\alpha_{H}\left(t_{j}\right)\cdot w}\otimes\widetilde{\widetilde{(\alpha_{H}\left(s_{j}s_{i}\right))\cdot \alpha_U(u)}}\right)\\
&  \stackrel{\left(\ref{eq41}\right)}{=}&\sum_{i,j}\overline{\widetilde{t_{i}\cdot v}}\otimes\left(\widetilde{t_{j}\cdot w}\otimes\widetilde{\widetilde{(s_{j}s_{i})\cdot \alpha_U(u)}}\right).
\end{eqnarray*}
Now since $\left(\alpha _H\psi _H\otimes\Delta\right)\left(R\right)=\sum_{i,j}s_{i}s_{j}\otimes t_{j}\otimes t_{i}$ by 
Remark \ref{remQT}, 
the left hand side and the right hand side of (\ref{eq45}) agree. So $c_{U,V}$ has the (H1) property.

Next we will confirm that $c_{U,V}$ satisfies the $\left(\text{H}2\right)$ property. That is, we need to show that
\begin{equation}
 ((\Phi_{W^{\alpha}}\otimes \text{id}_{U^{\alpha}})\otimes\text{id}_{V^{\psi\alpha}})\circ a^{-1}_{W^{\alpha},U^{\alpha},V^{\alpha}}\circ c_{U\otimes V,W^{\psi}}\circ a^{-1}_{U,V,W}=(c_{U,W^{\alpha}}\otimes\text{id}_{V^{\psi\alpha}})\circ a^{-1}_{U,W^{\alpha},V^{\alpha}}\circ(\text{id}_{U^{\psi}}\otimes c_{V,W})
	\label{eq50}
\end{equation}
for all objects $\left(U,\alpha_{U}\right)$, $\left(V,\alpha_{V}\right)$ and $\left(W,\alpha_{W}\right)$ in $H\textbf{-mod}$.

Let $u\in U$, $v\in V$ and $w\in W$. For the left hand side of (\ref{eq50}) we have that\\[2mm]
${\;\;\;}$
$\left(\left(\left(\Phi_{W^{\alpha}}\otimes \text{id}_{U^{\alpha}}\right)\otimes\text{id}_{V^{\psi\alpha}}\right)\circ a^{-1}_{W^{\alpha},U^{\alpha},V^{\alpha}}\circ c_{U\otimes V,W^{\psi}}\circ a^{-1}_{U,V,W}\right)\left(\overline{u}\otimes\left(v\otimes w\right)\right)$
\begin{eqnarray*}
& =&\left(\left(\left(\Phi_{W^{\alpha}}\otimes \text{id}_{U^{\alpha}}\right)\otimes\text{id}_{V^{\psi\alpha}}\right)\circ a^{-1}_{W^{\alpha},U^{\alpha},V^{\alpha}}\circ c_{U\otimes V,W^{\psi}}\right)\left(\left(u\otimes v\right)\otimes\overline{w}\right)\\ 
&  \stackrel{\left(\ref{eq27}\right)}{=}&\left(\left(\left(\Phi_{W^{\alpha}}\otimes \text{id}_{U^{\alpha}}\right)\otimes\text{id}_{V^{\psi\alpha}}\right)\circ a^{-1}_{W^{\alpha},U^{\alpha},V^{\alpha}}\right)\left(\sum_{i}t_{i}\cdot_{\alpha\psi}\widetilde{\overline{w}}\otimes s_{i}\cdot_{\alpha}\left(\widetilde{u\otimes v}\right)\right)\\
& =&\left(\left(\left(\Phi_{W^{\alpha}}\otimes \text{id}_{U^{\alpha}}\right)\otimes\text{id}_{V^{\psi\alpha}}\right)\circ a^{-1}_{W^{\alpha},U^{\alpha},V^{\alpha}}\right)\\
&&\left(\sum_{i}\overline{\widetilde{\psi_H\alpha_{H}\left(t_{i}\right)\cdot w}}\otimes\left(\widetilde{\left(\alpha_{H}\left(s_{i}\right)\right)_{\left(1\right)}\cdot u}\otimes\widetilde{\left(\alpha_{H}\left(s_{i}\right)\right)_{\left(2\right)}\cdot v}\right)\right)\\
&  \stackrel{\left(\ref{eq41}\right)}{=}&\left(\left(\Phi_{W^{\alpha}}\otimes \text{id}_{U^{\alpha}}\right)\otimes\text{id}_{V^{\psi\alpha}}\right)\left(\left(\sum_{i}\widetilde{\psi_H\left(t_{i}\right)\cdot w}\otimes\widetilde{\left(s_{i}\right)_{\left(1\right)}\cdot u}\right)\otimes\overline{\widetilde{\left(s_{i}\right)_{\left(2\right)}\cdot v}}\right)\\
&  \stackrel{\left(\ref{eq8}\right)}{=}&\left(\sum_{i}\widetilde{\widetilde{\alpha_H\psi_H\left(t_{i}\right)\cdot \alpha_W(w)}}\otimes\widetilde{\left(s_{i}\right)_{\left(1\right)}\cdot u}\right)\otimes\overline{\widetilde{\left(s_{i}\right)_{\left(2\right)}\cdot v}}.
\end{eqnarray*}
For the right hand side of (\ref{eq50}) we have that\\[2mm]
${\;\;\;}$
$\left(\left(c_{U,W^{\alpha}}\otimes\text{id}_{V^{\psi\alpha}}\right)\circ a^{-1}_{U,W^{\alpha},V^{\alpha}}\circ\left(\text{id}_{U^{\psi}}
\otimes c_{V,W}\right)\right)\left(\overline{u}\otimes\left(v\otimes w\right)\right)$
\begin{eqnarray*}
&   \stackrel{\left(\ref{eq27}\right)}{=}&\left(\left(c_{U,W^{\alpha}}\otimes\text{id}_{V^{\psi\alpha}}\right)\circ a^{-1}_{U,W^{\alpha},V^{\alpha}}\right)\left(\overline{u}\otimes\left(\sum_{i}t_{i}\cdot_{\alpha}\widetilde{w}\otimes s_{i}\cdot_{\alpha}\widetilde{v}\right)\right)\\
&=&\left(c_{U,W^{\alpha}}\otimes\text{id}_{V^{\psi\alpha}}\right)\left(\left(u\otimes\sum_{i}\widetilde{\alpha_{H}\left(t_{i}\right)\cdot w}\right)\otimes\overline{\widetilde{\alpha_{H}\left(s_{i}\right)\cdot v}}\right)\\
&  \stackrel{\left(\ref{eq27}\right)}{=}&\sum_{i,j}\left(t_{j}\cdot_{\alpha^{2}}\left(\widetilde{\widetilde{\alpha_{H}\left(t_{i}\right)\cdot w}}\right)\otimes s_{j}\cdot_{\alpha}\widetilde{u}\right)\otimes\overline{\widetilde{\alpha_{H}\left(s_{i}\right)\cdot v}}\\
&  \stackrel{}{=}&\sum_{i,j}\left(\widetilde{\widetilde{\alpha^2_H(t_j)\cdot(\alpha_{H}\left(t_{i}\right)\cdot w)}}\otimes \widetilde{\alpha_H(s_j)\cdot u}\right)\otimes\overline{\widetilde{\alpha_{H}\left(s_{i}\right)\cdot v}}\\
&  \stackrel{\left(\ref{eq9}\right)}{=}&\sum_{i,j}\left(\widetilde{\widetilde{(\alpha_H(t_j)\alpha_{H}\left(t_{i}\right))\cdot \alpha_W(w)}}\otimes \widetilde{\alpha_H(s_j)\cdot u}\right)\otimes\overline{\widetilde{\alpha_{H}\left(s_{i}\right)\cdot v}}\\
&  \stackrel{\left(\ref{eq41}\right)}{=}&\sum_{i,j}\left(\widetilde{\widetilde{(t_jt_{i})\cdot \alpha_W(w)}}\otimes \widetilde{s_j\cdot u}\right)\otimes\overline{\widetilde{s_{i}\cdot v}}.
\end{eqnarray*}
Since $\left(\Delta\otimes \alpha _H\psi _H\right)\left(R\right)=\sum_{i,j}s_{i} \otimes s_{j}\otimes t_{i}t_{j}$ 
by Remark \ref{remQT},  $c_{U,V}$ satisfies the (H2) condition.

We show that 
$c_{U^{\alpha},V^{\alpha}}=G\left(c_{U,V}\right)$, 
for all objects $\left(U,\alpha_{U}\right)$, $\left(V,\alpha_{V}\right)$ in $H\textbf{-mod}$ and 
$u\in U$, $v\in V$:
\begin{eqnarray*}
&&c_{U^{\alpha},V^{\alpha}}\left(\widetilde{u}\otimes\widetilde{v}\right)= \sum_{i}\left(t_{i}\cdot_{\alpha^{2}}\widetilde{\widetilde{v}}\right)
\otimes\left(s_{i}\cdot_{\alpha^{2}}\widetilde{\widetilde{u}}\right)
=\sum_{i}\left(\widetilde{\widetilde{\alpha^{2}_{H}\left(t_{i}\right)\cdot v}}\right)\otimes\left(\widetilde{\widetilde{\alpha^{2}_{H}\left(s_{i}\right)\cdot u}}\right), \\
&&G\left(c_{U,V}\right)\left(\widetilde{u}\otimes\widetilde{v}\right) = \widetilde{c_{U,V}\left(u\otimes v\right)}
= \sum_{i}\widetilde{\left(t_{i}\cdot_{\alpha}\widetilde{v}\right)\otimes\left(s_{i}\cdot_{\alpha}\widetilde{u}\right)}
= \sum_{i}\left(\widetilde{\widetilde{\alpha_{H}\left(t_{i}\right)\cdot v}}\right)\otimes\left(\widetilde{\widetilde{\alpha_{H}\left(s_{i}\right)\cdot u}}\right).
\end{eqnarray*}

Since $(\alpha_H\otimes \alpha_H)(R)=R$ we get $c_{U^{\alpha},V^{\alpha}}=G\left(c_{U,V}\right)$ and so 
$c_{U,V}$ is a hom-braiding in $\mathscr{H}$. Therefore $\mathscr{H}$ is a hom-braided category.


$\left(B\right)\Rightarrow\left(A\right)$ Suppose that $\mathscr{H}=\left(H\textbf{-mod},\otimes,F, G,a, \Phi ,c\right)$ is a hom-braided category, $\left(H,m_{H},\alpha_{H}\right)$ is a  strongly nondegenerate hom-associative algebra and let $R\in H\otimes H$ be given as $R=\sum_{i}s_{i}\otimes t_{i}$ such that 
$\left(\alpha_{H}\otimes\alpha_{H}\right)\left(R\right)=R=(\psi _H\otimes \psi _H)(R).$ 
We will show that conditions (\ref{eq29}), (\ref{eq30}) and (\ref{eq31}) from Definition \ref{def11} are satisfied for 
$H=\left(H,m_{H},\Delta_{H},\alpha_{H}, \psi _H, R\right)$.

First we will show that $H$ satisfies condition (\ref{eq29}). Let $\left(U,\alpha_{U}\right)$ and  $\left(V,\alpha_{V}\right)$ be objects in $\mathscr{H}$. Since $\mathscr{H}$ is a hom-braided category it means that the hom-braiding $c_{U,V}$ is a morphism of left $H$-modules. That is, for $h\in H$, $u\in U$, and $v\in V$, we have 
$c_{U,V}\left(h\cdot\left(u\otimes v\right)\right)=h\cdot c_{U,V}\left(\left(u\otimes v\right)\right)$. 
Just like above, for the left hand side of this equality we have
$c_{U,V}\left(h\cdot\left(u\otimes v\right)\right) 
=\sum_{i}\widetilde{\left(t_{i}h_{\left(2\right)}\right)\cdot\alpha_{V}\left(v\right)}
\otimes\widetilde{\left(s_{i}h_{\left(1\right)}\right)\cdot\alpha_{U}\left(u\right)}$, 
and for the right hand side we have 
$h\cdot c_{U,V}\left(u\otimes v\right) 
=\sum_{i}\widetilde{\left(h_{\left(1\right)}t_{i}\right)\cdot\alpha_{V}\left(v\right)}
\otimes\widetilde{\left(h_{\left(2\right)}s_{i}\right)\cdot\alpha_{U}\left(u\right)}$, 
and so we must have
\begin{equation*}
\sum_{i}\left(t_{i}h_{\left(2\right)}\right)\cdot\alpha_{V}\left(v\right)
\otimes\left(s_{i}h_{\left(1\right)}\right)\cdot\alpha_{U}\left(u\right)=
\sum_{i}\left(h_{\left(1\right)}t_{i}\right)\cdot\alpha_{V}\left(v\right)
\otimes\left(h_{\left(2\right)}s_{i}\right)\cdot\alpha_{U}\left(u\right).
\end{equation*}
Since $H$ is strongly nondegenerate and by using Lemma \ref{lemmanondeg}, this equation implies that
$\sum_{i}t_{i}h_{\left(2\right)}\otimes s_{i}h_{\left(1\right)}=\sum_{i}h_{\left(1\right)}t_{i}\otimes h_{\left(2\right)}s_{i}$, 
or equivalently 
$R\Delta\left(h\right)=\Delta^{\text{cop}}\left(h\right)R$, for every $h\in H$. Thus, $R$ satisfies (\ref{eq29}).

Next we will show that $R$ satisfies conditions (\ref{eq30}) and (\ref{eq31}).  One can show that\\[2mm]
${\;\;\;\;\;\;\;\;\;\;\;\;}$
$\left(\left(\left(\Phi_{W^{\alpha}}\otimes \text{id}_{U^{\alpha}}\right)\otimes\text{id}_{V^{\psi\alpha}}\right)\circ a^{-1}_{W^{\alpha},U^{\alpha},V^{\alpha}}\circ c_{U\otimes V,W^{\psi}}\circ a^{-1}_{U,V,W}\right)\left(\overline{u}\otimes\left(v\otimes w\right)\right)$
\begin{eqnarray*}
& \stackrel{}{=}&\left(\sum_{i}\widetilde{\widetilde{\alpha_H\psi_H\left(t_{i}\right)\cdot \alpha_W(w)}}\otimes\widetilde{\left(s_{i}\right)_{\left(1\right)}\cdot u}\right)\otimes\overline{\widetilde{\left(s_{i}\right)_{\left(2\right)}\cdot v}},
\end{eqnarray*}
${\;\;\;\;\;\;\;\;\;\;\;\;}$
$\left(\left(c_{U,W^{\alpha}}\otimes\text{id}_{V^{\psi\alpha}}\right)\circ a^{-1}_{U,W^{\alpha},V^{\alpha}}\circ\left(\text{id}_{U^{\psi}}
\otimes c_{V,W}\right)\right)\left(\overline{u}\otimes\left(v\otimes w\right)\right)$
\begin{eqnarray*}
&  \stackrel{}{=}&\sum_{i,j}\left(\widetilde{\widetilde{(t_jt_{i})\cdot \alpha_W(w)}}\otimes \widetilde{s_j\cdot u}\right)\otimes\overline{\widetilde{s_{i}\cdot v}}.
\end{eqnarray*}

Since  $c_{U,V}$ satisfies the (H2) property and $H$ is strongly nondegenerate, by using again 
Lemma \ref{lemmanondeg} we obtain $\left(\Delta _H\otimes \alpha _H\psi _H\right)\left(R\right)=\sum_{i,j}s_{i} \otimes s_{j}\otimes t_{i}t_{j}$. 
Similarly one can show that $\left(\alpha _H\psi _H\otimes\Delta _H\right)\left(R\right)
=\sum_{i,j}s_{i}s_{j}\otimes t_{j}\otimes t_{i}$. 
Therefore, by Remark \ref{remQT}, $(H,R)$ is a quasitriangular hom-bialgebra.
\end{proof}

\section{The Hom-Yang-Baxter Equation}\label{sec6}
 D. Yau introduced the \emph{hom-Yang-Baxter equation} in \cite{YA:HYB} from which we recall the following definition.

\begin{Def}\label{def35}
Let $V$ be a $\h$-vector space, $\alpha_{V}:V\rightarrow V$ be a $\h$-linear map and $d_{V}:V\otimes V\rightarrow V\otimes V$ be a $\h$-linear map such that $\left(\alpha_{V}\otimes\alpha_{V}\right)\circ d_{V}=d_{V}\circ\left(\alpha_{V}\otimes\alpha_{V}\right)$. We say that $d_{V}$ is a solution to the \emph{hom-Yang-Baxter equation with respect to $\alpha_V$} if it satisfies the following condition:
\begin{eqnarray}
&&\left(d_{V}\otimes\alpha_{V}\right)\circ\left(\alpha_{V}\otimes d_{V}\right)\circ\left(d_{V}\otimes\alpha_{V}\right)=\left(\alpha_{V}\otimes d_{V}\right)\circ\left(d_{V}\otimes\alpha_{V}\right)\circ\left(\alpha_{V}\otimes d_{V}\right). \label{eq145}
\end{eqnarray}
\end{Def}
The goal of this section is to describe two categorical versions of Definition \ref{def35}. 
The following lemma will be useful in proving some of these results. 

\begin{Lem} Assume that $d$ is a hom-commutativity constraint as in Definition \ref{def8}.   Then 
the commutative diagram in Figure \ref{fig97} is equivalent to the commutative diagram  in Figure \ref{fig4}. We will call the equation from Figure \ref{fig97} the $\left(\text{H}^{\prime}2\right)$ property of the hom-commutativity constraint $d$.
\label{lem8}
\end{Lem}

\begin{center}
\begin{figure}[!ht]
\begin{center}
\begin{tikzpicture}
   \matrix (m) [matrix of math nodes,row sep=2.5em,column sep=2em,minimum width=2em] 
     {  &  \stackrel{\stackrel{\mbox{$FG\left(W\right)\otimes\left(G\left(U\right)\otimes G\left(V\right)\right)$}}{\mid\mid}}{GF\left(W\right)\otimes G\left(U\otimes V\right)} & \\                 
		\left(U\otimes V\right)\otimes F\left(W\right) & & FG^{2}\left(W\right)\otimes\left(G\left(U\right)\otimes G\left(V\right)\right) \\
		                       & & \\
       F\left(U\right)\otimes\left(V\otimes W\right) & & \left(G^{2}\left(W\right)\otimes G\left(U\right)\right)\otimes FG\left(V\right) \\
			                       & & \\
 F\left(U\right)\otimes \left(G\left(W\right)\otimes G\left(V\right)\right) & & \left(U\otimes G\left(W\right)\right)\otimes FG\left(V\right) \\};
  \path[-stealth]
	  (m-2-1) edge node [above] {\tiny{$d_{U\otimes V,F\left(W\right)} ~\ ~\ ~\ ~\ ~\ $}} (m-1-2)
		(m-1-2) edge[line width=1.5pt] node [above] {\tiny{$ \hskip1in \Phi_{GF\left(W\right)}\otimes \left(\text{id}_{G\left(U\right)}\otimes\text{id}_{G\left(V\right)}\right)$}} (m-2-3)
		(m-2-3) edge[line width=1.5pt] node [left] {\tiny{$ ~\ ~\ a^{-1}_{G^{2}\left(W\right),G\left(U\right),G\left(V\right)}$}} (m-4-3)
    (m-4-1) edge node [right] {\tiny{$a^{-1}_{U,V,W} ~\ $}} (m-2-1)
		        edge node [right] {\tiny{$\text{id}_{F\left(U\right)}\otimes d_{V,W} ~\ $}} (m-6-1)
		(m-6-1) edge node [above] {\tiny{$a^{-1}_{U,G\left(W\right),G\left(V\right)}$}} (m-6-3)
		(m-6-3) edge node [left] {\tiny{$d_{U,G\left(W\right)}\otimes\text{id}_{FG\left(V\right)}$}} (m-4-3);
\end{tikzpicture}
\end{center}
\caption{The  $\left(\text{H}^{\prime}2\right)$ property}
\label{fig97}
\end{figure}
\end{center}

\begin{proof}  
First observe that the commutative diagrams for the (H2) property and the $\left(\text{H}^{\prime}2\right)$ property are identical except on two of their edges and the object between those edges. These exceptions are indicated by bold edges in Figure \ref{fig97}. So it suffices to show that the diagram in Figure \ref{fig109} commutes. 
The hom-associativity constraint $a$ is a natural isomorphism, so  $a^{-1}$ is also natural and the diagram in Figure \ref{fig109} is a particular case of the naturality of $a^{-1}$ plus the fact that $F(\Phi_{G(W)})=\Phi_{FG(W)}$. Thus the diagram in Figure \ref{fig109} commutes and we have the equivalence between the properties (H2) and $\left(\text{H}^{\prime}2\right)$. 
\end{proof}

\begin{center}
\begin{figure}[!ht]
\begin{center}
\begin{tikzpicture}[scale=0.9, every node/.style={scale=0.9}]
  \matrix (m) [matrix of math nodes,row sep=3em,column sep=4em,minimum width=2em] 
     {FG\left(W\right)\otimes\left(G\left(U\right)\otimes G\left(V\right)\right) & & \left(G\left(W\right)\otimes G\left(U\right)\right)\otimes FG\left(V\right)\\
      FG^{2}\left(W\right)\otimes\left(G\left(U\right)\otimes G\left(V\right)\right) & & \left(G^{2}\left(W\right)\otimes G\left(U\right)\right)\otimes FG\left(V\right)\\};
  \path[-stealth]
	  (m-1-1) edge node [above] {\tiny{$a^{-1}_{G\left(W\right),G\left(U\right),G\left(V\right)}$}} (m-1-3)
            edge node [left] {\tiny{$\Phi_{GF\left(W\right)}\otimes \left(\text{id}_{G\left(U\right)}\otimes\text{id}_{G\left(V\right)}\right)$}} (m-2-1)
    (m-1-3) edge node [right] {\tiny{$\left(\Phi_{G\left(W\right)}\otimes \text{id}_{G\left(U\right)}\right)\otimes\text{id}_{FG\left(V\right)}$}} (m-2-3)
		(m-2-1) edge node [above] {\tiny{$a^{-1}_{G^{2}\left(W\right),G\left(U\right),G\left(V\right)}$}} (m-2-3);
\end{tikzpicture}
\end{center}
\caption{Naturality of $a^{-1}$ and $F(\Phi_{G(W)})=\Phi_{FG(W)}$}
\label{fig109}
\end{figure}
\end{center}

\begin{Def} We say that a hom-commutativity constraint  $d$ has the \emph{hom-Yang-Baxter property} if 
$d$ satisfies the equation in Figure \ref{fig17}.

\begin{center}
\begin{figure}[!ht]
\begin{center}
\begin{tikzpicture}[scale=0.9, every node/.style={scale=0.9}]
  \matrix (m) [matrix of math nodes,row sep=3em,column sep=4em,minimum width=2em] 
     {                  & \left(U\otimes V\right)\otimes F\left(W\right) & \\ 
		  \left(G\left(V\right)\otimes G\left(U\right)\right)\otimes F\left(W\right) & & F\left(U\right)\otimes\left(V\otimes W\right)\\
      FG\left(V\right)\otimes\left(G\left(U\right)\otimes W\right)  & & F\left(U\right)\otimes\left(G\left(W\right)\otimes G\left(V\right)\right)\\
			 FG\left(V\right)\otimes\left(G\left(W\right)\otimes G^{2}\left(U\right)\right) & & \left(U\otimes G\left(W\right)\right)\otimes FG\left(V\right)\\
			\left(G\left(V\right)\otimes G\left(W\right)\right)\otimes FG^{2}\left(U\right) & & \left(G^{2}\left(W\right)\otimes G\left(U\right)\right)\otimes FG\left(V\right)\\
			\left(G^{2}\left(W\right)\otimes G^{2}\left(V\right)\right)\otimes FG^{2}\left(U\right) & & FG^{2}\left(W\right)\otimes\left(G\left(U\right)\otimes G\left(V\right)\right)\\
												 & FG^{2}\left(W\right)\otimes\left(G^{2}\left(V\right)\otimes G^{2}\left(U\right)\right) & \\};
  \path[-stealth]
	  (m-1-2) edge node [left] {\tiny{$ d_{U,V}\otimes\text{id}_{F\left(W\right)} \hskip0.25in $}} (m-2-1)
		        edge node [right] {\tiny{$ ~\ ~\ a_{U,V,W}$}} (m-2-3)
		(m-2-1) edge node [right] {\tiny{$ ~\ a_{G\left(V\right),G\left(U\right),W}$}} (m-3-1)
    (m-2-3) edge node [right] {\tiny{$\text{id}_{F\left(U\right)}\otimes d_{V,W} ~\ $}} (m-3-3)
		(m-3-1) edge node [right] {\tiny{$\text{id}_{FG\left(V\right)}\otimes d_{G\left(U\right),W} ~\ $}} (m-4-1)
		(m-3-3) edge node [right] {\tiny{$a^{-1}_{U,G\left(W\right),G\left(V\right)}$}} (m-4-3)
		(m-4-1) edge node [right] {\tiny{$a^{-1}_{G\left(V\right),G\left(W\right),G^{2}\left(U\right)}$}} (m-5-1)
		(m-4-3) edge node [right] {\tiny{$d_{U,G\left(W\right)}\otimes\text{id}_{FG\left(V\right)}$}} (m-5-3)
		(m-5-1) edge node [right] {\tiny{$d_{G\left(V\right),G\left(W\right)}\otimes\text{id}_{FG^{2}\left(U\right)}$}} (m-6-1)
		(m-5-3) edge node [right] {\tiny{$a_{G^{2}\left(W\right),G\left(U\right),G\left(V\right)}$}} (m-6-3)
		(m-6-1) edge node [left] {\tiny{$a_{G^{2}\left(W\right),G^{2}\left(V\right),G^{2}\left(U\right)} ~\ ~\ $}} (m-7-2)
		(m-6-3) edge node [right] {\tiny{$ \hskip0.25in \text{id}_{FG^{2}\left(W\right)}\otimes d_{G\left(U\right),G\left(V\right)}$}} (m-7-2);
\end{tikzpicture}
\end{center}
\caption{The hom-Yang-Baxter property}
\label{fig17}
\end{figure}
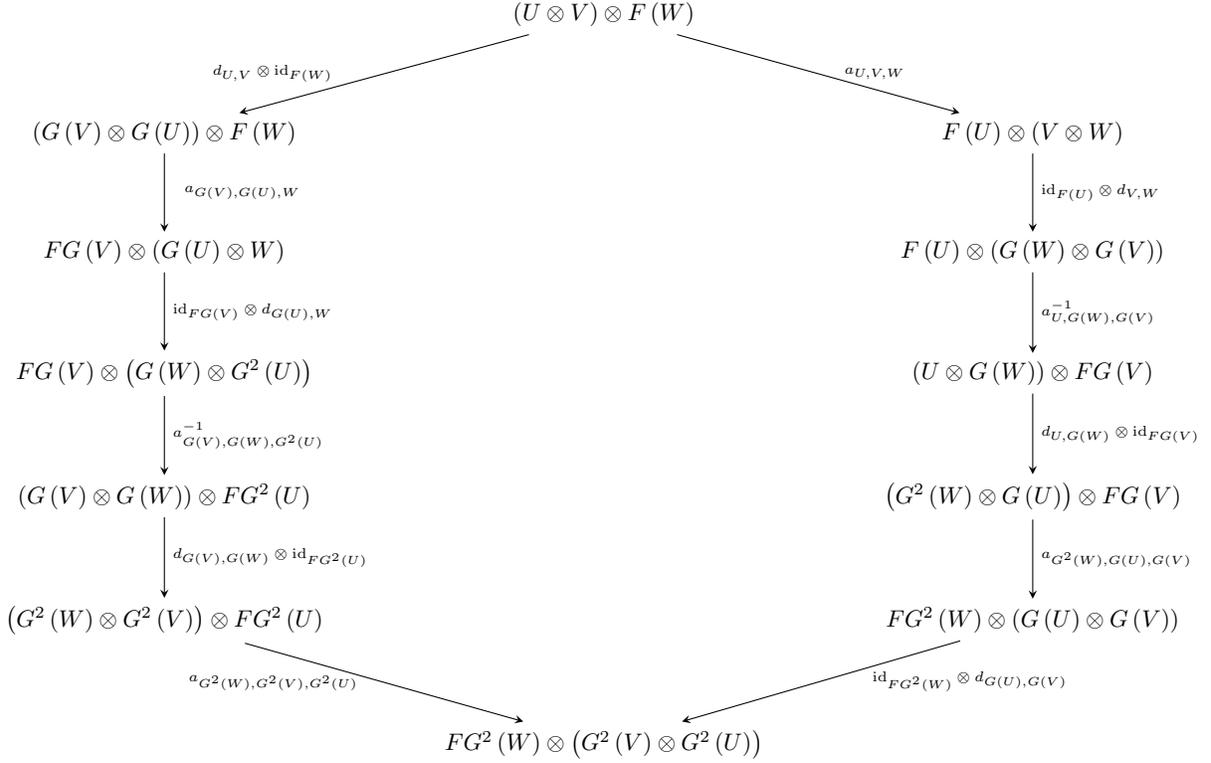
\end{center}

\end{Def}
\begin{Prop}
Let $\left(\mathcal{C},\otimes,F,G,a,\Phi,d\right)$ be a a hom-braided category. Then $d$ has the hom-Yang-Baxter property.
\label{prop5}
\end{Prop}

\begin{proof}
Let $\mathscr{C}=\left(\mathcal{C},\otimes,F,G,a,\Phi,d\right)$ be a hom-braided category. The commutative diagram in Figure \ref{fig97}  together with the hom-associativity constraint $a$ being an  isomorphism implies that the diagram  in Figure \ref{fig1091} commutes for all objects $U,V,W$ in $\mathscr{C}$. Observe that the commutative diagram in Figure \ref{fig1091} is precisely the bold portion of the diagram  in Figure \ref{fig53}.

Next, consider substituting the object $G(V)$ for $U$ and the object $G(U)$ for $V$ in the diagram of Figure \ref{fig1091}. Then the commutative diagram in Figure \ref{fig1091} implies that the diagram in Figure \ref{fig113} commutes for all objects $U,V,W$ in $\mathscr{C}$. Observe that the commutative diagram in Figure \ref{fig113} is precisely the dashed portion of the diagram seen in Figure \ref{fig53}. Thus, to prove that the diagram in Figure \ref{fig17} commutes it suffices to show that the outermost perimeter of the diagram seen in Figure \ref{fig43} commutes.

Indeed, the square diagram indicated by \textbf{1$\#$}  in Figure \ref{fig43} commutes as a consequence of the naturality of $d$ and the fact that $G(d_{U,V})=d_{G(U),G(V)}$. While the diagram indicated by  \textbf{2$\#$} in Figure \ref{fig43} commutes because $\otimes$ is a functor. Therefore, the diagram in Figure \ref{fig17} commutes. 
\end{proof}

\begin{center}
\begin{figure}[!ht]
\begin{center}
\begin{tikzpicture}
   \matrix (m) [matrix of math nodes,row sep=2.5em,column sep=2em,minimum width=2em] 
     {  &  \stackrel{\stackrel{\mbox{$FG\left(W\right)\otimes\left(G\left(U\right)\otimes G\left(V\right)\right)$}}{\mid\mid}}{GF\left(W\right)\otimes G\left(U\otimes V\right)} & \\                 
		\left(U\otimes V\right)\otimes F\left(W\right) & & FG^{2}\left(W\right)\otimes\left(G\left(U\right)\otimes G\left(V\right)\right) \\
		                       & & \\
       F\left(U\right)\otimes\left(V\otimes W\right) & & \left(G^{2}\left(W\right)\otimes G\left(U\right)\right)\otimes FG\left(V\right) \\
			                       & & \\
 F\left(U\right)\otimes \left(G\left(W\right)\otimes G\left(V\right)\right) & & \left(U\otimes G\left(W\right)\right)\otimes FG\left(V\right) \\};
  \path[-stealth]
	  (m-2-1) edge node [above] {\tiny{$d_{U\otimes V,F\left(W\right)} ~\ ~\ ~\ ~\ ~\ $}} (m-1-2)
		        edge node [right] {\tiny{$a_{U,V,W} ~\ $}} (m-4-1)
		(m-1-2) edge node [above] {\tiny{$ \hskip1in \Phi_{GF\left(W\right)}\otimes \left(\text{id}_{G\left(U\right)}\otimes\text{id}_{G\left(V\right)}\right)$}} (m-2-3)
		(m-4-3) edge node [left] {\tiny{$ ~\ ~\ a_{G^{2}\left(W\right),G\left(U\right),G\left(V\right)}$}} (m-2-3)
    (m-4-1) edge node [right] {\tiny{$\text{id}_{F\left(U\right)}\otimes d_{V,W} ~\ $}} (m-6-1)
		(m-6-1) edge node [above] {\tiny{$a^{-1}_{U,G\left(W\right),G\left(V\right)}$}} (m-6-3)
		(m-6-3) edge node [left] {\tiny{$d_{U,G\left(W\right)}\otimes\text{id}_{FG\left(V\right)}$}} (m-4-3);
\end{tikzpicture}
\end{center}
\caption{The commutative bold portion of Figure \ref{fig53}}
\label{fig1091}
\end{figure}
\end{center}

\begin{center}
\begin{figure}[!ht]
\begin{center}
\begin{tikzpicture}
   \matrix (m) [matrix of math nodes,row sep=2.5em,column sep=2em,minimum width=2em] 
     {  &  \stackrel{\stackrel{\mbox{$FG\left(W\right)\otimes \left(G^{2}\left(V\right)\otimes G^{2}\left(U\right)\right)$}}{\mid\mid}}{GF\left(W\right)\otimes G\left(G\left(V\right)\otimes G\left(U\right)\right)} & \\                 
		\left(G\left(V\right)\otimes G\left(U\right)\right)\otimes F\left(W\right) & & FG^{2}\left(W\right)\otimes\left(G^{2}\left(V\right)\otimes G^{2}\left(U\right)\right) \\
		                       & & \\
       FG\left(V\right)\otimes\left(G\left(U\right)\otimes W\right) & & \left(G^{2}\left(W\right)\otimes G^{2}\left(V\right)\right)\otimes FG^{2}\left(U\right) \\
			                       & & \\
 FG\left(V\right)\otimes\left(G\left(W\right)\otimes G^{2}\left(U\right)\right) & & \left(G\left(V\right)\otimes G\left(W\right)\right)\otimes FG^{2}\left(U\right) \\};
  \path[-stealth]
	  (m-2-1) edge node [above] {\tiny{$d_{G\left(V\right)\otimes G\left(U\right),F\left(W\right)} \hskip.5in $}} (m-1-2)
		        edge node [right] {\tiny{$a_{G\left(V\right),G\left(U\right),W} ~\ $}} (m-4-1)
		(m-1-2) edge node [above] {\tiny{$ \hskip1in \Phi_{GF\left(W\right)}\otimes\left(\text{id}_{G^{2}\left(V\right)}\otimes\text{id}_{G^{2}\left(U\right)}\right)$}} (m-2-3)
		(m-4-3) edge node [left] {\tiny{$ ~\ ~\ a_{G^{2}\left(W\right),G^{2}\left(V\right),G^{2}\left(U\right)}$}} (m-2-3)
    (m-4-1) edge node [right] {\tiny{$\text{id}_{FG\left(V\right)}\otimes d_{G\left(U\right),W} ~\ $}} (m-6-1)
		(m-6-1) edge node [above] {\tiny{$a^{-1}_{G\left(V\right),G\left(W\right),G^{2}\left(U\right)}$}} (m-6-3)
		(m-6-3) edge node [left] {\tiny{$d_{G\left(V\right),G\left(W\right)}\otimes\text{id}_{FG^{2}\left(U\right)}$}} (m-4-3);
\end{tikzpicture}
\end{center}
\caption{The commutative dashed portion of Figure \ref{fig53}}
\label{fig113}
\end{figure}
\end{center}

\begin{center}
\begin{figure}[!ht]
\begin{center}
\begin{tikzpicture}[scale=0.9, every node/.style={scale=0.9}]
  \matrix (m) [matrix of math nodes,row sep=3em,column sep=4em,minimum width=2em] 
     {                  & \left(U\otimes V\right)\otimes F\left(W\right) & \\ 
		  \left(G\left(V\right)\otimes G\left(U\right)\right)\otimes F\left(W\right) & & F\left(U\right)\otimes\left(V\otimes W\right)\\
      FG\left(V\right)\otimes\left(G\left(U\right)\otimes W\right)  & GF\left(W\right)\otimes \left(G\left(U\right)\otimes G\left(V\right)\right) & F\left(U\right)\otimes\left(G\left(W\right)\otimes G\left(V\right)\right)\\
			 FG\left(V\right)\otimes\left(G\left(W\right)\otimes G^{2}\left(U\right)\right) & & \left(U\otimes G\left(W\right)\right)\otimes FG\left(V\right)\\
			\left(G\left(V\right)\otimes G\left(W\right)\right)\otimes FG^{2}\left(U\right) & GF\left(W\right)\otimes \left(G^{2}\left(V\right)\otimes G^{2}\left(U\right)\right) & \left(G^{2}\left(W\right)\otimes G\left(U\right)\right)\otimes FG\left(V\right)\\
			\left(G^{2}\left(W\right)\otimes G^{2}\left(V\right)\right)\otimes FG^{2}\left(U\right) & & FG^{2}\left(W\right)\otimes\left(G\left(U\right)\otimes G\left(V\right)\right)\\
												 & FG^{2}\left(W\right)\otimes\left(G^{2}\left(V\right)\otimes G^{2}\left(U\right)\right) & \\};
  \path[-stealth]
	  (m-1-2) edge node [left] {\tiny{$ d_{U,V}\otimes\text{id}_{F\left(W\right)} \hskip0.25in $}} (m-2-1)
		        edge[line width=1.5pt] node [right] {\tiny{$ ~\ ~\ a_{U,V,W}$}} (m-2-3)
						edge[line width=1.5pt] node [right] {\tiny{$d_{U\otimes V,F\left(W\right)}$}} (m-3-2)
		(m-2-1) edge[line width=1.5pt,dashed] node [left] {\tiny{$a_{G\left(V\right),G\left(U\right),W}$}} (m-3-1)
		        edge[line width=1.5pt,dashed] node [right] {\tiny{$d_{G\left(V\right)\otimes G\left(U\right),F\left(W\right)}$}} (m-5-2)
    (m-2-3) edge[line width=1.5pt] node [right] {\tiny{$\text{id}_{F\left(U\right)}\otimes d_{V,W} ~\ $}} (m-3-3)
		(m-3-1) edge[line width=1.5pt,dashed] node [right] {\tiny{$\text{id}_{FG\left(V\right)}\otimes d_{G\left(U\right),W} ~\ $}} (m-4-1)
		(m-3-2) edge[line width=1.5pt,bend right=15,looseness=1] node [right] {\tiny{$\Phi_{GF\left(W\right)}\otimes\left(\text{id}_{G\left(U\right)}\otimes\text{id}_{G\left(V\right)}\right)$}} (m-6-3)
		(m-3-3) edge[line width=1.5pt] node [right] {\tiny{$a^{-1}_{U,G\left(W\right),G\left(V\right)}$}} (m-4-3)
		(m-4-1) edge[line width=1.5pt,dashed] node [right] {\tiny{$a^{-1}_{G\left(V\right),G\left(W\right),G^{2}\left(U\right)}$}} (m-5-1)
		(m-4-3) edge[line width=1.5pt] node [right] {\tiny{$d_{U,G\left(W\right)}\otimes\text{id}_{FG\left(V\right)}$}} (m-5-3)
		(m-5-1) edge[line width=1.5pt,dashed] node [right] {\tiny{$d_{G\left(V\right),G\left(W\right)}\otimes\text{id}_{FG^{2}\left(U\right)}$}} (m-6-1)
		(m-5-2) edge[line width=1.5pt,dashed,bend right=20,looseness=1] node [right] {\tiny{$\Phi_{GF\left(W\right)}\otimes\left(\text{id}_{G^{2}\left(V\right)}\otimes\text{id}_{G^{2}\left(U\right)}\right)$}} (m-7-2)
		(m-5-3) edge[line width=1.5pt] node [right] {\tiny{$a_{G^{2}\left(W\right),G\left(U\right),G\left(V\right)}$}} (m-6-3)
		(m-6-1) edge[line width=1.5pt,dashed] node [left] {\tiny{$a_{G^{2}\left(W\right),G^{2}\left(V\right),G^{2}\left(U\right)} ~\ ~\ $}} (m-7-2)
		(m-6-3) edge node [right] {\tiny{$ \hskip0.25in \text{id}_{FG^{2}\left(W\right)}\otimes d_{G\left(U\right),G\left(V\right)}$}} (m-7-2);
\end{tikzpicture}
\end{center}
\caption{Proof of the hom-Yang-Baxter property}
\label{fig53}
\end{figure}
\end{center}

\begin{center}
\begin{figure}[!ht]
\begin{center}
\begin{tikzpicture}
  \matrix (m) [matrix of math nodes,row sep=4em,column sep=3em,minimum width=0.5em] 
     { \left(G\left(V\right)\otimes G\left(U\right)\right)\otimes F\left(W\right) & & \left(U\otimes V\right)\otimes F\left(W\right) \\ 
		    GF\left(W\right)\otimes \left(G^{2}\left(V\right)\otimes G^{2}\left(U\right)\right) & & GF\left(W\right)\otimes \left(G\left(U\right)\otimes G\left(V\right)\right) \\
			  FG^{2}\left(W\right)\otimes\left(G^{2}\left(V\right)\otimes G^{2}\left(U\right)\right) & & FG^{2}\left(W\right)\otimes\left(G\left(U\right)\otimes G\left(V\right)\right) \\};
  \path[-stealth]
	  (m-1-3) edge node [above] {\tiny{$ \hskip1mm d_{U,V}\otimes\text{id}_{F\left(W\right)} $}} (m-1-1)
		        edge[line width=1.25pt] node [right] {\tiny{$d_{U\otimes V,F\left(W\right)}$}} (m-2-3)
		(m-1-1) edge[line width=1.25pt,dashed] node [left] {\tiny{$d_{G\left(V\right)\otimes G\left(U\right),F\left(W\right)}$}} node [right] {$\hskip3cm \textbf{1}\# $} (m-2-1)
    (m-2-3) edge[line width=1.25pt] node [right] {\tiny{$\Phi_{GF\left(W\right)}\otimes\left(\text{id}_{G\left(U\right)}\otimes\text{id}_{G\left(V\right)}\right)$}} (m-3-3)
		        edge node [below] {\tiny{$ \hskip7mm \text{id}_{FG\left(W\right)}\otimes d_{G\left(U\right),G\left(V\right)}$}} (m-2-1)
		(m-2-1) edge[line width=1.25pt,dashed] node [left] {\tiny{$\Phi_{GF\left(W\right)}\otimes\left(\text{id}_{G^{2}\left(V\right)}\otimes\text{id}_{G^{2}\left(U\right)}\right)$}} node [right] {$\hskip3cm \textbf{2}\# $} (m-3-1)
		(m-3-3) edge node [below] {\tiny{$ \hskip7mm \text{id}_{FG^{2}\left(W\right)}\otimes d_{G\left(U\right),G\left(V\right)}$}} (m-3-1);
\end{tikzpicture}
\end{center}
\caption{Naturality for $d$ and $G(d_{U,V})=d_{G(U),G(V)}$}
\label{fig43}
\end{figure}
\end{center}

Next we  introduce a variation on the definition of hom-braided categories. 

\begin{Def}
Let $\mathscr{C}=\left(\mathcal{C},\otimes,F, G,a,\Phi\right)$ be a hom-tensor category. We say that a hom-commutativity constraint $d$ (as in Definition \ref{def8}) satisfies the \emph{(wH1) property} if the diagram in Figure \ref{fig127} commutes for all objects $U$, $V$ and $W$ of the category $\mathscr{C}$. Furthermore, we say that a hom-commutativity constraint $d$ satisfies the \emph{(wH2) property} if the diagram in Figure \ref{fig131} commutes for all objects $U$, $V$ and $W$ of the category $\mathscr{C}$.

\begin{center}
\begin{figure}[!ht]
\begin{center}
\begin{tikzpicture}
  \matrix (m) [matrix of math nodes,row sep=2.5em,column sep=2em,minimum width=2em] 
     {  &  \stackrel{\stackrel{\mbox{$\left(G\left(V\right)\otimes G\left(W\right)\right)\otimes FG\left(U\right)$}}{\mid\mid}}{G\left(V\otimes W\right)\otimes GF\left(U\right)} & \\                 
		F\left(U\right)\otimes\left(V\otimes W\right) & & \stackrel{\stackrel{\mbox{$FG\left(V\right)\otimes \left(G\left(W\right)\otimes G\left(U\right)\right)$}}{\mid\mid}}{GF\left(V\right)\otimes \left(G\left(W\right)\otimes G\left(U\right)\right)} \\
		                       & & \\
       \left(U\otimes V\right)\otimes F\left(W\right) & & G^{2}F\left(V\right)\otimes\left(G^{2}\left(W\right)\otimes G^{2}\left(U\right)\right) \\
			                       & & \\
 \stackrel{\stackrel{\mbox{$\left(G\left(V\right)\otimes G\left(U\right)\right)\otimes GF\left(W\right)$}}{\mid\mid}}{\left(G\left(V\right)\otimes G\left(U\right)\right)\otimes FG\left(W\right)} & & \stackrel{\stackrel{\mbox{$FG\left(V\right)\otimes \left(G\left(U\right)\otimes G\left(W\right)\right)$}}{\mid\mid}}{GF\left(V\right)\otimes \left(G\left(U\right)\otimes G\left(W\right)\right)} \\};
  \path[-stealth]
	  (m-2-1) edge node [above] {\tiny{$ \hskip-0.5in d_{F\left(U\right),V\otimes W} $}} (m-1-2)
		(m-1-2) edge node [above] {\tiny{$ \hskip0.5in a_{G\left(V\right),G\left(W\right),G\left(U\right)}$}} (m-2-3)
		(m-2-3) edge node [left] {\tiny{$ ~\ ~\ \Phi_{GF\left(V\right)}\otimes\left(\Phi_{G\left(W\right)}\otimes\Phi_{G\left(U\right)}\right)$}} (m-4-3)
    (m-4-1) edge node [right] {\tiny{$a_{U,V,W} ~\ $}} (m-2-1)
		        edge node [right] {\tiny{$d_{U,V}\otimes\Phi_{F\left(W\right)}~\ $}} (m-6-1)
		(m-6-1) edge node [above] {\tiny{$a_{G\left(V\right),G\left(U\right),G\left(W\right)}$}} (m-6-3)
		(m-6-3) edge node [left] {\tiny{$\Phi_{GF\left(V\right)}\otimes d_{G\left(U\right),G\left(W\right)}$}} (m-4-3);
\end{tikzpicture}
\end{center}
\caption{The  (wH1) property}
\label{fig127}
\end{figure}
\end{center}

\begin{center}
\begin{figure}[!ht]
\begin{center}
\begin{tikzpicture}
   \matrix (m) [matrix of math nodes,row sep=2.5em,column sep=2em,minimum width=2em] 
     {  &  \stackrel{\stackrel{\mbox{$FG\left(W\right)\otimes\left(G\left(U\right)\otimes G\left(V\right)\right)$}}{\mid\mid}}{GF\left(W\right)\otimes G\left(U\otimes V\right)} & \\                 
		\left(U\otimes V\right)\otimes F\left(W\right) & & \stackrel{\stackrel{\mbox{$\left(G\left(W\right)\otimes G\left(U\right)\right)\otimes FG\left(V\right)$}}{\mid\mid}}{\left(G\left(W\right)\otimes G\left(U\right)\right)\otimes GF\left(V\right)} \\
		                       & & \\
       F\left(U\right)\otimes\left(V\otimes W\right) & & \left(G^{2}\left(W\right)\otimes G^{2}\left(U\right)\right)\otimes G^{2}F\left(V\right) \\
			                       & & \\
 \stackrel{\stackrel{\mbox{$GF\left(U\right)\otimes\left(G\left(W\right)\otimes G\left(V\right)\right)$}}{\mid\mid}}{FG\left(U\right)\otimes\left(G\left(W\right)\otimes G\left(V\right)\right)} & & \stackrel{\stackrel{\mbox{$\left(G\left(U\right)\otimes G\left(W\right)\right)\otimes FG\left(V\right)$}}{\mid\mid}}{\left(G\left(U\right)\otimes G\left(W\right)\right)\otimes GF\left(V\right)} \\};
  \path[-stealth]
	  (m-2-1) edge node [above] {\tiny{$ \hskip-0.5in d_{U\otimes V,F\left(W\right)} $}} (m-1-2)
		(m-1-2) edge node [above] {\tiny{$ \hskip0.5in a^{-1}_{G\left(W\right),G\left(U\right),G\left(V\right)}$}} (m-2-3)
		(m-2-3) edge node [left] {\tiny{$ ~\ ~\ \left(\Phi_{G\left(W\right)}\otimes \Phi_{G\left(U\right)}\right)\otimes\Phi_{GF\left(V\right)}$}} (m-4-3)
    (m-4-1) edge node [right] {\tiny{$a^{-1}_{U,V,W} ~\ $}} (m-2-1)
		        edge node [right] {\tiny{$\Phi_{F\left(U\right)}\otimes d_{V,W} ~\ $}} (m-6-1)
		(m-6-1) edge node [above] {\tiny{$a^{-1}_{G\left(U\right),G\left(W\right),G\left(V\right)}$}} (m-6-3)
		(m-6-3) edge node [left] {\tiny{$d_{G\left(U\right),G\left(W\right)}\otimes\Phi_{GF\left(V\right)}$}} (m-4-3);
\end{tikzpicture}
\end{center}
\caption{The  (wH2) property}
\label{fig131}
\end{figure}
\end{center}
\label{def84}
\end{Def} 
\begin{Def}
Let $\mathscr{C}=\left(\mathcal{C},\otimes,F, G,a,\Phi\right)$ be a hom-tensor category. A \emph{weak hom-braiding} $d$ in $\mathscr{C}$ is a hom-commutativity constraint with the following conditions:
\begin{enumerate}[label=(\roman*)]
 \item $d$ satisfies (wH1) and (wH2);
 \item For all objects $U$ and $V$ in the category $\mathscr{C}$,  $G\left(d_{U,V}\right)=d_{G\left(U\right),G\left(V\right)}$.
\end{enumerate}
A \emph{weakly hom-braided category} $\left(\mathcal{C},\otimes,F,G,a,\Phi,d\right)$ is a hom-tensor category with weak hom-braiding $d$.
\label{def87}
\end{Def}

 Using arguments similar to those found in the proof for Lemma \ref{lem8}, one can show the following.

\begin{Lem}
The  commutative diagram in Figure \ref{fig137} is equivalent to the  (wH2) property of a hom-commutativity constraint $d$. We call the relation in Figure \ref{fig137} the $\left(\text{wH}^{\prime}2\right)$ property of $d$.

\begin{center}
\begin{figure}[!ht]
\begin{center}
\begin{tikzpicture}
   \matrix (m) [matrix of math nodes,row sep=2.5em,column sep=2em,minimum width=2em] 
     {  &  \stackrel{\stackrel{\mbox{$FG\left(W\right)\otimes\left(G\left(U\right)\otimes G\left(V\right)\right)$}}{\mid\mid}}{GF\left(W\right)\otimes G\left(U\otimes V\right)} & \\                 
		\left(U\otimes V\right)\otimes F\left(W\right) & & \stackrel{\stackrel{\mbox{$G^{2}F\left(W\right)\otimes \left(G^{2}\left(U\right)\otimes G^{2}\left(V\right)\right)$}}{\mid\mid}}{FG^{2}\left(W\right)\otimes \left(G^{2}\left(U\right)\otimes G^{2}\left(V\right)\right)} \\
		                       & & \\
       F\left(U\right)\otimes\left(V\otimes W\right) & & \left(G^{2}\left(W\right)\otimes G^{2}\left(U\right)\right)\otimes G^{2}F\left(V\right) \\
			                       & & \\
 \stackrel{\stackrel{\mbox{$GF\left(U\right)\otimes\left(G\left(W\right)\otimes G\left(V\right)\right)$}}{\mid\mid}}{FG\left(U\right)\otimes\left(G\left(W\right)\otimes G\left(V\right)\right)} & & \stackrel{\stackrel{\mbox{$\left(G\left(U\right)\otimes G\left(W\right)\right)\otimes FG\left(V\right)$}}{\mid\mid}}{\left(G\left(U\right)\otimes G\left(W\right)\right)\otimes GF\left(V\right)} \\};
  \path[-stealth]
	  (m-2-1) edge node [above] {\tiny{$ \hskip-0.5in d_{U\otimes V,F\left(W\right)} $}} (m-1-2)
		(m-1-2) edge node [above] {\tiny{$ \hskip1in \Phi_{GF\left(W\right)}\otimes \left(\Phi_{G\left(U\right)}\otimes\Phi_{G\left(V\right)}\right)$}} (m-2-3)
		(m-2-3) edge node [left] {\tiny{$ ~\ ~\ a^{-1}_{G^{2}\left(W\right),G^{2}\left(U\right),G^{2}\left(V\right)} $}} (m-4-3)
    (m-4-1) edge node [right] {\tiny{$a^{-1}_{U,V,W} ~\ $}} (m-2-1)
		        edge node [right] {\tiny{$\Phi_{F\left(U\right)}\otimes d_{V,W} ~\ $}} (m-6-1)
		(m-6-1) edge node [above] {\tiny{$a^{-1}_{G\left(U\right),G\left(W\right),G\left(V\right)}$}} (m-6-3)
		(m-6-3) edge node [left] {\tiny{$d_{G\left(U\right),G\left(W\right)}\otimes\Phi_{GF\left(V\right)}$}} (m-4-3);
\end{tikzpicture}
\end{center}
\caption{The commutative diagram for the $\left(\text{wH}^{\prime}2\right)$ property}
\label{fig137}
\end{figure}
\end{center}

\label{lem16}
\end{Lem}

\begin{Def} Let $\mathscr{C}=\left(\mathcal{C},\otimes,F,G,a,\Phi \right)$ be a hom-tensor category  and $d$  a hom-commutativity constraint. We say that $d$ has the {\em weak hom-Yang-Baxter property} if  $d$ satisfies the equation in  Figure \ref{fig59}.

\begin{center}
\begin{figure}[!ht]
\begin{center}
\begin{tikzpicture}[scale=0.85, every node/.style={scale=0.85}] 

  \matrix (m) [matrix of math nodes,row sep=3em,column sep=4em,minimum width=2em] 
     {                  & \left(U\otimes V\right)\otimes F\left(W\right) & \\ 
		  \left(G\left(V\right)\otimes G\left(U\right)\right)\otimes GF\left(W\right) & & F\left(U\right)\otimes\left(V\otimes W\right)\\
      FG\left(V\right)\otimes\left(G\left(U\right)\otimes G\left(W\right)\right)  & & GF\left(U\right)\otimes\left(G\left(W\right)\otimes G\left(V\right)\right)\\
			 G^{2}F\left(V\right)\otimes\left(G^{2}\left(W\right)\otimes G^{2}\left(U\right)\right) & & \left(G\left(U\right)\otimes G\left(W\right)\right)\otimes FG\left(V\right)\\
			\left(G^{2}\left(V\right)\otimes G^{2}\left(W\right)\right)\otimes FG^{2}\left(U\right) & & \left(G^{2}\left(W\right)\otimes G^{2}\left(U\right)\right)\otimes G^{2}F\left(V\right)\\
			\left(G^{3}\left(W\right)\otimes G^{3}\left(V\right)\right)\otimes G^{3}F\left(U\right) & & FG^{2}\left(W\right)\otimes\left(G^{2}\left(U\right)\otimes G^{2}\left(V\right)\right)\\
												 & FG^{3}\left(W\right)\otimes\left(G^{3}\left(V\right)\otimes G^{3}\left(U\right)\right) & \\};
  \path[-stealth]
	  (m-1-2) edge node [left] {\tiny{$ d_{U,V}\otimes\Phi_{F\left(W\right)} \hskip0.25in $}} (m-2-1)
		        edge node [right] {\tiny{$ ~\ ~\ a_{U,V,W}$}} (m-2-3)
		(m-2-1) edge node [right] {\tiny{$ ~\ a_{G\left(V\right),G\left(U\right),G\left(W\right)}$}} (m-3-1)
    (m-2-3) edge node [right] {\tiny{$\Phi_{F\left(U\right)}\otimes d_{V,W} ~\ $}} (m-3-3)
		(m-3-1) edge node [right] {\tiny{$\Phi_{GF\left(V\right)}\otimes d_{G\left(U\right),G\left(W\right)} ~\ $}} (m-4-1)
		(m-3-3) edge node [right] {\tiny{$a^{-1}_{G\left(U\right),G\left(W\right),G\left(V\right)}$}} (m-4-3)
		(m-4-1) edge node [right] {\tiny{$a^{-1}_{G^{2}\left(V\right),G^{2}\left(W\right),G^{2}\left(U\right)}$}} (m-5-1)
		(m-4-3) edge node [right] {\tiny{$d_{G\left(U\right),G\left(W\right)}\otimes\Phi_{GF\left(V\right)}$}} (m-5-3)
		(m-5-1) edge node [right] {\tiny{$d_{G^{2}\left(V\right),G^{2}\left(W\right)}\otimes\Phi_{G^{2}F\left(U\right)}$}} (m-6-1)
		(m-5-3) edge node [right] {\tiny{$a_{G^{2}\left(W\right),G^{2}\left(U\right),G^{2}\left(V\right)}$}} (m-6-3)
		(m-6-1) edge node [left] {\tiny{$a_{G^{3}\left(W\right),G^{3}\left(V\right),G^{3}\left(U\right)} ~\ ~\ $}} (m-7-2)
		(m-6-3) edge node [right] {\tiny{$ \hskip0.25in \Phi_{G^{2}F\left(W\right)}\otimes d_{G^{2}\left(U\right),G^{2}\left(V\right)}$}} (m-7-2);
\end{tikzpicture}
\end{center}
\caption{The weak hom-Yang-Baxter property}
\label{fig59}
\end{figure}
\end{center}
\end{Def}

\begin{Prop}\label{prop6.9}
Let $\mathscr{C}=\left(\mathcal{C},\otimes,F,G,a,\Phi,d\right)$ be a weakly hom-braided category. Then $d$ has the weak hom-Yang-Baxter property.
\label{prop20}
\end{Prop}
\begin{proof}
Let $\mathscr{C}=\left(\mathcal{C},\otimes,F,G,a,\Phi,d\right)$ be a weakly hom-braided category. The commutative diagram in Figure \ref{fig137} together with the hom-associativity constraint $a$ being an isomorphism implies that the diagram in Figure \ref{fig139} commutes for all objects $U,V,W$ in $\mathscr{C}$. Observe that the commutative diagram in Figure \ref{fig139} is precisely the bold portion of the diagram seen in Figure \ref{fig151}.

Next, consider substituting the object $G(V)$ for $U$, the object $G(U)$ for $V$ and the object $G(W)$ for $W$ in the diagram of Figure \ref{fig139}. Then the commutative diagram in Figure \ref{fig139} implies that the diagram in Figure \ref{fig149} commutes for all objects $U,V,W$ in $\mathscr{C}$. Observe that the commutative diagram in Figure \ref{fig149} is precisely the dashed portion of the diagram  in Figure \ref{fig151}. Thus, to prove that the diagram in Figure \ref{fig59} commutes it suffices to show that the outermost perimeter of the diagram in Figure \ref{fig157} commutes.

Indeed, the square diagram labeled \textbf{1$\#$}  in Figure \ref{fig157} commutes as a consequence of the naturality of $d$ and the fact that $G(\Phi_{F(W)})=\Phi_{GF(W)}$ and $G(d_{U,V})=d_{G(U),G(V)}$. The square diagram labeled  \textbf{2$\#$} commutes because $\otimes$ is a functor, $d$ is natural  and $G(\Phi_{G(V)})=\Phi_{G^2(V)}$. Thus, the diagram in Figure \ref{fig59} commutes. 
\end{proof}

\begin{center}
\begin{figure}[!ht]
\begin{center}
\begin{tikzpicture}
   \matrix (m) [matrix of math nodes,row sep=2.5em,column sep=2em,minimum width=2em] 
     {  &  GF\left(W\right)\otimes\left(G\left(U\right)\otimes G\left(V\right)\right) & \\                 
		\left(U\otimes V\right)\otimes F\left(W\right) & & G^{2}F\left(W\right)\otimes \left(G^{2}\left(U\right)\otimes G^{2}\left(V\right)\right) \\
		                       & & \\
       F\left(U\right)\otimes\left(V\otimes W\right) & & \left(G^{2}\left(W\right)\otimes G^{2}\left(U\right)\right)\otimes G^{2}F\left(V\right) \\
			                       & & \\
 GF\left(U\right)\otimes\left(G\left(W\right)\otimes G\left(V\right)\right) & & \left(G\left(U\right)\otimes G\left(W\right)\right)\otimes FG\left(V\right) \\};
  \path[-stealth]
	  (m-2-1) edge node [above] {\tiny{$ \hskip-0.5in d_{U\otimes V,F\left(W\right)} $}} (m-1-2)
		        edge node [right] {\tiny{$a_{U,V,W} ~\ $}} (m-4-1)
		(m-1-2) edge node [above] {\tiny{$ \hskip1in \Phi_{GF\left(W\right)}\otimes \left(\Phi_{G\left(U\right)}\otimes\Phi_{G\left(V\right)}\right)$}} (m-2-3)
		(m-4-3) edge node [left] {\tiny{$ ~\ ~\ a_{G^{2}\left(W\right),G^{2}\left(U\right),G^{2}\left(V\right)} $}} (m-2-3)
    
		(m-4-1) edge node [right] {\tiny{$\Phi_{F\left(U\right)}\otimes d_{V,W} ~\ $}} (m-6-1)
		(m-6-1) edge node [above] {\tiny{$a^{-1}_{G\left(U\right),G\left(W\right),G\left(V\right)}$}} (m-6-3)
		(m-6-3) edge node [left] {\tiny{$d_{G\left(U\right),G\left(W\right)}\otimes\Phi_{GF\left(V\right)}$}} (m-4-3);
\end{tikzpicture}
\end{center}
\caption{The commutative bold portion of Figure \ref{fig151}}
\label{fig139}
\end{figure}
\end{center}

\begin{center}
\begin{figure}
\begin{center}
\begin{tikzpicture}
   \matrix (m) [matrix of math nodes,row sep=2.5em,column sep=2em,minimum width=2em] 
     {  &  G^{2}F\left(W\right)\otimes\left(G^{2}\left(V\right)\otimes G^{2}\left(U\right)\right) & \\                 
		\left(G\left(V\right)\otimes G\left(U\right)\right)\otimes GF\left(W\right) & & FG^{3}\left(W\right)\otimes\left(G^{3}\left(V\right)\otimes G^{3}\left(U\right)\right) \\
		                       & & \\
       FG\left(V\right)\otimes\left(G\left(U\right)\otimes G\left(W\right)\right) & & \left(G^{3}\left(W\right)\otimes G^{3}\left(V\right)\right)\otimes G^{3}F\left(U\right) \\
			                       & & \\
 G^{2}F\left(V\right)\otimes\left(G^{2}\left(W\right)\otimes G^{2}\left(U\right)\right) & & \left(G^{2}\left(V\right)\otimes G^{2}\left(W\right)\right)\otimes FG^{2}\left(U\right) \\};
  \path[-stealth]
	  (m-2-1) edge node [above] {\tiny{$ \hskip-0.5in d_{G\left(V\right)\otimes G\left(U\right),GF\left(W\right)} $}} (m-1-2)
		        edge node [right] {\tiny{$a_{G\left(V\right),G\left(U\right),G\left(W\right)} ~\ $}} (m-4-1)
		(m-1-2) edge node [above] {\tiny{$ \hskip1in \Phi_{G^{2}F\left(W\right)}\otimes \left(\Phi_{G^{2}\left(V\right)}\otimes\Phi_{G^{2}\left(U\right)}\right)$}} (m-2-3)
		(m-4-3) edge node [left] {\tiny{$ ~\ ~\ a_{G^{3}\left(W\right),G^{3}\left(V\right),G^{3}\left(U\right)} $}} (m-2-3)
    
		(m-4-1) edge node [right] {\tiny{$\Phi_{FG\left(V\right)}\otimes d_{G\left(U\right),G\left(W\right)} ~\ $}} (m-6-1)
		(m-6-1) edge node [above] {\tiny{$a^{-1}_{G^{2}\left(V\right),G^{2}\left(W\right),G^{2}\left(U\right)}$}} (m-6-3)
		(m-6-3) edge node [left] {\tiny{$d_{G^{2}\left(V\right),G^{2}\left(W\right)}\otimes\Phi_{G^{2}F\left(U\right)}$}} (m-4-3);
\end{tikzpicture}
\end{center}
\caption{The commutative dashed portion of Figure \ref{fig151}}
\label{fig149}
\end{figure}
\end{center}

\begin{center}
\begin{figure}[!ht]
\begin{center}
\begin{tikzpicture}[scale=0.85, every node/.style={scale=0.85}]
  \matrix (m) [matrix of math nodes,row sep=3em,column sep=4em,minimum width=2em] 
     {                  & \left(U\otimes V\right)\otimes F\left(W\right) & \\ 
		  \left(G\left(V\right)\otimes G\left(U\right)\right)\otimes GF\left(W\right) & & F\left(U\right)\otimes\left(V\otimes W\right)\\
      FG\left(V\right)\otimes\left(G\left(U\right)\otimes G\left(W\right)\right)  & GF\left(W\right)\otimes \left(G\left(U\right)\otimes G\left(V\right)\right) & GF\left(U\right)\otimes\left(G\left(W\right)\otimes G\left(V\right)\right)\\
			 G^{2}F\left(V\right)\otimes\left(G^{2}\left(W\right)\otimes G^{2}\left(U\right)\right) & & \left(G\left(U\right)\otimes G\left(W\right)\right)\otimes FG\left(V\right)\\
			\left(G^{2}\left(V\right)\otimes G^{2}\left(W\right)\right)\otimes FG^{2}\left(U\right) & G^{2}F\left(W\right)\otimes\left(G^{2}\left(V\right)\otimes G^{2}\left(U\right)\right) & \left(G^{2}\left(W\right)\otimes G^{2}\left(U\right)\right)\otimes G^{2}F\left(V\right)\\
			\left(G^{3}\left(W\right)\otimes G^{3}\left(V\right)\right)\otimes G^{3}F\left(U\right) & & FG^{2}\left(W\right)\otimes\left(G^{2}\left(U\right)\otimes G^{2}\left(V\right)\right)\\
												 & FG^{3}\left(W\right)\otimes\left(G^{3}\left(V\right)\otimes G^{3}\left(U\right)\right) & \\};
  \path[-stealth]
	  (m-1-2) edge node [left] {\tiny{$ d_{U,V}\otimes\Phi_{F\left(W\right)} \hskip0.25in $}} (m-2-1)
		        edge[line width=1.5pt] node [right] {\tiny{$ ~\ ~\ a_{U,V,W}$}} (m-2-3)
						edge[line width=1.5pt] node [right] {\tiny{$d_{U\otimes V,F\left(W\right)}$}} (m-3-2)
		(m-2-1) edge[line width=1.5pt,dashed] node [left] {\tiny{$ ~\ a_{G\left(V\right),G\left(U\right),G\left(W\right)}$}} (m-3-1)
		        edge[line width=1.5pt,dashed] node [right] {\tiny{$d_{G\left(V\right)\otimes G\left(U\right),GF\left(W\right)}$}} (m-5-2)
    (m-2-3) edge[line width=1.5pt] node [right] {\tiny{$\Phi_{F\left(U\right)}\otimes d_{V,W} ~\ $}} (m-3-3)
		(m-3-1) edge[line width=1.5pt,dashed] node [right] {\tiny{$\Phi_{GF\left(V\right)}\otimes d_{G\left(U\right),G\left(W\right)} ~\ $}} (m-4-1)
		(m-3-2) edge[line width=1.5pt,bend right=15,looseness=1] node [right] {\tiny{$\Phi_{GF\left(W\right)}\otimes\left(\Phi_{G\left(U\right)}\otimes\Phi_{G\left(V\right)}\right)$}} (m-6-3)
		(m-3-3) edge[line width=1.5pt] node [right] {\tiny{$a^{-1}_{G\left(U\right),G\left(W\right),G\left(V\right)}$}} (m-4-3)
		(m-4-1) edge[line width=1.5pt,dashed] node [right] {\tiny{$a^{-1}_{G^{2}\left(V\right),G^{2}\left(W\right),G^{2}\left(U\right)}$}} (m-5-1)
		(m-4-3) edge[line width=1.5pt] node [right] {\tiny{$d_{G\left(U\right),G\left(W\right)}\otimes\Phi_{GF\left(V\right)}$}} (m-5-3)
		(m-5-1) edge[line width=1.5pt,dashed] node [right] {\tiny{$d_{G^{2}\left(V\right),G^{2}\left(W\right)}\otimes\Phi_{G^{2}F\left(U\right)}$}} (m-6-1)
		(m-5-2) edge[line width=1.5pt,dashed,bend right=20,looseness=1] node [right] {\tiny{$\Phi_{G^{2}F\left(W\right)}\otimes\left(\Phi_{G^{2}\left(V\right)}\otimes\Phi_{G^{2}\left(U\right)}\right)$}} (m-7-2)
		(m-5-3) edge[line width=1.5pt] node [right] {\tiny{$a_{G^{2}\left(W\right),G^{2}\left(U\right),G^{2}\left(V\right)}$}} (m-6-3)
		(m-6-1) edge[line width=1.5pt,dashed] node [left] {\tiny{$a_{G^{3}\left(W\right),G^{3}\left(V\right),G^{3}\left(U\right)} ~\ ~\ $}} (m-7-2)
		(m-6-3) edge node [right] {\tiny{$ \hskip0.25in \Phi_{G^{2}F\left(W\right)}\otimes d_{G^{2}\left(U\right),G^{2}\left(V\right)}$}} (m-7-2);
\end{tikzpicture}
\end{center}
\caption{Proof of the weak hom-Yang-Baxter property}
\label{fig151}
\end{figure}
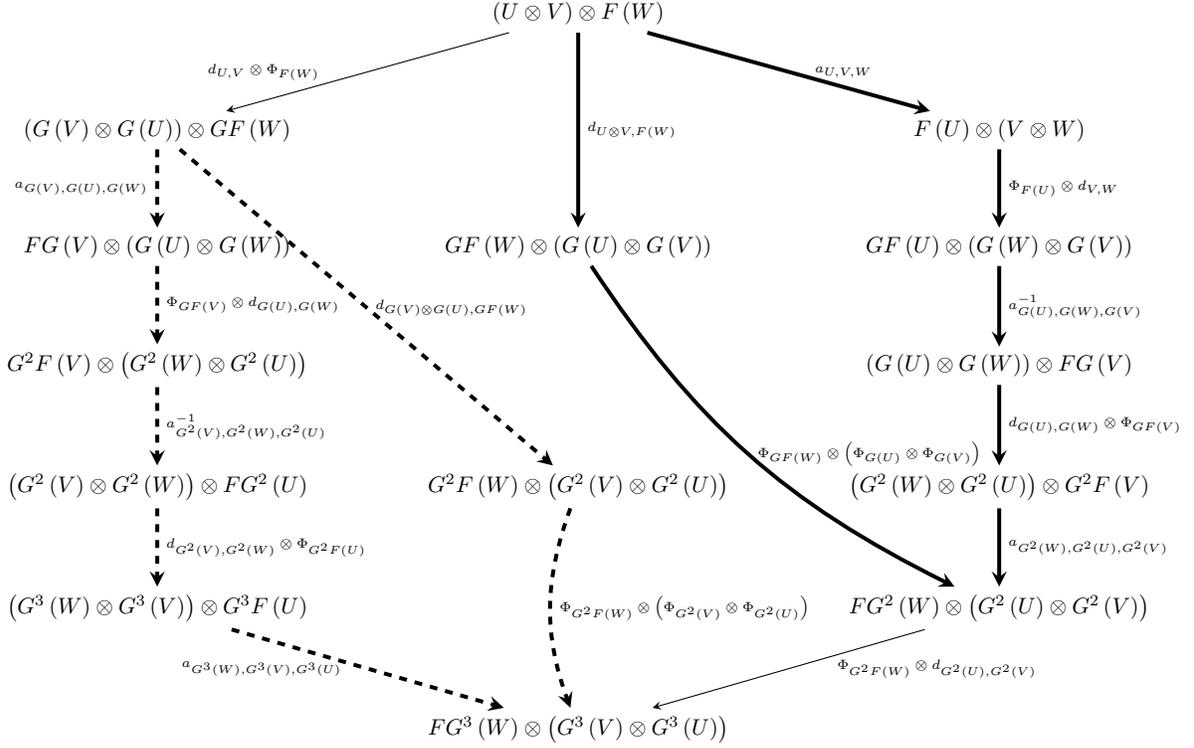
\end{center}

\begin{center}
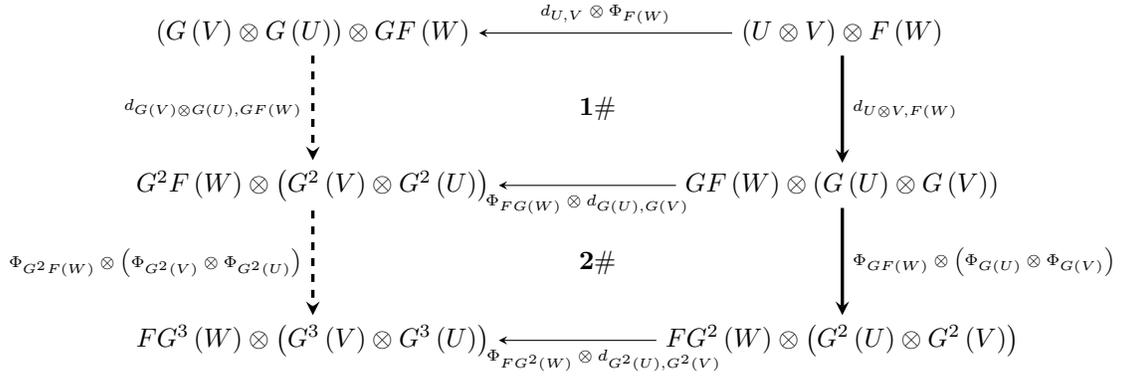
\begin{figure}
\begin{center}
\begin{tikzpicture}
  \matrix (m) [matrix of math nodes,row sep=4em,column sep=3em,minimum width=0.5em] 
     { \left(G\left(V\right)\otimes G\left(U\right)\right)\otimes GF\left(W\right) & & \left(U\otimes V\right)\otimes F\left(W\right) \\ 
		    G^{2}F\left(W\right)\otimes \left(G^{2}\left(V\right)\otimes G^{2}\left(U\right)\right) & & GF\left(W\right)\otimes \left(G\left(U\right)\otimes G\left(V\right)\right) \\
			  FG^{3}\left(W\right)\otimes\left(G^{3}\left(V\right)\otimes G^{3}\left(U\right)\right) & & FG^{2}\left(W\right)\otimes\left(G^{2}\left(U\right)\otimes G^{2}\left(V\right)\right) \\};
  \path[-stealth]
	  (m-1-3) edge node [above] {\tiny{$ d_{U,V}\otimes\Phi_{F\left(W\right)} $}} (m-1-1)
		        edge[line width=1.25pt] node [right] {\tiny{$d_{U\otimes V,F\left(W\right)}$}} (m-2-3)
		(m-1-1) edge[line width=1.25pt,dashed] node [left] {\tiny{$d_{G\left(V\right)\otimes G\left(U\right),GF\left(W\right)}$}} node [right] {$ \hskip 3.4cm \textbf{1}\# $} (m-2-1)
    (m-2-3) edge[line width=1.25pt] node [right] {\tiny{$\Phi_{GF\left(W\right)}\otimes\left(\Phi_{G\left(U\right)}\otimes\Phi_{G\left(V\right)}\right)$}} (m-3-3)
		        edge node [below] {\tiny{$\Phi_{FG\left(W\right)}\otimes d_{G\left(U\right),G\left(V\right)}$}} (m-2-1)
		(m-2-1) edge[line width=1.25pt,dashed] node [left] {\tiny{$\Phi_{G^{2}F\left(W\right)}\otimes\left(\Phi_{G^{2}\left(V\right)}\otimes\Phi_{G^{2}\left(U\right)}\right)$}} node [right] {$ \hskip 3.4cm \textbf{2}\# $} (m-3-1)
		(m-3-3) edge node [below] {\tiny{$ \hskip 7mm \Phi_{FG^{2}\left(W\right)}\otimes d_{G^{2}\left(U\right),G^{2}\left(V\right)}$}} (m-3-1);
\end{tikzpicture}
\end{center}
\caption{Final step in the proof of Proposition \ref{prop20}}
\label{fig157}
\end{figure}
\end{center}

We can give now the connection between hom-braided categories and weakly hom-braided categories. 
\begin{Prop}\label{prop6.10}
If $\mathscr{C}=\left(\mathcal{C},\otimes,F,G,a,\Phi,d\right)$ is a hom-braided category then $\mathscr{C}=\left(\mathcal{C},\otimes,F,G,a,\Phi,d\right)$ is a weakly hom-braided category. 
\label{prop25}
\end{Prop}
\begin{proof}
After comparing Definitions \ref{def7} and \ref{def87} for hom-braiding and weak hom-braiding, it suffices to show that the (H1) property implies the (wH1) property and the (H2) property implies the (wH2) property. 

We prove that the (H1) property implies the (wH1) property. Observe that the outermost perimeter of the diagram in Figure \ref{fig163}, indicated by a dashed line, is the commutative diagram of the (H1) property from Figure \ref{fig3}. In addition, observe that the innermost 7-gon of the diagram in Figure \ref{fig163}, indicated by the bold line, is the diagram of the (wH1) property as seen in Figure \ref{fig127}. The plan is to show that each of the square portions of this diagram, labeled $\textbf{id}_{1}$, $\textbf{id}_{2}$, $\textbf{id}_{3}$, $\textbf{1\#}$, $\textbf{2\#}$, $\textbf{3\#}$ and $\textbf{4\#}$ commute. Once this is established then the assumed commutativity of the dashed portion of the diagram in Figure \ref{fig163}, and the fact that $(\text{id}_U\otimes \text{id}_V)\otimes \text{id}_{F(W)}$  is an isomorphism, will imply the commutativity of the bold portion. That is, the (H1) property implies the (wH1) property.

Clearly, the square portions labeled $\textbf{id}_{1}$, $\textbf{id}_{2}$ and $\textbf{id}_{3}$ of the diagram in Figure \ref{fig163} commute. Additionally, both the square portions labeled $\textbf{1\#}$ and $\textbf{4\#}$ commute since $\otimes$ is a functor. More precisely, we have
\begin{eqnarray*}
d_{U,V}\otimes\Phi_{F\left(W\right)}&=&\left(\left(\text{id}_{G\left(V\right)}
\otimes\text{id}_{G\left(U\right)}\right)\otimes\Phi_{F\left(W\right)}\right)\left(d_{U,V}
\otimes\text{id}_{F\left(W\right)}\right)\\
&=&\left(d_{U,V}\otimes\Phi_{F\left(W\right)}\right)\left(\left(\text{id}_{U}
\otimes\text{id}_{V}\right)\otimes\text{id}_{F\left(W\right)}\right), 
\end{eqnarray*} 
\begin{eqnarray*}
\Phi_{FG\left(V\right)}\otimes\left(\Phi_{G\left(W\right)}\otimes\Phi_{G\left(U\right)}\right)&=& \left(\Phi_{FG\left(V\right)}\otimes\left(\Phi_{G\left(W\right)}
\otimes\text{id}_{G^{2}\left(U\right)}\right)\right)\left(\text{id}_{FG\left(V\right)}
\otimes\left(\text{id}_{G\left(W\right)}\otimes\Phi_{G\left(U\right)}\right)\right)\\
&=& \left(\Phi_{FG\left(V\right)}\otimes\left(\Phi_{G\left(W\right)}
\otimes\Phi_{G\left(U\right)}\right)\right)\left(\text{id}_{FG\left(V\right)}
\otimes\left(\text{id}_{G\left(W\right)}\otimes\text{id}_{G\left(U\right)}\right)\right).
\end{eqnarray*}  

Furthermore, the square portion labeled $\textbf{2\#}$ commutes since the hom-associativity constraint $a$ is natural and $F(\Phi_W)=\Phi_{F(W)}$. Finally, the square portion labeled $\textbf{3\#}$ commutes since the hom-commutativity constraint $d$ is natural  and $G(\Phi_W)=\Phi_{G(W)}$. Indeed, we have $\left(G(\Phi_{W})\otimes\text{id}_{G^{2}\left(U\right)}\right)\circ d_{G\left(U\right),W}=d_{G\left(U\right),G\left(W\right)}\circ\left(\text{id}_{G\left(U\right)}\otimes\Phi_{W}\right)$ by the naturality of $d$, and if we use that $\Phi_{G(W)}=G(\Phi_W)$ we have that
\begin{eqnarray*}
\Phi_{FG\left(V\right)}\otimes (\left(\Phi_{G\left(W\right)}\otimes\text{id}_{G^{2}\left(U\right)}\right)\circ d_{G\left(U\right),W}) &=& \Phi_{FG\left(V\right)}\otimes(d_{G\left(U\right),G\left(W\right)}\circ\left(\text{id}_{G\left(U\right)}\otimes\Phi_{W}\right)),
\end{eqnarray*}
which gives $\textbf{3\#}$, by using again the fact that $\otimes $ is a functor.

A similar argument shows that the (H2) property implies the (wH2) property.
\end{proof}
\begin{Rem}
The above proof (i.e. the diagram in Figure \ref{fig163}) also shows that the converse of Proposition \ref{prop25} is true if the natural transformation $\Phi$ is assumed to be an isomorphism.
\label{rem73}
\end{Rem}

We recall the following result (\cite{YA:HA2}, Theorem 4.4).
\begin{Prop}
Let $(H, m_H, \Delta _H, \alpha _H, \alpha _H, R)$ be a quasitriangular hom-bialgebra such that $(\alpha _H\otimes 
\alpha _H)(R)=R$ and $(M, \alpha _M)$ a left $H$-module. Then the linear map $B:M\otimes M\rightarrow M\otimes M$, 
$B(m\otimes m')=\sum _it_i\cdot m'\otimes s_i\cdot m$, for all $m, m'\in M$ (where we denoted as before 
$R=\sum _it_i\otimes s_i$) is a solution of the hom-Yang-Baxter equation with respect to $\alpha _M$. 
\end{Prop}

We want to obtain (a more general version of) this result as a consequence of the theory we developed. 
\begin{Prop}\label{corolus1}
Let $(H, m_H, \Delta _H, \alpha _H, \psi _H, R)$ be a quasitriangular hom-bialgebra, with notation 
$R=\sum _it_i\otimes s_i$, such that $(\alpha _H\otimes 
\alpha _H)(R)=R=(\psi _H\otimes \psi _H)(R)$. If $(M, \alpha _M)$ is a left $H$-module, then the linear map $B:M\otimes M\rightarrow M\otimes M$, 
$B(m\otimes m')=\sum _it_i\cdot m'\otimes s_i\cdot m$, for all $m, m'\in M$, is a solution of the hom-Yang-Baxter equation with respect to $\alpha _M$. 
\end{Prop} 
\begin{proof}
By Proposition \ref{prop2}, the linear map $d_{U, V}:U\otimes V\rightarrow G(V)\otimes G(U)$, defined by 
$d_{U, V}(u\otimes v)=\sum _it_i\cdot v\otimes s_i\cdot u$, is a hom-braiding. By Proposition \ref{prop6.10}, it is 
also a weak hom-braiding, so, by Proposition \ref{prop6.9}, the diagram in Figure \ref{fig59} is commutative. We write the 
diagram in Figure \ref{fig59} with $U=V=W=M$ and we note that, since $(\alpha _H\otimes 
\alpha _H)(R)=R$, we have $d_{M, M}=d_{G(M), G(M)}=d_{G^2(M), G^2(M)}=B$. So, the commutativity of the 
diagram in Figure \ref{fig59} reads 
$$\left(B\otimes\alpha_{M}\right)\circ\left(\alpha_{M}\otimes B\right)
\circ\left(B\otimes\alpha_{M}\right)=\left(\alpha_{M}\otimes B\right)\circ\left(B\otimes\alpha_{M}\right)
\circ\left(\alpha_{M}\otimes B\right),$$
which is exactly the hom-Yang-Baxter equation for $B$ with respect to $\alpha _M$. 
\end{proof}

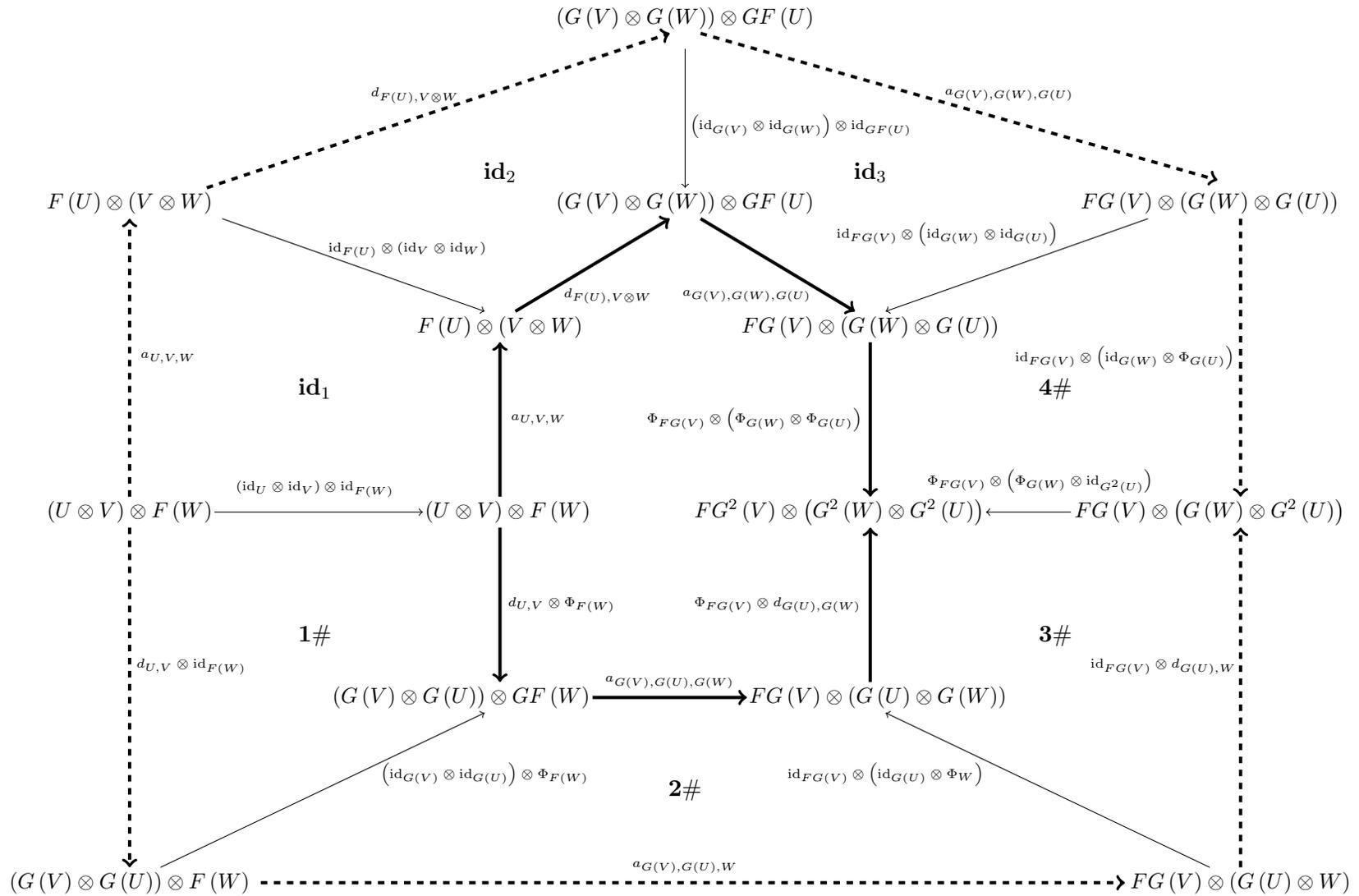
\begin{sidewaysfigure}
\vskip 165mm
\begin{center}
\begin{tikzpicture}
	\begin{pgfonlayer}{nodelayer}
		\node [style=none] (0) at (-2.875, -0) {$\left(U\otimes V\right)\otimes F\left(W\right)$};
		\node [style=none] (1) at (-9, -0) {$\left(U\otimes V\right)\otimes F\left(W\right)$};
		\node [style=none] (2) at (2.5, -0) {$FG^{2}\left(V\right)\otimes\left(G^{2}\left(W\right)\otimes G^{2}\left(U\right)\right)$};
		\node [style=none] (3) at (8.5, -0) {$FG\left(V\right)\otimes\left(G\left(W\right)\otimes G^{2}\left(U\right)\right)$};
		\node [style=none] (4) at (-3.625, -3) {$\left(G\left(V\right)\otimes G\left(U\right)\right)\otimes GF\left(W\right)$};
		\node [style=none] (5) at (3.125, -3) {$FG\left(V\right)\otimes\left(G\left(U\right)\otimes G\left(W\right)\right)$};
		\node [style=none] (6) at (3, 3) {$FG\left(V\right)\otimes\left(G\left(W\right)\otimes G\left(U\right)\right)$};
		\node [style=none] (7) at (-3, 3) {$F\left(U\right)\otimes\left(V\otimes W\right)$};
		\node [style=none] (8) at (0, 5) {$\left(G\left(V\right)\otimes G\left(W\right)\right)\otimes GF\left(U\right)$};
		\node [style=none] (9) at (-9, -6) {$\left(G\left(V\right)\otimes G\left(U\right)\right)\otimes F\left(W\right)$};
		\node [style=none] (10) at (9, -6) {$FG\left(V\right)\otimes\left(G\left(U\right)\otimes W\right)$};
		\node [style=none] (11) at (-9, 5) {$F\left(U\right)\otimes\left(V\otimes W\right)$};
		\node [style=none] (12) at (8.5, 5) {$FG\left(V\right)\otimes\left(G\left(W\right)\otimes G\left(U\right)\right)$};
		\node [style=none] (13) at (0, 8) {$\left(G\left(V\right)\otimes G\left(W\right)\right)\otimes GF\left(U\right)$};
		\node [style=none] (14) at (-7.75, 5.25) {};
		\node [style=none] (15) at (-0.25, 7.75) {};
		\node [style=none] (16) at (0.25, 7.75) {};
		\node [style=none] (17) at (8.625, 5.375) {};
		\node [style=none] (18) at (-9, 4.75) {};
		\node [style=none] (19) at (-9, 0.25) {};
		\node [style=none] (20) at (9, 4.75) {};
		\node [style=none] (21) at (9, 0.25) {};
		\node [style=none] (22) at (0, 7.5) {};
		\node [style=none] (23) at (0, 5.25) {};
		\node [style=none] (24) at (-7.5, 4.75) {};
		\node [style=none] (25) at (-3.25, 3.25) {};
		\node [style=none] (26) at (7.5, 4.75) {};
		\node [style=none] (27) at (3.25, 3.25) {};
		\node [style=none] (28) at (-0.25, 4.75) {};
		\node [style=none] (29) at (-2.75, 3.25) {};
		\node [style=none] (30) at (-3, 2.75) {};
		\node [style=none] (31) at (-3, 0.25) {};
		\node [style=none] (32) at (0.25, 4.75) {};
		\node [style=none] (33) at (2.75, 3.25) {};
		\node [style=none] (34) at (3, 2.75) {};
		\node [style=none] (35) at (3, 0.25) {};
		\node [style=none] (36) at (-7.625, -0) {};
		\node [style=none] (37) at (-4.25, -0) {};
		\node [style=none] (38) at (4.875, -0) {};
		\node [style=none] (39) at (6.25, -0) {};
		\node [style=none] (40) at (-3, -0.25) {};
		\node [style=none] (41) at (-3, -2.75) {};
		\node [style=none] (42) at (-1.5, -3) {};
		\node [style=none] (43) at (1, -3) {};
		\node [style=none] (44) at (3, -0.25) {};
		\node [style=none] (45) at (3, -2.75) {};
		\node [style=none] (46) at (-9, -5.75) {};
		\node [style=none] (47) at (-6.875, -6) {};
		\node [style=none] (48) at (7.125, -6) {};
		\node [style=none] (49) at (9, -5.75) {};
		\node [style=none] (50) at (8.5, -5.75) {};
		\node [style=none] (51) at (3.25, -3.25) {};
		\node [style=none] (52) at (-8.5, -5.75) {};
		\node [style=none] (53) at (-3.25, -3.25) {};
		\node [style=none] (54) at (-9, -0.25) {};
		\node [style=none] (55) at (9, -0.25) {};
		\node [style=none] (56) at (-6, 2) {\large{$\textbf{id}_{1}$}};
		\node [style=none] (57) at (6, 2) {\large{$\textbf{4}\#$}};
		\node [style=none] (58) at (-2.375, 1.5) {\tiny{$a_{U,V,W}$}};
		\node [style=none] (59) at (-2, -1.5) {\tiny{$d_{U,V}\otimes\Phi_{F\left(W\right)}$}};
		\node [style=none] (60) at (-0.25, -2.75) {\tiny{$a_{G\left(V\right),G\left(U\right),G\left(W\right)}$}};
		\node [style=none] (61) at (1.5, -1.5) {\tiny{$\Phi_{FG\left(V\right)}\otimes d_{G\left(U\right),G\left(W\right)}$}};
		\node [style=none] (62) at (1.125, 1.5) {\tiny{$\Phi_{FG\left(V\right)}\otimes\left(\Phi_{G\left(W\right)}\otimes\Phi_{G\left(U\right)}\right)$}};
		\node [style=none] (63) at (1, 3.5) {\tiny{$a_{G\left(V\right),G\left(W\right),G\left(U\right)}$}};
		\node [style=none] (64) at (-1.25, 3.5) {\tiny{$d_{F\left(U\right),V\otimes W}$}};
		\node [style=none] (65) at (-6, -2) {\large{$\textbf{1}\#$}};
		\node [style=none] (66) at (0, -4.5) {\large{$\textbf{2}\#$}};
		\node [style=none] (67) at (6, -2) {\large{$\textbf{3}\#$}};
		\node [style=none] (68) at (-3.25, -4.25) {\tiny{$\left(\text{id}_{G\left(V\right)}\otimes\text{id}_{G\left(U\right)}\right)\otimes\Phi_{F\left(W\right)}$}};
		\node [style=none] (69) at (3.25, -4.25) {\tiny{$\text{id}_{FG\left(V\right)}\otimes\left(\text{id}_{G\left(U\right)}\otimes\Phi_{W}\right)$}};
		\node [style=none] (70) at (5.75, 0.5) {\tiny{$\Phi_{FG\left(V\right)}\otimes\left(\Phi_{G\left(W\right)}\otimes\text{id}_{G^{2}\left(U\right)}\right)$}};
		\node [style=none] (71) at (-6, 0.375) {\tiny{$\left(\text{id}_{U}\otimes\text{id}_{V}\right)\otimes\text{id}_{F\left(W\right)}$}};
		\node [style=none] (72) at (0, -5.75) {\tiny{$a_{G\left(V\right),G\left(U\right),W}$}};
		\node [style=none] (73) at (-8, -2.5) {\tiny{$d_{U,V}\otimes\text{id}_{F\left(W\right)}$}};
		\node [style=none] (74) at (-8.375, 2.5) {\tiny{$a_{U,V,W}$}};
		\node [style=none] (75) at (7.125, 2.5) {\tiny{$\text{id}_{FG\left(V\right)}\otimes\left(\text{id}_{G\left(W\right)}\otimes\Phi_{G\left(U\right)}\right)$}};
		\node [style=none] (76) at (7.75, -2.5) {\tiny{$\text{id}_{FG\left(V\right)}\otimes d_{G\left(U\right),W}$}};
		\node [style=none] (77) at (-4.5, 4.25) {\tiny{$\text{id}_{F\left(U\right)}\otimes\left(\text{id}_{V}\otimes\text{id}_{W}\right)$}};
		\node [style=none] (78) at (4.25, 4.5) {\tiny{$\text{id}_{FG\left(V\right)}\otimes\left(\text{id}_{G\left(W\right)}\otimes\text{id}_{G\left(U\right)}\right)$}};
		\node [style=none] (79) at (1.875, 6.25) {\tiny{$\left(\text{id}_{G\left(V\right)}\otimes\text{id}_{G\left(W\right)}\right)\otimes\text{id}_{GF\left(U\right)}$}};
		\node [style=none] (80) at (-4.375, 6.75) {\tiny{$d_{F\left(U\right),V\otimes W}$}};
		\node [style=none] (81) at (5.25, 6.75) {\tiny{$a_{G\left(V\right),G\left(W\right),G\left(U\right)}$}};
		\node [style=none] (82) at (-3, 5.5) {\large{$\textbf{id}_{2}$}};
		\node [style=none] (83) at (3, 5.5) {\large{$\textbf{id}_{3}$}};
	\end{pgfonlayer}
	\begin{pgfonlayer}{edgelayer}
		\draw[->,line width=1.5pt,dashed] (14.center) to (15.center);
		\draw[->,line width=1.5pt,dashed] (16.center) to (17.center);
		\draw[->,line width=1.5pt,dashed] (19.center) to (18.center);
		\draw[->,line width=1.5pt,dashed] (20.center) to (21.center);
		\draw[->] (24.center) to (25.center);
		\draw[->] (22.center) to (23.center);
		\draw[->] (26.center) to (27.center);
		\draw[->,line width=1.5pt] (29.center) to (28.center);
		\draw[->,line width=1.5pt] (32.center) to (33.center);
		\draw[->,line width=1.5pt] (31.center) to (30.center);
		\draw[->,line width=1.5pt] (34.center) to (35.center);
		\draw[->] (36.center) to (37.center);
		\draw[->] (39.center) to (38.center);
		\draw[->,line width=1.5pt] (40.center) to (41.center);
		\draw[->,line width=1.5pt] (42.center) to (43.center);
		\draw[->,line width=1.5pt] (45.center) to (44.center);
		\draw[->,line width=1.5pt,dashed] (54.center) to (46.center);
		\draw[->] (52.center) to (53.center);
		\draw[->,line width=1.5pt,dashed] (47.center) to (48.center);
		\draw[->] (50.center) to (51.center);
		\draw[->,line width=1.5pt,dashed] (49.center) to (55.center);
	\end{pgfonlayer}
\end{tikzpicture}
\end{center}
\caption{The (H1) property implies the (wH1) property}
\label{fig163}
\end{sidewaysfigure}
\section{Yetter-Drinfeld modules}\label{sec7}
Throughout this section, $H=(H, m_H, \Delta _H, \alpha _H, \psi _H)$ will be a hom-bialgebra for which 
$\alpha _H=\psi _H$ and $\alpha _H$ is bijective. We recall the following concept and results from \cite{MP:HA}. 
\begin{Def}
Let $M$ be a $\h$-vector space and 
$\alpha _M:M\rightarrow M$ a $\h$-linear map such that $(M, \alpha _M)$ is a left $H$-module with 
action $H\ot M\rightarrow M$, $h\ot m\mapsto h\cdot m$ and a 
left $H$-comodule with coaction $M\rightarrow H\ot M$, $m\mapsto \sum m_{(-1)}\ot m_{(0)}$. 
Then $(M, \alpha _M)$ is called a (left-left) {\em Yetter-Drinfeld module} over $H$  if the 
following identity holds, for all $h\in H$, $m\in M$:
\begin{eqnarray}
&&\sum (h_{(1)}\cdot m)_{(-1)}\alpha _H^2(h_{(2)})\ot (h_{(1)}\cdot m)_{(0)}=\sum 
\alpha _H^2(h_{(1)})\alpha _H(m_{(-1)})\ot \alpha _H(h_{(2)})\cdot m_{(0)}. 
\label{homYD}
\end{eqnarray}

We denote by $_H^H\mathbb{YD}$ (respectively $_H^H{\mathcal YD}$) the category whose objects are 
Yetter-Drinfeld modules 
$(M, \alpha _M)$ over $H$ (respectively Yetter-Drinfeld modules 
$(M, \alpha _M)$ over $H$ with $\alpha _M$ bijective); the morphisms in each of these categories are morphisms 
of left $H$-modules and left $H$-comodules.
\end{Def}
\begin{Prop} 
Let 
$(M, \alpha _M)$ and $(N, \alpha _N)$ be two Yetter-Drinfeld modules over $H$, with notation 
as above, and define the $\h$-linear maps
\begin{eqnarray*}
&&H\ot (M\ot N)\rightarrow M\ot N, \;\;\;h\ot (m\ot n)\mapsto \sum h_{(1)}\cdot m\ot h_{(2)}\cdot n, \\
&&M\ot N\rightarrow H\ot (M\ot N), \;\;\;m\ot n\mapsto \sum \alpha _H^{-2}(m_{(-1)}n_{(-1)})\ot 
(m_{(0)}\ot n_{(0)}). 
\end{eqnarray*}
Then $(M\ot N, \alpha _M\ot \alpha _N)$ with these structures is a Yetter-Drinfeld module 
over $H$, 
denoted by $M\hot N$. 
\end{Prop}
\begin{Prop}\label{prop7.3}
$_H^H{\mathcal YD}$ is a quasi-braided category, with tensor product 
$\hot $ and associativity constraints and quasi-braiding defined, for 
$(M, \alpha _M)$, $(N, \alpha _N)$, $(P, \alpha _P)$  objects in 
$_H^H{\mathcal YD}$ by 
\begin{eqnarray*}
&&b_{M, N, P}:(M\hot N)\hot P\rightarrow M\hot (N\hot P), \;\;\;
b_{M, N, P}((m\ot n)\ot p)=\alpha _M^{-1}(m)\ot (n\ot \alpha _P(p)), \\
&&c_{M, N}:M\hot N\rightarrow N\hot M, \;\;\;c_{M, N}(m\ot n)=\sum 
\alpha _N^{-1}(\alpha _H^{-1}(m_{(-1)})\cdot n)\ot \alpha _M^{-1}(m_{(0)}).
\end{eqnarray*}
\end{Prop}
\begin{Prop} \label{Bmaps}
Let 
$(M, \alpha _M)$, $(N, \alpha _N)$$\;\in \;$$_H^H\mathbb{YD}$ and define the $\h$-linear map 
\begin{eqnarray}
&&B_{M, N}:M\ot N\rightarrow N\ot M, \;\;\;B_{M, N}(m\ot n)=\sum \alpha _H^{-1}(m_{(-1)})\cdot n
\ot m_{(0)}. \label{defB}
\end{eqnarray}
Then, we have $(\alpha _N\ot \alpha _M)\circ B_{M, N}=B_{M, N}\circ 
(\alpha _M\ot \alpha _N)$ and, 
if $(P, \alpha _P)$ is another Yetter-Drinfeld module over $H$, the maps $B_{-, -}$ 
satisfy the hom-Yang-Baxter equation
\begin{eqnarray}
&&(\alpha _P\ot B_{M, N})\circ (B_{M, P}\ot \alpha _N)\circ (\alpha _M\ot B_{N, P})
=(B_{N, P}\ot \alpha _M)\circ (\alpha _N\ot B_{M, P})
\circ (B_{M, N}\ot \alpha _P). \label{hYBeB}
\end{eqnarray}
\end{Prop}

Our aim now is to prove the following result. 
\begin{Prop}\label{prop7.5}
$_H^H\mathbb{YD}$ may be organized as a hom-braided category, with tensor product $\hot$ and hom-braiding 
the family of maps $B_{-,-}$ defined by (\ref{defB}).
\end{Prop}
\begin{proof}
We define the hom-associativity constraints $a_{M, N, P}$, the functors $F$ and $G$ and the 
natural transformation $\Phi $. Since any Yetter-Drinfeld module over $H$ is in particular a left $H$-module, 
$a$, $F$, $G$ and $\Phi $ will be, at the level of left $H$-modules, the ones defined in Proposition \ref{prop3.3} 
(since we assumed that $\psi _H=\alpha _H$, we actually have $F=G$ in this case). To simplify the notation, 
we will denote the elements in $F(M)$ by $m$ instead of $\overline{m}$.  
We need to extend the functor 
$F$ (=$G$) from left $H$-modules to Yetter-Drinfeld modules. If $(M, \alpha _M)$ is a Yetter-Drinfeld module, 
with left $H$-comodule structure $M\rightarrow H\ot M$, $m\mapsto \sum m_{(-1)}\ot m_{(0)}$, then $F(M)$ becomes a 
left $H$-comodule with structure $F(M)\rightarrow H\ot F(M)$, $m\mapsto \sum m_{<-1>}\ot m_{<0>}:=
\sum \alpha _H^{-1}(m_{(-1)})\ot m_{(0)}$, and moreover $F(M)$ with these structures becomes a 
Yetter-Drinfeld module over $H$ and, if $M, N$$\;\in \;$$_H^H\mathbb{YD}$, then 
$F(M\hot N)=F(M)\hot F(N)$ as objects in $_H^H\mathbb{YD}$. Then, $a_{M, N, P}:(M\hot N)\hot F(P)
\rightarrow F(M)\hot (N\hot P)$ and $\Phi _M:M\rightarrow F(M)$ are also morphisms of left $H$-comodules, 
and $F(\Phi _M)=\Phi _{F(M)}$, for $M, N, P$$\;\in \;$$_H^H\mathbb{YD}$. The proofs of these statements are 
similar to those in Proposition \ref{prop3.3} and will be omitted. The conclusion is that 
$(_H^H\mathbb{YD}, \hot , F, F, a, \Phi )$ is a hom-tensor category. 

We begin to prove that the family of maps $B_{-,-}$ is a hom-braiding. First, we need to prove that, for 
$(M, \alpha _M), (N, \alpha _N)$$\;\in \;$$_H^H\mathbb{YD}$, the maps $B_{M, N}$, regarded as maps 
from $M\hot N$ to $F(N)\hot F(M)$, are morphisms in $_H^H\mathbb{YD}$. We know from Proposition 
\ref{Bmaps} that $(\alpha _N\ot \alpha _M)\circ B_{M, N}=B_{M, N}\circ 
(\alpha _M\ot \alpha _N)$. Now we prove that $B_{M, N}(h\cdot (m\otimes n))=h\cdot _{\alpha }B_{M, N}(m\otimes n)$, 
for all $h\in H$, $m\in M$, $n\in N$. We compute:
\begin{eqnarray*}
B_{M, N}(h\cdot (m\otimes n))&=&B_{M, N}(\sum h_{(1)}\cdot m\otimes h_{(2)}\cdot n)\\
&=&\sum \alpha _H^{-1}((h_{(1)}\cdot m)_{(-1)})\cdot  (h_{(2)}\cdot n)\otimes (h_{(1)}\cdot m)_{(0)}\\
&=&\sum \alpha _H(\alpha _H^{-2}((h_{(1)}\cdot m)_{(-1)}))\cdot  (h_{(2)}\cdot n)\otimes (h_{(1)}\cdot m)_{(0)}\\
&\stackrel{\left(\ref{eq9}\right)}{=}&\sum (\alpha _H^{-2}((h_{(1)}\cdot m)_{(-1)})h_{(2)})\cdot \alpha _N(n)\otimes 
(h_{(1)}\cdot m)_{(0)}\\
&\stackrel{\left(\ref{homYD}\right)}{=}&\sum (h_{(1)}\alpha _H^{-1}(m_{(-1)}))\cdot \alpha _N(n)\otimes 
\alpha _H(h_{(2)})\cdot m_{(0)}\\
&\stackrel{\left(\ref{eq9}\right)}{=}&\sum \alpha _H(h_{(1)})\cdot (\alpha _H^{-1}(m_{(-1)})\cdot n)\otimes 
\alpha _H(h_{(2)})\cdot m_{(0)}\\
&=&\sum h_{(1)}\cdot _{\alpha }(\alpha _H^{-1}(m_{(-1)})\cdot n)\otimes h_{(2)}\cdot _{\alpha }m_{(0)}\\
&=&h\cdot _{\alpha }(\sum \alpha _H^{-1}(m_{(-1)})\cdot n\otimes m_{(0)})
=h\cdot _{\alpha }B_{M, N}(m\otimes n), \;\;\;q.e.d.
\end{eqnarray*}

We prove that $B_{M, N}$ is a morphism of left $H$-comodules, i.e. $(\text{id}_H\otimes B_{M, N})\circ 
\lambda _{M\hot N}=\lambda _{F(N)\hot F(M)}\circ B_{M, N}$. We compute, for $m\in M$, $n\in N$:\\[1mm]
${\;\;\;\;\;}$$\lambda _{F(N)\hot F(M)}(B_{M, N}(m\otimes n))$
\begin{eqnarray*}
&=&\lambda _{F(N)\hot F(M)}(\sum \alpha _H^{-1}(m_{(-1)})\cdot n\ot m_{(0)})\\
&=&\sum \alpha _H^{-2}((\alpha _H^{-1}(m_{(-1)})\cdot n)_{<-1>}(m_{(0)})_{<-1>})\ot 
(\alpha _H^{-1}(m_{(-1)})\cdot n)_{<0>}\ot (m_{(0)})_{<0>}\\
&=&\sum \alpha _H^{-3}((\alpha _H^{-1}(m_{(-1)})\cdot n)_{(-1)}(m_{(0)})_{(-1)})\ot 
(\alpha _H^{-1}(m_{(-1)})\cdot n)_{(0)}\ot (m_{(0)})_{(0)}\\
&\stackrel{\left(\ref{comodul2}\right)}{=}&\sum \alpha _H^{-3}((\alpha _H^{-2}((m_{(-1)})_{(1)})\cdot n)_{(-1)}(m_{(-1)})_{(2)})\ot 
(\alpha _H^{-2}((m_{(-1)})_{(1)})\cdot n)_{(0)}\ot \alpha _M(m_{(0)})\\
&=&\sum \alpha _H^{-3}((\alpha _H^{-2}((m_{(-1)})_{(1)})\cdot n)_{(-1)}
\alpha _H^2(\alpha _H^{-2}((m_{(-1)})_{(2)})))\ot 
(\alpha _H^{-2}((m_{(-1)})_{(1)})\cdot n)_{(0)}\ot \alpha _M(m_{(0)})\\
&=&\sum \alpha _H^{-3}((\alpha _H^{-2}(m_{(-1)})_{(1)}\cdot n)_{(-1)}
\alpha _H^2(\alpha _H^{-2}(m_{(-1)})_{(2)}))\ot 
(\alpha _H^{-2}(m_{(-1)})_{(1)}\cdot n)_{(0)}\ot \alpha _M(m_{(0)})\\
&\stackrel{\left(\ref{homYD}\right)}{=}&\sum \alpha _H^{-3}(\alpha _H^2(\alpha _H^{-2}(m_{(-1)})_{(1)})
\alpha _H(n_{(-1)}))\ot \alpha _H(\alpha _H^{-2}(m_{(-1)})_{(2)})\cdot n_{(0)}\ot \alpha _M(m_{(0)})\\
&=&\sum \alpha _H^{-3}((m_{(-1)})_{(1)})
\alpha _H^{-2}(n_{(-1)})\ot \alpha _H^{-1}((m_{(-1)})_{(2)})\cdot n_{(0)}\ot \alpha _M(m_{(0)})\\
&\stackrel{\left(\ref{comodul2}\right)}{=}&\sum \alpha _H^{-2}(m_{(-1)})
\alpha _H^{-2}(n_{(-1)})\ot \alpha _H^{-1}((m_{(0)})_{(-1)})\cdot n_{(0)}\ot (m_{(0)})_{(0)}\\
&=&\sum \alpha _H^{-2}(m_{(-1)}n_{(-1)})\ot \alpha _H^{-1}((m_{(0)})_{(-1)})\cdot n_{(0)}\ot (m_{(0)})_{(0)}\\
&=&(\text{id}_H\ot B_{M, N})(\sum \alpha _H^{-2}(m_{(-1)}n_{(-1)})\ot (m_{(0)}\ot n_{(0)}))
=(\text{id}_H\otimes B_{M, N})(\lambda _{M\hot N}(m\ot n)), \;\;\;q.e.d.
\end{eqnarray*}

The fact that $B_{-,-}$ is natural is easy to prove and left to the reader. We prove now that  $B_{-,-}$ 
satisfies (H1); to prove that it satisfies (H2) is similar and left to the reader. So, let 
$(U, \alpha _U), (V, \alpha _V), (W, \alpha _W)$$\;\in \;$$_H^H\mathbb{YD}$ and $u\in U$, 
$v\in V$, $w\in W$. We compute:\\[1mm]
${\;\;\;}$
$(\text{id}_{F^2(V)}\ot B_{F(U), W})\circ a_{F(V), F(U), W}\circ (B_{U, V}\ot \text{id}_{F(W)})((u\ot v)\ot w)$
\begin{eqnarray*}
&=&(\text{id}_{F^2(V)}\ot B_{F(U), W})\circ a_{F(V), F(U), W}((\sum \alpha _H^{-1}(u_{(-1)})\cdot v\ot u_{(0)})\ot w)\\
&=&(\text{id}_{F^2(V)}\ot B_{F(U), W})(\sum \alpha _H^{-1}(u_{(-1)})\cdot v\ot (u_{(0)}\ot w))\\
&=&\sum \alpha _H^{-1}(u_{(-1)})\cdot v\ot \alpha _H^{-1}((u_{(0)})_{<-1>})\cdot w\ot (u_{(0)})_{<0>}\\
&=&\sum \alpha _H^{-1}(u_{(-1)})\cdot v\ot \alpha _H^{-2}((u_{(0)})_{(-1)})\cdot w\ot (u_{(0)})_{(0)}\\
&\stackrel{\left(\ref{comodul2}\right)}{=}&\sum \alpha _H^{-2}((u_{(-1)})_{(1)})\cdot v\ot 
\alpha _H^{-2}((u_{(-1)})_{(2)})\cdot w\ot \alpha _U(u_{(0)})\\
&=&\sum \alpha _H^{-2}(u_{(-1)})_{(1)}\cdot v\ot 
\alpha _H^{-2}(u_{(-1)})_{(2)}\cdot w\ot \alpha _U(u_{(0)})\\
&=&(\text{id}_{F^2(V)}\ot (\text{id}_{F(W)}\ot \Phi _{F(U)}))\circ a_{F(V), F(W), F(U)}
(\sum (\alpha _H^{-2}(u_{(-1)})_{(1)}\cdot v\\
&&\ot 
\alpha _H^{-2}(u_{(-1)})_{(2)}\cdot w)\ot u_{(0)})\\
&=&(\text{id}_{F^2(V)}\ot (\text{id}_{F(W)}\ot \Phi _{F(U)}))\circ a_{F(V), F(W), F(U)}(\sum (\alpha _H^{-1}(u_{<-1>})\cdot 
(v\ot w)\ot u_{<0>})\\
&=&(\text{id}_{F^2(V)}\ot (\text{id}_{F(W)}\ot \Phi _{F(U)}))\circ a_{F(V), F(W), F(U)}\circ B_{F(U), V\ot W}\circ 
a_{U, V, W}((u\ot v)\ot w), \;\;\;q.e.d. 
\end{eqnarray*}

The only thing left to prove is that $F(B_{M, N})=B_{F(M), F(N)}$, for all $(M, \alpha _M), (N, \alpha _N)$$\;\in \;$$_H^H\mathbb{YD}$, that is $B_{M, N}=B_{F(M), F(N)}$. For $m\in M$, $n\in N$ we compute:
\begin{eqnarray*}
B_{F(M), F(N)}(m\ot n)&=&\sum \alpha _H^{-1}(m_{<-1>})\cdot _{\alpha }n\ot m_{<0>}
=\sum \alpha _H^{-2}(m_{(-1)})\cdot _{\alpha }n\ot m_{(0)}\\
&=&\sum \alpha _H^{-1}(m_{(-1)})\cdot n\ot m_{(0)}
=B_{M, N}(m\ot n), 
\end{eqnarray*}
finishing the proof.
\end{proof}
\begin{Rem}
Similarly to what we did in Proposition \ref{corolus1}, we can reobtain the relation (\ref{hYBeB}) in 
Proposition \ref{Bmaps} as a consequence 
of our theory. Indeed, since $B_{-,-}$ is a hom-braiding, it is also a weak hom-braiding, so the diagram in Figure \ref{fig59} 
is commutative. But since the functor $G$ ($=F$) acts as identity on morphisms and we know that 
$B_{G(M), G(N)}=G(B_{M, N})$, for all $(M, \alpha _M), (N, \alpha _N)$$\;\in \;$$_H^H\mathbb{YD}$, it 
follows that the commutativity of the diagram in Figure \ref{fig59} reduces to 
\begin{eqnarray*}
&&(\alpha _W\ot B_{U, V})\circ (B_{U, W}\ot \alpha _V)\circ (\alpha _U\ot B_{V, W})
=(B_{V, W}\ot \alpha _U)\circ (\alpha _V\ot B_{U, W})
\circ (B_{U, V}\ot \alpha _W), 
\end{eqnarray*}
which is exactly the hom-Yang-Baxter equation (\ref{hYBeB}). 
\end{Rem}

\section{Hom-Tensor Categories versus Tensor Categories}\label{sec8}

We show that under certain conditions one can associate a pre-tensor category to a  hom-tensor category. 

\begin{Prop} \label{prop3333d}
Let $\mathscr{C}=\left(\mathcal{C},\otimes,F,G,a,\Phi\right)$ be a hom-tensor category. Suppose that $\Theta:\text{id}_{\mathcal{C}}\rightarrow F$ is a natural isomorphism such that $\Theta_{M\otimes N}=
\Theta_{M}\otimes\Theta_{N}$ for all objects $M,N\in\mathcal{C}$. Consider  $b_{U,V,W}:\left(U\otimes V\right)\otimes W\rightarrow U\otimes\left(V\otimes W\right)$ defined for all objects $U,V,W\in\mathcal{C}$ by 
\begin{equation}
b_{U,V,W}=\left(\Theta^{-1}_{U}\otimes\text{id}_{V\otimes W}\right)\circ a_{U,V,W}\circ \left(\text{id}_{U\otimes V}\otimes\Theta_{W}\right), 
\label{eq3333c}
\end{equation}
see Figure \ref{fig3333a}. Then $b_{U,V,W}$ is an associativity constraint for the pre-tensor category $\left(\mathcal{C},\otimes,b\right)$.  

\end{Prop}

\begin{center}
\begin{figure}
\begin{center}
\begin{tikzpicture}
	\begin{pgfonlayer}{nodelayer}
		\node [style=none] (0) at (-3, 0.75) {};
		\node [style=none] (1) at (3, 1) {$U\otimes\left(V\otimes W\right)$};
		\node [style=none] (2) at (3, -1) {$F\left(U\right)\otimes\left(V\otimes W\right)$};
		\node [style=none] (3) at (-3, -1) {$\left(U\otimes V\right)\otimes F\left(W\right)$};
		\node [style=none] (4) at (-2, 1) {};
		\node [style=none] (5) at (1.875, 1) {};
		\node [style=none] (6) at (-3, 1) {$\left(U\otimes V\right)\otimes W$};
		\node [style=none] (7) at (-3, -0.75) {};
		\node [style=none] (8) at (-1.625, -1) {};
		\node [style=none] (9) at (1.625, -1) {};
		\node [style=none] (10) at (2.875, 0.75) {};
		\node [style=none] (11) at (3.125, 0.75) {};
		\node [style=none] (12) at (2.875, -0.75) {};
		\node [style=none] (13) at (3.125, -0.75) {};
		\node [style=none] (14) at (0, 1.125) {\tiny{$b_{U,V,W}$}};
		\node [style=none] (16) at (0, -0.875) {\tiny{$a_{U,V,W}$}};
		\node [style=none] (17) at (-3.875, -0) {\tiny{$\text{id}_{U\otimes V}\otimes\Theta_{W}$}};
		\node [style=none] (18) at (2.1, -0) {\tiny{$\Theta_{U}\otimes\text{id}_{V\otimes W} ~\ $}};
		\node [style=none] (19) at (4, -0) {\tiny{$ ~\ \Theta^{-1}_{U}\otimes\text{id}_{V\otimes W}$}};
	\end{pgfonlayer}
	\begin{pgfonlayer}{edgelayer}
		\draw [->] (4.center) to (5.center);
		\draw [->] (0.center) to (7.center);
		\draw [->] (8.center) to (9.center);
		\draw [->] (10.center) to (12.center);
		\draw [->,dashed] (13.center) to (11.center);
	\end{pgfonlayer}
\end{tikzpicture}
\end{center}
\caption{The definition of $b_{U,V,W}:\left(U\otimes V\right)\otimes W\rightarrow U\otimes\left(V\otimes W\right)$}
\label{fig3333a}
\end{figure}
\end{center}
\begin{proof}
Notice that $b_{U,V,W}$ is a natural isomorphism being a composition of natural isomorphisms. So we only need to check that $b_{U,V,W}$ satisfies the Pentagon axiom of an associativity constraint.  

Consider the diagram in Figure \ref{fig3333}. Observe that the inner bold portion of this diagram is the  "Pentagon" axiom of the hom-associativity constraint $a_{U,V,W}$ and that the outer dashed portion of this diagram is  the Pentagon axiom of the mapping $b_{U,V,W}$. We will show that commutativity of the bold portion of the diagram implies the commutativity of the dashed portion of the diagram. This will be done by proving that each square portion of the diagram labeled $\textbf{1}\#$, $\textbf{2}\#$, $\textbf{3}\#$, $\textbf{4}\#$ and $\textbf{5}\#$ commutes for all objects $U,V,W,X\in\mathcal{C}$.

We begin with the square portion of Figure \ref{fig3333} labeled $\textbf{1}\#$. We have: \\[2mm]
${\;\;\;\;\;\;\;\;\;\;\;\;\;}$$\left(\left(\Theta_{U}\otimes\text{id}_{V\otimes W}\right)\otimes\left(F\left(\Theta_{X}\right)\circ\Theta_{X}\right)\right)\circ\left(b_{U,V,W}\otimes\text{id}_{X}\right)$
\begin{eqnarray*}
&=& \left(\left(\Theta_{U}\otimes\text{id}_{V\otimes W}\right)\circ b_{U,V,W}\right)\otimes\left(\left(F\left(\Theta_{X}\right)\circ\Theta_{X}\right)\circ \text{id}_{X}\right)\\
&\stackrel{(Fig. \ref{fig3333a})}{=}& \left(a_{U,V,W}\circ\left(\text{id}_{U\otimes V}\otimes\Theta_{W}\right)\right)\otimes\left(\text{id}_{F^{2}\left(X\right)}\circ\left(F\left(\Theta_{X}\right)
\circ\Theta_{X}\right)\right)\\
&=& \left(a_{U,V,W}\otimes\text{id}_{F^{2}\left(X\right)}\right)\circ\left(\left(\text{id}_{U\otimes V}\otimes\Theta_{W}\right)\otimes\left(F\left(\Theta_{X}\right)\circ\Theta_{X}\right)\right).
\end{eqnarray*}
The first and third equality follow  from the functoriality of the tensor product. The second equality  follows from  the definition of $b_{U,V,W}$ (see Figure  \ref{fig3333a}). 
\begin{center}
\begin{figure}
\begin{center}
\begin{tikzpicture}[scale=0.6, every node/.style={scale=0.6}] 
  \matrix (m) [matrix of math nodes,row sep=2.5em,column sep=2em,minimum width=2em] 
     { \left(U\otimes\left(V\otimes W\right)\right)\otimes X & & & & U\otimes\left(\left(V\otimes W\right)\otimes X\right)\\ 
										& & \textbf{2}\# & &\\
		 & \left(F\left(U\right)\otimes\left(V\otimes W\right)\right)\otimes F^{2}\left(X\right) & & F^{2}\left(U\right)\otimes\left(\left(V\otimes W\right)\otimes F\left(X\right)\right) &  \\
											     & \left(\left(U\otimes V\right)\otimes F\left(W\right)\right)\otimes F^{2}\left(X\right) & & F^{2}\left(U\right)\otimes\left(F\left(V\right)\otimes \left(W\otimes X\right)\right)& \\
       \left(\left(U\otimes V\right)\otimes W\right)\otimes X & \hskip 0.5in \textbf{4}\# & \stackrel{\stackrel{\mbox{$\left(F\left(U\right)\otimes F\left(V\right)\right)\otimes F\left(W\otimes X\right)$}}{\mid\mid}}{F\left(U\otimes V\right)\otimes\left(F\left(W\right)\otimes F\left(X\right)\right)} & \textbf{5}\# \hskip 0.5in & U\otimes\left(V\otimes \left(W\otimes X\right)\right)\\
			                             & & & &\\
			                        & & \left(U\otimes V\right)\otimes \left(W\otimes X\right) & & \\
														 };
  \path[-stealth]
	  (m-1-1) edge[dashed] node [above] {\tiny{$b_{U,V\otimes W,X}$}} (m-1-5)
		        edge node [right] {\tiny{$ ~\ \left(\Theta_{U}\otimes\text{id}_{V\otimes W}\right)\otimes\left(F\left(\Theta_{X}\right)\circ\Theta_{X}\right) $}} (m-3-2)
		(m-1-5) edge node [left] {\tiny{$\left(F\left(\Theta_{U}\right)\circ\Theta_{U}\right)\otimes\left(\text{id}_{V\otimes W}\otimes \Theta_{X}\right)$}} (m-3-4)
		        edge[dashed] node [left] {$\textbf{3}\# \hskip 0.7in $} node [right] {\tiny{$ \text{id}_{U}\otimes b_{V,W,X}$}} (m-5-5)
		(m-3-2) edge[thick] node [above] {\tiny{$ a_{F\left(U\right),V\otimes W,F\left(X\right)}$}} (m-3-4)
    (m-3-4) edge[thick] node [left] {\tiny{$ \text{id}_{F^{2}\left(U\right)}\otimes a_{V,W,X}$}} (m-4-4)
		(m-4-2) edge[thick] node [right] {\tiny{$ a_{U,V,W}\otimes \text{id}_{F^{2}\left(X\right)}$}} (m-3-2)
		        edge[thick] node [right] {\tiny{$a_{U\otimes V,F\left(W\right),F\left(X\right)} $}} (m-5-3)
		(m-5-1) edge[dashed] node [left] {\tiny{$b_{U,V,W}\otimes\text{id}_{X}$}} node [right] {$ \hskip0.7in \textbf{1}\# $} (m-1-1)
		        edge node [right] {\tiny{$ ~\ ~\ \left(\text{id}_{U\otimes V}\otimes\Theta_{W}\right)\otimes\left(F\left(\Theta_{X}\right)\circ\Theta_{X}\right) $}} (m-4-2)
						edge[dashed] node [left] {\tiny{$b_{U\otimes V,W,X} ~\ ~\ $}} (m-7-3)
		(m-5-3) edge[thick] node [left] {\tiny{$a_{F\left(U\right),F\left(V\right),W\otimes X} ~\ $}} (m-4-4)
		(m-5-5) edge node [left] {\tiny{$\left(F\left(\Theta_{U}\right)\circ\Theta_{U}\right)\otimes\left(\Theta_{V}\otimes\text{id}_{W\otimes X}\right) ~\ ~\ $}} (m-4-4)
		(m-7-3) edge node [right] {\tiny{$ \left(\Theta_{U}\otimes\Theta_{V}\right)\otimes\left(\Theta_{W}\otimes\Theta_{X}\right)$}} (m-5-3)
		        edge[dashed] node [right] {\tiny{$ \hskip5mm b_{U,V,W\otimes X} $}} (m-5-5);
\end{tikzpicture}
\end{center}
\caption{Hom-associativity for $a$ implies associativity for $b$}
\label{fig3333}
\end{figure}
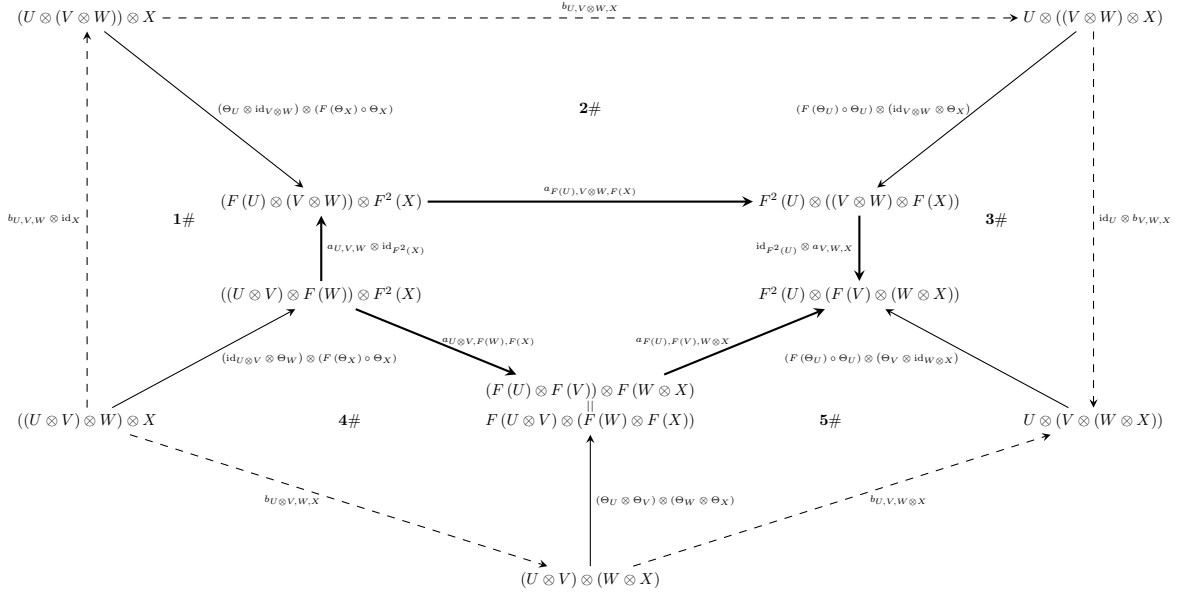
\end{center}

\begin{center}
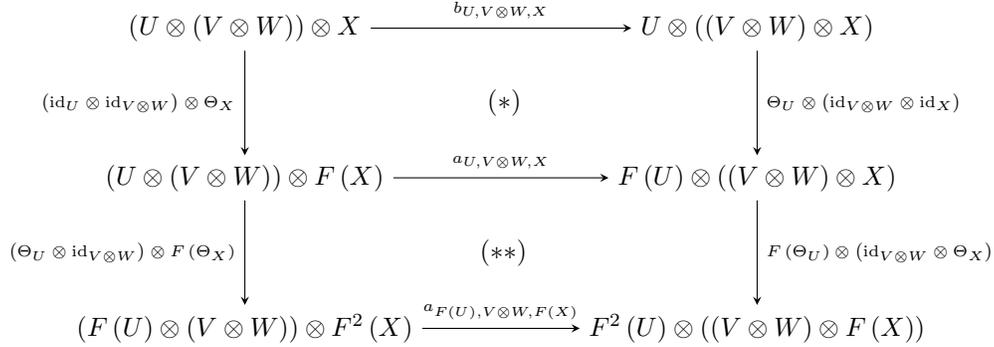
\begin{figure}
\begin{center}
\begin{tikzpicture}
  \matrix (m) [matrix of math nodes,row sep=4em,column sep=3em,minimum width=0.5em] 
     { \left(U\otimes \left(V\otimes W\right)\right)\otimes X & & U\otimes \left(\left(V\otimes W\right)\otimes X\right) \\ 
		    \left(U\otimes \left(V\otimes W\right)\right)\otimes F\left(X\right) & & F\left(U\right)\otimes \left(\left(V\otimes W\right)\otimes X\right) \\
			  \left(F\left(U\right)\otimes\left(V\otimes W\right)\right)\otimes F^{2}\left(X\right) & & F^{2}\left(U\right)\otimes\left(\left(V\otimes W\right)\otimes F\left(X\right)\right) \\};
  \path[-stealth]
	  (m-1-1) edge node [above] {\tiny{$ b_{U,V\otimes W,X} $}} (m-1-3)
		(m-1-3) edge node [right] {\tiny{$\Theta_{U}\otimes\left(\text{id}_{V\otimes W}\otimes\text{id}_{X}\right)$}} (m-2-3)
		(m-1-1) edge node [left] {\tiny{$\left(\text{id}_{U}\otimes\text{id}_{V\otimes W}\right)\otimes\Theta_{X}$}} node [right] {$\hskip3.1cm \left(\ast\right) $} (m-2-1)
    (m-2-3) edge node [right] {\tiny{$F\left(\Theta_{U}\right)\otimes\left(\text{id}_{V\otimes W}\otimes\Theta_{X}\right)$}} (m-3-3)
		(m-2-1) edge node [left] {\tiny{$\left(\Theta_{U}\otimes\text{id}_{V\otimes W}\right)\otimes F\left(\Theta_{X}\right)$}} node [right] {$\hskip3cm \left(\ast\ast\right) $} (m-3-1)
		        edge node [above] {\tiny{$ a_{U,V\otimes W,X}$}} (m-2-3)
		(m-3-1) edge node [above] {\tiny{$ a_{F\left(U\right),V\otimes W,F\left(X\right)}$}} (m-3-3);
\end{tikzpicture} 
\end{center}
\caption{The diagram for  $\textbf{2}\#$}
\label{fig3333b}
\end{figure}
\end{center}
Next we consider the square portion of Figure \ref{fig3333} labeled $\textbf{2}\#$. Notice that the  diagram  in Figure \ref{fig3333b} commutes for all objects $U,V,W,X\in\mathcal{C}$. Indeed, the portion of the diagram labeled $\left(\ast\right)$ commutes by  the definition of $b_{U,V\otimes W,X}$,  and the portion of the diagram labeled $\left(\ast\ast\right)$ commutes by the naturality of $a_{U,V\otimes W,X}$. 

Now we consider the square portion of Figure \ref{fig3333} labeled $\textbf{3}\#$. Computing we get that\\[2mm]
${\;\;\;\;\;\;\;\;\;\;\;\;\;\;\;\;\;}$
$\left(\left(F\left(\Theta_{U}\right)\circ\Theta_{U}\right)\otimes\left(\Theta_{V}\otimes\text{id}_{W\otimes X}\right)\right)\circ\left(\text{id}_{U}\otimes b_{V,W,X}\right)$
\begin{eqnarray*}
&=& \left(F\left(\Theta_{U}\right)\circ\Theta_{U}\right)\otimes \left(\left(\Theta_{V}\otimes\text{id}_{W\otimes X}\right)\circ b_{V,W,X}\right)\\
&\stackrel{(Fig. \ref{fig3333a} )}{=} &\left(F\left(\Theta_{U}\right)\circ\Theta_{U}\right)\otimes\left(a_{V,W,X}\circ\left(\text{id}_{V\otimes W}
\otimes\Theta_{X}\right)\right)\\
&=& \left(\text{id}_{F^{2}\left(U\right)}\otimes a_{V,W,X}\right)\circ \left(\left(F\left(\Theta_{U}\right)\circ\Theta_{U}\right)\otimes\left(\text{id}_{V\otimes W}\otimes\Theta_{X}\right)\right).
\end{eqnarray*}
The first and third equality are consequence of the functoriality of the tensor product, while the second equality follows from the definition of $b_{V,W,X}$. 

Next we consider the square portion of Figure \ref{fig3333} labeled $\textbf{4}\#$. Notice that the diagram  in Figure \ref{fig3333g} 
commutes for all objects $U,V,W,X\in\mathcal{C}$. Indeed, the portion of the diagram labeled $\left(\bullet\right)$ commutes by the definition of $b_{U\otimes V,W,X}$ and the portion of the diagram labeled $\left(\bullet\bullet\right)$ commutes by the naturality of $a_{U\otimes V,W,X}$. 
\begin{center}
\begin{figure}
\begin{center}
\begin{tikzpicture}
  \matrix (m) [matrix of math nodes,row sep=4em,column sep=3em,minimum width=0.5em] 
     { \left(\left(U\otimes V\right)\otimes W\right)\otimes X & & \left(U\otimes V\right)\otimes\left(W\otimes X\right) \\ 
		    \left(\left(U\otimes V\right)\otimes W\right)\otimes F\left(X\right) & & F\left(U\otimes V\right)\otimes\left(W\otimes X\right) \\
			  \left(\left(U\otimes V\right)\otimes F\left(W\right)\right)\otimes F^{2}\left(X\right) & & F\left(U\otimes V\right)\otimes\left(F\left(W\right)\otimes F\left(X\right)\right) \\};
  \path[-stealth]
	  (m-1-1) edge node [above] {\tiny{$b_{U\otimes V,W,X} $}} (m-1-3)
		(m-1-3) edge node [right] {\tiny{$\Theta_{U\otimes V}\otimes\text{id}_{W\otimes X}$}} (m-2-3)
		(m-1-1) edge node [left] {\tiny{$\left(\text{id}_{U\otimes V}\otimes\text{id}_{W}\right)\otimes\Theta_{X}$}} node [right] {$\hskip3.1cm \left(\bullet\right) $} (m-2-1)
    (m-2-3) edge node [right] {\tiny{$\text{id}_{F\left(U\otimes V\right)}\otimes\left(\Theta_{W}\otimes\Theta_{X}\right)$}} (m-3-3)
		(m-2-1) edge node [left] {\tiny{$\left(\text{id}_{U\otimes V}\otimes\Theta_{W}\right)\otimes F\left(\Theta_{X}\right)$}} node [right] {$\hskip3cm \left(\bullet\bullet\right) $} (m-3-1)
		        edge node [above] {\tiny{$a_{U\otimes V,W,X}$}} (m-2-3)
		(m-3-1) edge node [above] {\tiny{$a_{U\otimes V,F\left(W\right),F\left(X\right)}$}} (m-3-3);
\end{tikzpicture} 
\end{center}
\caption{The diagram for  $\textbf{4}\#$}
\label{fig3333g}
\end{figure}
\end{center}

\begin{center}
\begin{figure}
\begin{center}
\begin{tikzpicture}
  \matrix (m) [matrix of math nodes,row sep=4em,column sep=3em,minimum width=0.5em] 
     { \left(U\otimes V\right)\otimes \left(W\otimes X\right) & & U\otimes\left(V\otimes\left(W\otimes X\right)\right) \\ 
		    \left(U\otimes V\right)\otimes F\left(W\otimes X\right) & & F\left(U\right)\otimes \left(V\otimes\left(W\otimes X\right)\right) \\
			  \left(F\left(U\right)\otimes F\left(V\right)\right)\otimes F\left(W\otimes X\right) & & F^{2}\left(U\right)\otimes\left(F\left(V\right)\otimes\left(W\otimes X\right)\right) \\};
  \path[-stealth]
	  (m-1-1) edge node [above] {\tiny{$ b_{U,V,W\otimes X} $}} (m-1-3)
		(m-1-3) edge node [right] {\tiny{$\Theta_{U}\otimes\left(\text{id}_{V}\otimes\text{id}_{W\otimes X}\right)$}} (m-2-3)
		(m-1-1) edge node [left] {\tiny{$\text{id}_{U\otimes V}\otimes\Theta_{W\otimes X}$}} node [right] {$\hskip3.1cm \left(\diamond\right) $} (m-2-1)
    (m-2-3) edge node [right] {\tiny{$F\left(\Theta_{U}\right)\otimes\left(\Theta_{V}\otimes\text{id}_{W\otimes X}\right)$}} (m-3-3)
		(m-2-1) edge node [left] {\tiny{$\left(\Theta_{U}\otimes\Theta_{V}\right)\otimes\text{id}_{F\left(W\otimes X\right)}$}} node [right] {$\hskip3cm \left(\diamond\diamond\right) $} (m-3-1)
		        edge node [above] {\tiny{$a_{U,V,W\otimes X}$}} (m-2-3)
		(m-3-1) edge node [above] {\tiny{$ a_{F\left(U\right),F\left(V\right),W\otimes X}$}} (m-3-3);
\end{tikzpicture} 
\end{center}
\caption{The  diagram for  $\textbf{5}\#$}
\label{fig3333h}
\end{figure}
\end{center}

Notice that in order to glue together $\textbf{4}\#$ and  $\textbf{5}\#$ we use that $\Theta_{U\otimes V}=\Theta_U\otimes \Theta_V$, $\Theta_{W\otimes X}=\Theta_W\otimes \Theta_X$ and the functoriality of $\otimes$.  

Finally we consider the square portion of Figure \ref{fig3333} labeled $\textbf{5}\#$. Notice that the  diagram  in Figure \ref{fig3333h} commutes for all objects $U,V,W,X\in\mathcal{C}$. The portion of the diagram labeled $\left(\diamond\right)$ commutes by the definition of $b_{U,V,W\otimes X}$ and the portion of the diagram labeled $\left(\diamond\diamond\right)$ commutes by the naturality of $a_{U,V,W\otimes X}$. 
\end{proof}


Next we turn our attention to the relation between  hom-braided and quasi-braided categories. 

\begin{Prop} \label{prop9999b}
Let $\mathscr{C}=\left(\mathcal{C},\otimes,F,G,a,\Phi,d\right)$ be a hom-braided category, $\Theta$ and $b_{U,V,W}$ as in Proposition \ref{prop3333d}. Suppose that $\Phi$  is a natural isomorphism and $G(\Theta_U)=\Theta_{G(U)}$. Define $c_{U,V}:U\otimes V\rightarrow V\otimes U$, 
\begin{equation}
c_{U,V}=\left(\Phi^{-1}_{V}\otimes\Phi^{-1}_{U}\right)\circ d_{U,V}
\label{eq9999d}
\end{equation}
for all objects $U,V\in\mathcal{C}$. Then $c_{U,V}$ is a quasi-braiding for the quasi-braided category $\left(\mathcal{C},\otimes,b,c\right)$. 

\end{Prop}

\begin{center}
\begin{figure}[!ht]
\begin{center}
\begin{tikzpicture}
	\begin{pgfonlayer}{nodelayer}
		\node [style=none] (0) at (-2.125, 0.75) {};
		\node [style=none] (1) at (3, 1) {$V\otimes U$};
		\node [style=none] (2) at (3, -1) {$G\left(V\right)\otimes G\left(U\right)$};
		\node [style=none] (3) at (-2.125, -1) {$U\otimes V$};
		\node [style=none] (4) at (-1.625, 1) {};
		\node [style=none] (5) at (2.375, 1) {};
		\node [style=none] (6) at (-2.125, 1) {$U\otimes V$};
		\node [style=none] (7) at (-2.125, -0.75) {};
		\node [style=none] (8) at (-1.625, -1) {};
		\node [style=none] (9) at (1.75, -1) {};
		\node [style=none] (10) at (2.875, 0.75) {};
		\node [style=none] (11) at (3.125, 0.75) {};
		\node [style=none] (12) at (2.875, -0.75) {};
		\node [style=none] (13) at (3.125, -0.75) {};
		\node [style=none] (14) at (0, 1.125) {\tiny{$c_{U,V}$}};
		\node [style=none] (16) at (0, -0.875) {\tiny{$d_{U,V}$}};
		\node [style=none] (17) at (-2.6, -0) {\tiny{$\text{id}_{U\otimes V}$}};
		\node [style=none] (18) at (2.25, -0) {\tiny{$\Phi_{V}\otimes\Phi_{U}$}};
		\node [style=none] (19) at (3.875, -0) {\tiny{$\Phi^{-1}_{V}\otimes\Phi^{-1}_{U}$}};
	\end{pgfonlayer}
	\begin{pgfonlayer}{edgelayer}
		\draw [->] (4.center) to (5.center);
		\draw [->] (0.center) to (7.center);
		\draw [->] (8.center) to (9.center);
		\draw [->] (10.center) to (12.center);
		\draw [->,dashed] (13.center) to (11.center);
	\end{pgfonlayer}
\end{tikzpicture}
\end{center}
\caption{Definition of $c_{U,V}:U\otimes V\rightarrow V\otimes U$}
\label{fig9999a}
\end{figure}
\end{center}

\begin{proof}
 Being a composition of natural morphisms, $c_{U,V}$ is a natural morphism. 
 So we only need to check that $c_{U,V}$ satisfies the Hexagon axiom of a braiding. We check the first Hexagon axiom for $c_{U,V}$.

Consider the diagram in Figure \ref{fig9999}. Observe that the inner bold portion of this diagram is the  (H1) property of the hom-braiding $d_{U,V}$ and the outer dashed portion of this diagram is  the first Hexagon axiom property for $c_{U,V}$. We will show that commutativity of the bold portion of the diagram implies the commutativity of the dashed portion of the diagram. This will be done by proving that each square portion of the diagram labeled $\textbf{1}\#$, $\textbf{2}\#$, $\textbf{3}\#$, $\textbf{4}\#$, $\textbf{5}\#$, $\textbf{6}\#$ and $\textbf{7}\#$ commutes for all objects $U,V,W\in\mathcal{C}$.
\begin{sidewaysfigure}
\vspace{165mm}
\begin{center}
\begin{tikzpicture}
	\begin{pgfonlayer}{nodelayer}
		\node [style=none] (0) at (-4, -0) {$\left(U\otimes V\right)\otimes F\left(W\right)$};
		\node [style=none] (1) at (-4, 0.2500001) {};
		\node [style=none] (2) at (-4, -0.2500001) {};
		\node [style=none] (3) at (-4, 2.75) {};
		\node [style=none] (4) at (-4, 3) {$\left(G\left(V\right)\otimes G\left(U\right)\right)\otimes F\left(W\right)$};
		\node [style=none] (5) at (-2, 3) {};
		\node [style=none] (6) at (1.625, 3) {};
		\node [style=none] (7) at (3.5, 3) {$FG\left(V\right)\otimes \left(G\left(U\right)\otimes W\right)$};
		\node [style=none] (8) at (4, 2.75) {};
		\node [style=none] (9) at (4, 0.2500001) {};
		\node [style=none] (10) at (3.5, -0) {$FG\left(V\right)\otimes \left(G\left(W\right)\otimes G^{2}\left(U\right)\right)$};
		\node [style=none] (11) at (4, -0.2500001) {};
		\node [style=none] (12) at (0, -5.75) {};
		\node [style=none] (13) at (-0.25, -4.5) {};
		\node [style=none] (14) at (0.25, -4.5) {};
		\node [style=none] (15) at (-4, -2.75) {};
		\node [style=none] (16) at (-4, -3) {$F\left(U\right)\otimes\left(V\otimes W\right)$};
		\node [style=none] (17) at (-3.75, -3.25) {};
		\node [style=none] (18) at (4, -2.75) {};
		\node [style=none] (19) at (4, -3) {$FG\left(V\right)\otimes \left(G\left(W\right)\otimes G\left(U\right)\right)$};
		\node [style=none] (20) at (3.75, -3.25) {};
		\node [style=none] (21) at (-10, 0.25) {};
		\node [style=none] (22) at (-10, -0) {$\left(U\otimes V\right)\otimes W$};
		\node [style=none] (23) at (-10, -0.25) {};
		\node [style=none] (24) at (0, -8.5) {$\left(V\otimes W\right)\otimes U$};
		\node [style=none] (25) at (-0.2499996, -8.25) {};
		\node [style=none] (26) at (0.2499996, -8.25) {};
		\node [style=none] (27) at (-9.75, -5.25) {};
		\node [style=none] (28) at (-10, -5) {$U\otimes\left(V\otimes W\right)$};
		\node [style=none] (29) at (-10, -4.75) {};
		\node [style=none] (30) at (-10, 5.75) {};
		\node [style=none] (31) at (-10, 6) {$\left(V\otimes U\right)\otimes W$};
		\node [style=none] (32) at (-8.875, 6) {};
		\node [style=none] (33) at (8.875, 6) {};
		\node [style=none] (34) at (10, 6) {$V\otimes\left(U\otimes W\right)$};
		\node [style=none] (35) at (10, 5.75) {};
		\node [style=none] (36) at (10, 0.25) {};
		\node [style=none] (37) at (10, -0.25) {};
		\node [style=none] (38) at (10, -0) {$V\otimes\left(W\otimes U\right)$};
		\node [style=none] (39) at (10, -4.75) {};
		\node [style=none] (40) at (10, -5) {$V\otimes\left(W\otimes U\right)$};
		\node [style=none] (41) at (9.75, -5.25) {};
		\node [style=none] (42) at (0, -5.25) {};
		\node [style=none] (43) at (0, -8.125) {};
		\node [style=none] (44) at (9.5, -4.75) {};
		\node [style=none] (45) at (4.25, -3.25) {};
		\node [style=none] (46) at (5.75, -0) {};
		\node [style=none] (47) at (8.875, -0) {};
		\node [style=none] (48) at (4.25, 3.25) {};
		\node [style=none] (49) at (9.5, 5.75) {};
		\node [style=none] (50) at (-9.5, 5.75) {};
		\node [style=none] (51) at (-4.25, 3.25) {};
		\node [style=none] (52) at (-8.875, -0) {};
		\node [style=none] (53) at (-5.5, -0) {};
		\node [style=none] (54) at (-9.5, -4.75) {};
		\node [style=none] (55) at (-4.25, -3.25) {};
		\node [style=none] (56) at (-3, 1.5) {\tiny{$d_{U,V}\otimes\text{id}_{F\left(W\right)}$}};
		\node [style=none] (57) at (0, 3.25) {\tiny{$a_{G\left(V\right),G\left(U\right),W}$}};
		\node [style=none] (58) at (2.75, 1.5) {\tiny{$\text{id}_{FG\left(V\right)}\otimes d_{G\left(U\right),W}$}};
		\node [style=none] (59) at (2.125, -1.5) {\tiny{$\text{id}_{FG\left(V\right)}\otimes\left(\text{id}_{G\left(W\right)}\otimes\Phi_{G\left(U\right)}\right)$}};
		\node [style=none] (60) at (-3.5, -1.5) {\tiny{$a_{U,V,W}$}};
		\node [style=none] (61) at (-1.375, -3.75) {\tiny{$d_{F\left(U\right),V\otimes W}$}};
		\node [style=none] (62) at (1, -3.75) {\tiny{$a_{G\left(V\right),G\left(W\right),G\left(U\right)}$}};
		\node [style=none] (63) at (-7, 0.25) {\tiny{$\text{id}_{U\otimes V}\otimes\Theta_{W}$}};
		\node [style=none] (64) at (-7, 2.25) {$\textbf{1}\#$};
		\node [style=none] (65) at (-5.375, 4.5){\tiny{$ ~\ \left(\Phi_{V}\otimes\Phi_{U}\right)\otimes\Theta_{W} $}};
		\node [style=none] (66) at (4.625, 4.5) {\tiny{$\left(F\left(\Phi_{V}\right)\circ\Theta_{V}\right)\otimes\left(\Phi_{U}\otimes\text{id}_{W}\right)$}};
		\node [style=none] (67) at (0, 4.625) {$\textbf{2}\#$};
		\node [style=none] (68) at (0, 6.25) {\tiny{$b_{V,U,W}$}};
		\node [style=none] (69) at (-9.375, -2.25) {\tiny{$b_{U,V,W}$}};
		\node [style=none] (70) at (-9.125, 2.75) {\tiny{$c_{U,V}\otimes\text{id}_{W}$}};
		\node [style=none] (71) at (9, -2.25) {\tiny{$\text{id}_{V}\otimes\text{id}_{W\otimes U}$}};
		\node [style=none] (72) at (9.125, 2.75) {\tiny{$\text{id}_{V}\otimes c_{U,W}$}};
		\node [style=none] (73) at (-7, -2.25) {$\textbf{4}\#$};
		\node [style=none] (74) at (-6.375, -4.25) {\tiny{$\Theta_{U}\otimes\text{id}_{V\otimes W}$}};
		\node [style=none] (75) at (5.5, -4.25) {\tiny{$\left(G(\Theta_{V})\circ\Phi_{V}\right)\otimes\left(\Phi_{W}\otimes\Phi_{U}\right)  $}};
		\node [style=none] (76) at (6.999999, 0.4) {\tiny{$ ~\ \left(F\left(\Phi_{V}\right)\circ\Theta_{V}\right)\otimes\left(\Phi_{W}\otimes\left(G\left(\Phi_{U}\right)\circ\Phi_{U}\right)\right)$}};
		\node [style=none] (77) at (1.75, -6.5) {\tiny{$\left(\Phi_{V}\otimes\Phi_{W}\right)\otimes\left(G(\Theta_{U})\circ\Phi_{U}\right) $}};
		\node [style=none] (78) at (-4.875, -6.5) {\tiny{$c_{U,V\otimes W}$}};
		\node [style=none] (79) at (5, -6.5) {\tiny{$b_{V,W,U}$}};
		\node [style=none] (80) at (-3.5, -5.5) {$\textbf{5}\#$};
		\node [style=none] (81) at (3.5, -5.5) {$\textbf{6}\#$};
		\node [style=none] (82) at (6.999999, 2.25) {$\textbf{3}\#$};
		\node [style=none] (83) at (6.999999, -2.25) {$\textbf{7}\#$};
		\node [style=none] (84) at (0, -4.75) {$\left(G\left(V\right)\otimes G\left(W\right)\right)\otimes FG\left(U\right)$};
		\node [style=none] (85) at (0, -5.125) {\tiny{$\mid\mid$}};
		\node [style=none] (86) at (0, -5.5) {$G\left(V\otimes W\right)\otimes GF\left(U\right)$};
	\end{pgfonlayer}
	\begin{pgfonlayer}{edgelayer}
		\draw [->,thick] (1.center) to (3.center);
		\draw [->,thick] (5.center) to (6.center);
		\draw [->,thick] (8.center) to (9.center);
		\draw [->,thick] (2.center) to (15.center);
		\draw [->,thick] (17.center) to (13.center);
		\draw [->,thick] (14.center) to (20.center);
		\draw [->,thick] (18.center) to (11.center);
		\draw [->,dashed](21.center) to (30.center);
		\draw [->,dashed] (32.center) to (33.center);
		\draw [->,dashed] (35.center) to (36.center);
		\draw [->,dashed] (23.center) to (29.center);
		\draw [->,dashed] (27.center) to (25.center);
		\draw [->,dashed] (26.center) to (41.center);
		\draw [->,dashed] (39.center) to (37.center);
		\draw [->] (52.center) to (53.center);
		\draw [->](54.center) to (55.center);
		\draw [->] (43.center) to (12.center);
		\draw [->] (44.center) to (45.center);
		\draw [->] (47.center) to (46.center);
		\draw [->] (49.center) to (48.center);
		\draw [->] (50.center) to (51.center);
	\end{pgfonlayer}
\end{tikzpicture}
\end{center}
\caption{The (H1) property implies the first Hexagon axiom}
\label{fig9999}
\end{sidewaysfigure}
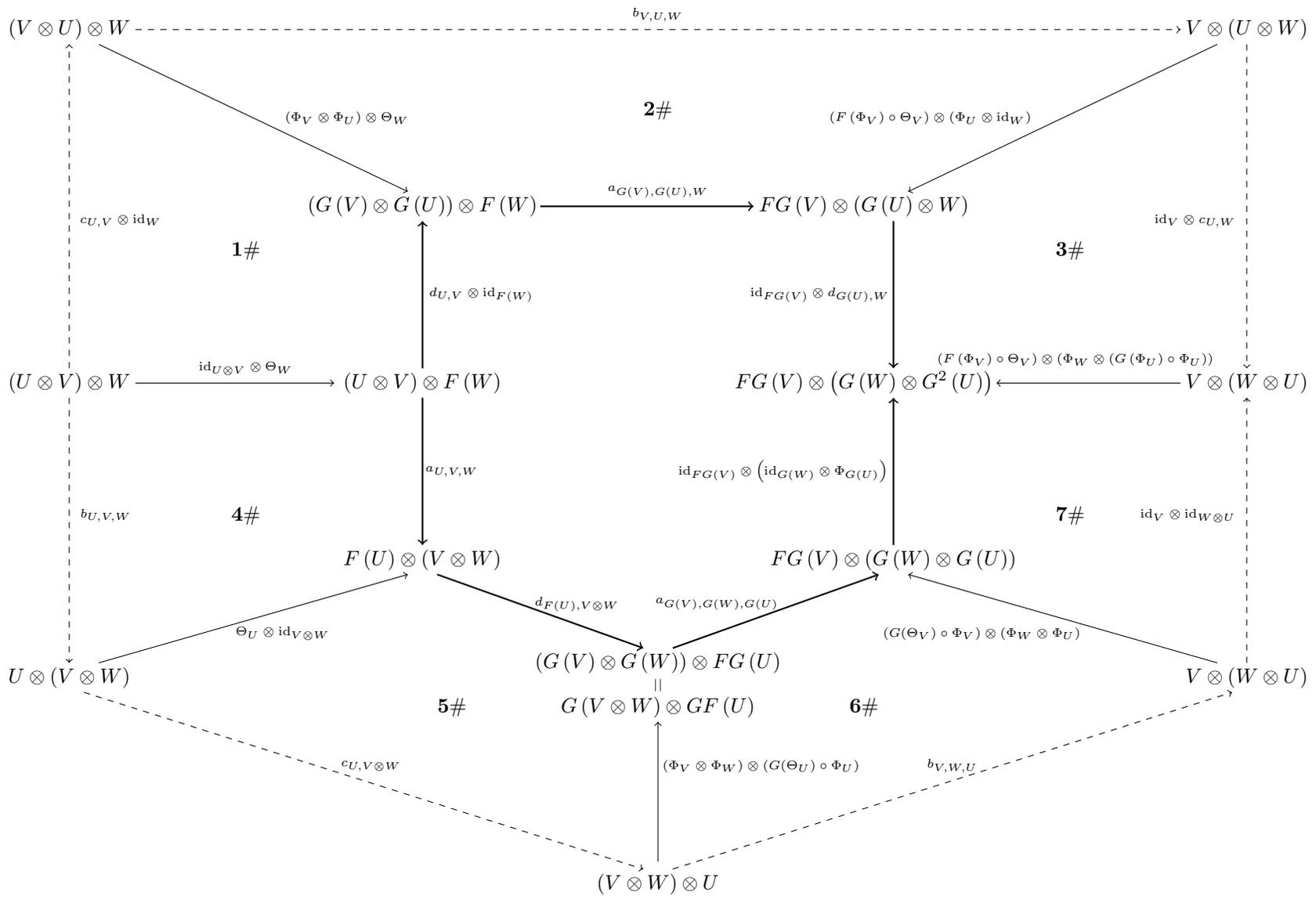

We begin with the square portion of Figure \ref{fig9999} labeled $\textbf{1}\#$. Computing we get that
\begin{eqnarray*}
\left(\left(\Phi_{V}\otimes\Phi_{U}\right)\otimes\Theta_{W}\right)\circ\left(c_{U,V}\otimes\text{id}_{W}\right) &=& \left(\left(\Phi_{V}\otimes\Phi_{U}\right)\circ c_{U,V}\right)\otimes\left(\Theta_{W}\circ \text{id}_{U}\right)\\
&=&d_{U,V}\otimes \Theta _{W}=
\left(d_{U,V}\otimes\text{id}_{F\left(W\right)}\right)\circ\left(\text{id}_{U\otimes V}\otimes\Theta_{W}\right).
\end{eqnarray*}
The first and third equality follow from the functoriality of the tensor product. The second equality is the definition of $c_{U,V}$.
\begin{center}
\begin{figure}[!ht]
\begin{center}
\begin{tikzpicture}
  \matrix (m) [matrix of math nodes,row sep=4em,column sep=3em,minimum width=0.5em] 
     { \left(V\otimes U\right)\otimes W & & V\otimes \left(U\otimes W\right) \\ 
		    \left(V\otimes U\right)\otimes F\left(W\right)  & & F\left(V\right)\otimes \left(U\otimes W\right) \\
			   \left(G\left(V\right)\otimes G\left(U\right)\right)\otimes F\left(W\right) & & FG\left(V\right)\otimes \left(G\left(U\right)\otimes W\right) \\};
  \path[-stealth]
	  (m-1-1) edge node [above] {\tiny{$ b_{V,U,W} $}} (m-1-3)
		(m-1-3) edge node [right] {\tiny{$\Theta_{V}\otimes\text{id}_{U\otimes W}$}} (m-2-3)
		(m-1-1) edge node [left] {\tiny{$\text{id}_{V\otimes U}\otimes\Theta_{W}$}} node [right] {$\hskip2.8cm \left(\ast\right) $} (m-2-1)
    (m-2-3) edge node [right] {\tiny{$F\left(\Phi_{V}\right)\otimes\left(\Phi_{U}\otimes\text{id}_{W}\right)$}} (m-3-3)
		(m-2-1) edge node [left] {\tiny{$\left(\Phi_{V}\otimes\Phi_{U}\right)\otimes\text{id}_{F\left(W\right)}$}} node [right] {$\hskip2.7cm \left(\ast\ast\right) $} (m-3-1)
		        edge node [above] {\tiny{$ a_{V,U,W}$}} (m-2-3)
		(m-3-1) edge node [above] {\tiny{$ a_{G\left(V\right),G\left(U\right),W}$}} (m-3-3);
\end{tikzpicture} 
\end{center}
\caption{The diagram for  $\textbf{2}\#$}
\label{fig9999f}
\end{figure}
\end{center}

Next we consider the square portion of Figure \ref{fig9999} labeled $\textbf{2}\#$, which coincides with the diagram 
in Figure \ref{fig9999f} by the functoriality of the tensor product. The portion of the diagram in Figure \ref{fig9999f} labeled $\left(\ast\right)$ commutes by the definition of $b$, and the portion labeled $\left(\ast\ast\right)$ commutes by the naturality of $a_{V,U,W}$. 

Now we consider the square portion of Figure \ref{fig9999} labeled $\textbf{3}\#$, which coincides with the diagram  in Figure \ref{fig9999g} by the functoriality of the tensor product. The portion of the diagram in Figure \ref{fig9999g} 
labeled $\left(\bullet\right)$ commutes by the definition of $c_{V,W}$ and the functoriality of the tensor product, 
and the portion of the diagram labeled $\left(\bullet\bullet\right)$ commutes by the naturality of $d_{U,W}$ and the 
functoriality of the tensor product.

\begin{center}
\begin{figure}[!ht]
\begin{center}
\begin{tikzpicture}
  \matrix (m) [matrix of math nodes,row sep=4em,column sep=3em,minimum width=0.5em] 
     { V\otimes \left(U\otimes W\right) & & V\otimes \left(W\otimes U\right) \\ 
		    F\left(V\right)\otimes \left(U\otimes W\right)  & & F\left(V\right)\otimes \left(G\left(W\right)\otimes G\left(U\right)\right) \\
			   FG\left(V\right)\otimes \left(G\left(U\right)\otimes W\right) & & FG\left(V\right)\otimes \left(G\left(W\right)\otimes G^{2}\left(U\right)\right) \\};
  \path[-stealth]
	  (m-1-1) edge node [above] {\tiny{$ \text{id}_{V}\otimes c_{U,W} $}} (m-1-3)
		(m-1-3) edge node [right] {\tiny{$\Theta_{V}\otimes\left(\Phi_{W}\otimes\Phi_{U}\right)$}} (m-2-3)
		(m-1-1) edge node [left] {\tiny{$\Theta_{V}\otimes\text{id}_{U\otimes W}$}} node [right] {$\hskip2.8cm \left(\bullet\right) $} (m-2-1)
    (m-2-3) edge node [right] {\tiny{$F\left(\Phi_{V}\right)\otimes\left(\text{id}_{G\left(W\right)}\otimes G\left(\Phi_{U}\right)\right)$}} (m-3-3)
		(m-2-1) edge node [left] {\tiny{$F\left(\Phi_{V}\right)\otimes\left(\Phi_{U}\otimes\text{id}_{W}\right)$}} node [right] {$\hskip2.7cm \left(\bullet\bullet\right) $} (m-3-1)
		        edge node [above] {\tiny{$ \text{id}_{F\left(V\right)}\otimes d_{U,W}$}} (m-2-3)
		(m-3-1) edge node [above] {\tiny{$ \text{id}_{FG\left(V\right)}\otimes d_{G\left(U\right),W}$}} (m-3-3);
\end{tikzpicture} 
\end{center}
\caption{The diagram for  $\textbf{3}\#$}
\label{fig9999g}
\end{figure}
\end{center}

The square portion labeled $\textbf{4}\#$ of Figure \ref{fig9999} commutes by the definition of $b_{U,V,W}$. 

Next we consider the square portion of Figure \ref{fig9999} labeled $\textbf{5}\#$. In Figure \ref{fig9999j}, the portion of the diagram labeled $\left(\diamond\right)$ commutes by the definition of $c_{U,V\otimes W}$ and the portion of the diagram labeled $\left(\diamond\diamond\right)$ commutes by the naturality of $d_{U,V\otimes W}$. Now we notice that the diagram in Figure 
\ref{fig9999j} coincides with the square portion of Figure \ref{fig9999} labeled $\textbf{5}\#$ by using the functoriality 
of the tensor product and the axiom $\Phi _{V\otimes W}=\Phi _V\otimes \Phi _W$ from the definition of a hom-tensor category. 

\begin{center}
\begin{figure}[!ht]
\begin{center}
\begin{tikzpicture}
  \matrix (m) [matrix of math nodes,row sep=4em,column sep=3em,minimum width=0.5em] 
     { U\otimes \left(V\otimes W\right) & & \left(V\otimes W\right)\otimes U \\ 
		    U\otimes \left(V\otimes W\right)  & & G\left(V\otimes W\right)\otimes G\left(U\right)  \\
			   F\left(U\right)\otimes \left(V\otimes W\right) & & G\left(V\otimes W\right)\otimes GF\left(U\right)  \\};
  \path[-stealth]
	  (m-1-1) edge node [above] {\tiny{$ c_{U,V\otimes W} $}} (m-1-3)
		(m-1-3) edge node [right] {\tiny{$\Phi_{V\otimes W}\otimes\Phi_{U}$}} (m-2-3)
		(m-1-1) edge node [left] {\tiny{$\text{id}_{U}\otimes\text{id}_{V\otimes W}$}} node [right] {$\hskip2.3cm \left(\diamond\right) $} (m-2-1)
    (m-2-3) edge node [right] {\tiny{$\text{id}_{G\left(V\otimes W\right)}\otimes G(\Theta_{U})$}} (m-3-3)
		(m-2-1) edge node [left] {\tiny{$\Theta_{U}\otimes\text{id}_{V\otimes W}$}} node [right] {$\hskip2.2cm \left(\diamond\diamond\right) $} (m-3-1)
		        edge node [above] {\tiny{$ d_{U,V\otimes W}$}} (m-2-3)
		(m-3-1) edge node [above] {\tiny{$ d_{F\left(U\right),V\otimes W}$}} (m-3-3);
\end{tikzpicture} 
\end{center}
\caption{The  diagram for $\textbf{5}\#$}
\label{fig9999j}
\end{figure}
\end{center}

Next we consider the square portion of Figure \ref{fig9999} labeled $\textbf{6}\#$.  In Figure \ref{fig9999m} the portion of the diagram labeled $\left(\natural\right)$ commutes by naturality of $b_{V,W,U}$ and the portion of the diagram labeled $\left(\natural\natural\right)$ commutes by the definition of $b_{G(V),G(W),G(U)}$. Now we notice that the diagram in Figure 
\ref{fig9999m} coincides with the square portion of Figure \ref{fig9999} labeled $\textbf{6}\#$ by using the functoriality 
of the tensor product and the fact that $G(\Theta _V)=\Theta _{G(V)}$, which also allows us to glue $\textbf{6}\#$ and $\textbf{7}\#$ together.

\begin{center}
\begin{figure}[!ht]
\begin{center}
\begin{tikzpicture}
  \matrix (m) [matrix of math nodes,row sep=4em,column sep=3em,minimum width=0.5em] 
     { \left(V\otimes W\right)\otimes U & & V\otimes \left(W\otimes U\right) \\ 
		    \left(G(V)\otimes G(W)\right)\otimes G\left(U\right)  & & G\left(V\right)\otimes \left(G(W)\otimes G(U)\right)  \\
			  \left(G\left(V\right)\otimes G\left(W\right)\right)\otimes FG\left(U\right) & & FG\left(V\right)\otimes \left(G\left(W\right)\otimes G\left(U\right)\right)  \\};
  \path[-stealth]
	  (m-1-1) edge node [above] {\tiny{$ b_{V,W,U} $}} (m-1-3)
		(m-1-3) edge node [right] {\tiny{$\Phi_{V}\otimes(\Phi_{W}\otimes \Phi_{U})$}} (m-2-3)
		(m-1-1) edge node [left] {\tiny{$(\Phi_V\otimes \Phi_W)\otimes\Phi_{U}$}} node [right] {$\hskip2.9cm \left(\natural\right) $} (m-2-1)
    (m-2-3) edge node [right] {\tiny{$\Theta_{G\left(V\right)}\otimes \left(\text{id}_{G(W)}\otimes\text{id}_{G(U)}\right)$}} (m-3-3)
		(m-2-1) edge node [left] {\tiny{$\left(\text{id}_{G(V)}\otimes \text{id}_{G(W)}\right)\otimes\Theta_{G\left(U\right)}$}} node [right] {$\hskip2.8cm \left(\natural\natural\right) $} (m-3-1)
		        edge node [above] {\tiny{$ b_{G(V),G(W),G(U)}$}} (m-2-3)
		(m-3-1) edge node [above] {\tiny{$ a_{G\left(V\right),G\left(W\right),G\left(U\right)}$}} (m-3-3);
\end{tikzpicture} 
\end{center}
\caption{The diagram for $\textbf{6}\#$}
\label{fig9999m}
\end{figure}
\end{center}

Finally we consider the square portion of Figure \ref{fig9999} labeled $\textbf{7}\#$. This is commutative by using the functoriality of $\otimes$, the fact that $\Phi _{G(U)}=G(\Phi _U)$ 
and because $\Theta_{G(V)}\circ\Phi_V=\Phi_{F(V)}\circ\Theta_V$.  The last statement is a a consequence of the fact that $\Theta$ is a natural transformation and $F(\Phi_V)=\Phi_{F(V)}$ (see Figure \ref{figlast}). 
\begin{center}
\begin{figure}[!ht]
\begin{center}
\begin{tikzpicture}
  \matrix (m) [matrix of math nodes,row sep=4em,column sep=3em,minimum width=0.5em] 
     { V& & G(V) \\ 
			  F(V) & & F(G(V))  \\};
  \path[-stealth]
	  (m-1-1) edge node [above] {\tiny{$ \Phi_V$}} (m-1-3)
		(m-1-3) edge node [right] {\tiny{$\Theta_{G(V)}$}}(m-2-3)
		(m-1-1) edge node [left] {\tiny{$\Theta_V$}} (m-2-1)
		(m-2-1) edge node [above] {\tiny{$F(\Phi_V)$}} (m-2-3);
\end{tikzpicture} 
\end{center}
\caption{$\Theta_{G(V)}\circ\Phi_V=\Phi_{F(V)}\circ\Theta_V$}
\label{figlast}
\end{figure}
\end{center}

 Thus the  (H1) property for $d_{U,V}$ implies the first Hexagon axiom property for $c_{U,V}$. Using similar arguments one can show that  the (H2) property for $d_{U,V}$ implies  the second Hexagon axiom property for $c_{U,V}$. Therefore, $c_{U,V}$ is a quasi-braiding for the quasi-braided category $\left(\mathcal{C},\otimes,b,c\right)$.
\end{proof}
\begin{Rem}
Let again $H=(H, m_H, \Delta _H, \alpha _H, \psi _H)$ be a hom-bialgebra for which 
$\alpha _H=\psi _H$ and $\alpha _H$ is bijective. We can give a new proof of Proposition \ref{prop7.3} based on the results 
obtained in this section. We proceed as follows. We know from Proposition  \ref{prop7.5} that $_H^H\mathbb{YD}$ 
is a hom-braided category, and it is obvious that its subcategory $_H^H{\mathcal YD}$ inherits this hom-braided structure. 
We apply Propositions \ref{prop3333d} and \ref{prop9999b} to this hom-braided category $_H^H{\mathcal YD}$, 
in this case the natural isomorphism $\Theta :\text{id}_{\mathcal{C}}\rightarrow F$ being the same as the natural isomorphism 
$\Phi :\text{id}_{\mathcal{C}}\rightarrow G$ (we recall that we have $F=G$ in this situation). Thus, we obtain that 
$_H^H{\mathcal YD}$ becomes a quasi-braided category, and it is very easy to see that its quasi-braided structure 
is exactly the one that appears in Proposition  \ref{prop7.3}. 
\end{Rem}


\end{document}